\documentclass[11pt]{article}
\usepackage[margin=1in]{geometry}
\usepackage{amsmath,amssymb,amsthm,mathtools,mathrsfs,booktabs,array}
\usepackage{microtype,needspace,enumitem,fancyvrb}
\usepackage[hidelinks]{hyperref}
\hypersetup{pdftitle={Prediction with Five Experts and Geometric Stopping: A Probabilistic Construction and Analytic Verification},
 pdfauthor={Erhan Bayraktar, Ibrahim Ekren, Nikolaos Kolliopoulos},
 pdfsubject={Probabilistic construction, analytic verification, and smoothness},
 pdfkeywords={Prediction with expert advice, regret minimization, geometric stopping,
 Hamilton-Jacobi-Bellman equations, hyperbolic minimum principles, reflected diffusions}}
\allowdisplaybreaks[2]
\setlist{nosep}
\newtheorem{theorem}{Theorem}[section]
\newtheorem{proposition}[theorem]{Proposition}
\newtheorem{lemma}[theorem]{Lemma}
\newtheorem{corollary}[theorem]{Corollary}
\theoremstyle{definition}

\theoremstyle{remark}
\newtheorem{remark}[theorem]{Remark}
\numberwithin{equation}{section}
\newcommand{\dd}{\,\mathrm d}
\newcommand{\R}{\mathbb R}
\newcommand{\cD}{\mathcal D}
\newcommand{\Kop}{\mathcal K}
\newcommand{\Top}{\mathcal T}
\newcommand{\Aop}{\mathcal A}
\newcommand{\csch}{\operatorname{csch}}
\newcommand{\sech}{\operatorname{sech}}

\title{Prediction with Five Experts and Geometric Stopping:\\
A Probabilistic Construction and Analytic Verification}
\author{Erhan Bayraktar\thanks{Department of Mathematics,
University of Michigan. Email: \texttt{erhan@umich.edu}.}
\and Ibrahim Ekren\thanks{Department of Mathematics,
University of Michigan. Email: \texttt{iekren@umich.edu}.}
\and Nikolaos Kolliopoulos\thanks{Department of Mathematics and Statistics,
University of Cyprus. Email: \texttt{kolliopoulos.nikolaos@ucy.ac.cy}.}}
\date{October 28, 2026}
\begin{document}
\maketitle
\begin{abstract}
We present a solution of the limiting five-expert prediction problem
with geometric stopping, based on stochastic calculus and partial
differential equations. If $\delta$ is the stopping probability,
the minimax expected regret from tied initial scores is
$45\pi^2/(512\sqrt{2\delta})+o(\delta^{-1/2})$ as $\delta\downarrow0$.
The adversarial rank control selecting the best and third-best experts
maximizes the limiting Hamiltonian at every state. Following the
four-expert construction of Bayraktar, Ekren, and Zhang~\cite{BEZ},
we represent the value correction as a discounted boundary-local-time expectation for a
degenerate obliquely reflected Brownian motion. The hyperbolic systems
governing its boundary traces are derived and solved explicitly.
To verify the nonlinear
Hamilton--Jacobi--Bellman equation, we combine heat-equation minimum
principles in equality directions, ODE minimum principles applied to
convex combinations of multiple controls, and comparisons across controls.
These arguments reduce the 48 regional control inequalities to a few
lower-dimensional boundary problems, with the remaining signs
established by explicit monotonicity arguments and by factorization of
multivariate integral kernels that are positive polynomials in some of
their variables. We prove global
$C^2$ regularity across both regional interfaces and changes in rank
ordering. We also show that the alternating COMB control only attains the
limiting Hamiltonian maximum exactly when the two highest scores
coincide and the third- and fourth-highest scores coincide.
\end{abstract}
\begin{quote}

\end{quote}
\tableofcontents

\section{Introduction and the prediction game}\label{paper:sec:setup}

Prediction with expert advice asks a player to compete with the best
of several experts without a probabilistic model for their gains.
Classical weighted-majority and expert-aggregation algorithms provide
regret guarantees in this setting~\cite{LW1994,CBFH1997}; a systematic
account is given by Cesa-Bianchi and Lugosi~\cite{CL2006}.

For a fixed small number of experts, exact minimax analysis gives more
detailed information about the value and the optimal strategies.
Cover's finite-horizon analysis of binary prediction~\cite{Cover1967}
provides the classical two-expert benchmark.
Gravin, Peres, and Sivan~\cite{GPS2016} solved the three-expert problem
with geometric stopping, identified a probability-matching structure
for two and three experts, and obtained improved four-expert bounds.
Their subsequent work~\cite{GPS2017} gives lower bounds, including
sharp results, for multiplicative-weights families under fixed and geometric horizons,
a question distinct from unrestricted minimax optimality.

For geometric stopping, the long-time scaling is described by a
stationary nonlinear PDE.  Drenska and Kohn~\cite{DK} established this
PDE viewpoint.  The four-expert geometric problem was solved asymptotically by
Bayraktar, Ekren, and Yili Zhang~\cite{BEZ}; their Proposition~5.4 is a
prototype of the hyperbolic minimum principles used below.
For a fixed finite horizon, Bayraktar, Ekren, and Xin
Zhang~\cite{BEZX2020} explicitly solved the four-expert limiting PDE
and established its $C^2$ regularity and the asymptotic optimality of
the associated strategies. Related PDE methods have been developed
for restricted adversarial corruption~\cite{BEZX2021} and for regret
bounds under partial monitoring~\cite{BEZX2023}, which modify the
adversary or the information structure of the prediction game.

For five experts, Chase~\cite{Chase2019} gave experimental evidence
against asymptotic comb optimality by comparing finite-horizon regrets
with those of an alternative rank-based strategy.
Numerical work on the geometric-stopping PDE by Calder, Drenska, and
Mosaphir~\cite{CDM} provided evidence for the failure of global comb
optimality with five through ten experts and supported global optimality
of the same alternative five-expert control. In our increasing
rank convention, this control is $00101$, not the alternating comb
control $01010$. The construction and verification below establish its
optimality in the limiting Hamiltonian.

The present work gives a probabilistic construction and an analytic
verification of the nonlinear equation.
Our construction uses a reflected-process representation, two linear
propagation operators, and a one-dimensional diagonal profile. The
verification emphasizes equality-direction reductions, positive
companion-kernel comparison, and cancellation through averaging
controls. The terminal sign arguments use explicit positive
rescalings and finite differentiation. The
regularity argument treats the internal interfaces and the ordering
walls separately; in particular, matching first derivatives at a wall
is used only after regularity up to the closed chamber is established.

Calder and Drenska~\cite{CD2026} independently solve the same
five-expert prediction PDE.\footnote{While preparing this manuscript
for circulation, we became aware of the preprint~\cite{CD2026}, first
submitted to arXiv on 14 September 2026. Its appearance prompted us
to make our independently developed approach available promptly.}
A detailed comparison of the constructions, smoothness arguments,
and nonlinear verifications is given in
Appendix~\ref{paper:sec:CDcomparison}.

\subsection{Discrete game and scaling}
There are five experts.  At each stage the player chooses a probability
vector $\alpha\in\Delta_5$ and the adversary chooses a distribution
$\beta$ on the subsets $J\subseteq\{1,\ldots,5\}$.  The expert gain
vector is $e_J=\sum_{j\in J}e_j$.  The player samples an expert from
$\alpha$ independently of the adversary's current randomization.
Both players observe the past.  The regret state is the vector of
expert cumulative gains minus the player's cumulative gain.

Independently of play, the game stops before the next move with
probability $\delta\in(0,1)$.  Thus the number $T$ of moves has
$\mathbb P(T=t)=\delta(1-\delta)^t$, $t\geq0$.  With terminal payoff
$g(x)=\max_i x_i$, the minimax value $V_\delta$ satisfies
\begin{equation}\label{paper:eq:DPP}
 V_\delta(x)=\delta g(x)+(1-\delta)
 \min_{\alpha\in\Delta_5}\max_{\beta\in\Delta_{32}}
 \sum_J\beta_J\bigl[V_\delta(x+e_J)-\alpha(J)\bigr],
 \qquad \alpha(J)=\sum_{j\in J}\alpha_j.
\end{equation}
Here translation equivariance,
$V_\delta(x+c\mathbf1)=V_\delta(x)+c$, removes the player's realized
gain from the argument of $V_\delta$.

Set
\begin{equation}\label{paper:eq:scaling}
 u_\delta(x)=\sqrt\delta\,V_\delta(x/\sqrt\delta).
\end{equation}
The continuum-limit theorem of~\cite{DK} identifies the locally uniform
limit with the unique viscosity solution of at most linear growth of
\begin{equation}\label{paper:eq:HJB}
 {\quad u(x)-\frac12\max_{\omega\in\{0,1\}^5}
       \omega^{\mathsf T}D^2u(x)\omega=g(x),\qquad x\in\R^5.\quad}
\end{equation}
The factor $1/2$ corresponds to Brownian variance one and discount rate
one.  Formally it comes from the second-order term in
$u_\delta(x+\sqrt\delta e_J)$.  At first order the player's choice
$\alpha=\nabla u$ cancels the adversary's linear contribution; the
limiting value is monotone and translation equivariant, so this
gradient is a probability vector wherever it exists.

\subsection{Ranks, gaps, and the verification problem}
Write $x^{(1)}\leq\cdots\leq x^{(5)}$ for the variables
$x_1,\ldots,x_5$ sorted in ascending order, and define the gap variables
\[
 q_i=x^{(i+1)}-x^{(i)}\geq 0,\qquad i=1,\ldots,4.
\]
A symmetric translation-equivariant candidate has the form
\begin{equation}\label{paper:eq:rankvalue}
 u(x)=x^{(5)}+v(q).
\end{equation}
For a binary rank control $\omega = (\omega_1,\omega_2,\omega_3,\omega_4,\omega_5) \in \{0,1\}^5$ (meaning a control for the ranked process $(X^{(1)}_t,\ldots,X^{(5)}_t)$), let
\begin{equation}\label{paper:eq:residual}
 d(\omega)=(\omega_2-\omega_1,\ldots,\omega_5-\omega_4),\qquad
 \mathcal R_\omega[v]=v-\frac12(d(\omega)\cdot\nabla_q)^2v.
\end{equation}
In the strict chamber, $x^{(1)}<\cdots< x^{(5)}$, where no two of the variables $x_i$ are equal,
the affine term $g(x)=x^{(5)}$ has zero Hessian;
therefore \eqref{paper:eq:HJB} is precisely
\begin{align}
\min_\omega\mathcal R_\omega[v]=0, \label{paper:eq:gapHJBb}
\end{align}
for which we must have
\begin{align}
\label{paper:eq:gapHJBa}
\mathcal R_\omega[v]\geq0\quad\hbox{for every }\omega \in \{0,1\}^5.
\end{align}
Since $d(\mathbf1-\omega)=-d(\omega)$ for $\mathbf1 = (1,1,1,1,1)$,
it suffices to check \eqref{paper:eq:gapHJBa} for one representative of
each of the sixteen complementary pairs $\{\omega,\mathbf1-\omega\}$.
The solution we derive is given by a different formula in each of three
regions that we call A, B, and C. For Regions A and B, we use controls
with at most two ones, while the Region C proof uses controls whose
first coordinate is zero. To obtain \eqref{paper:eq:gapHJBb}, however,
the inequality \eqref{paper:eq:gapHJBa} must hold as an equality for
some choice of $\omega$. We will see that equality is indeed attained for
\begin{equation}\label{paper:eq:star}
 \omega = \omega_* :=(0,0,1,0,1),\qquad
 d(\omega_*)=(0,1,-1,1),
\end{equation}
since the following linear equation is satisfied in all three regions
A, B, and C:
\begin{equation}\label{paper:eq:linear}
 v-\frac12(\partial_{q_2}-\partial_{q_3}+\partial_{q_4})^2v=0.
\end{equation}
Proving that $\omega_*$ is an optimal control also requires all the
inequalities in \eqref{paper:eq:gapHJBa}, along with enough regularity to interpret
them across the regional and ordering interfaces.


\subsection{General approach and main challenges}
\label{paper:sec:valueformula}

Equation \eqref{paper:eq:gapHJBb} belongs to a class of nonlinear PDEs
for which there is no standard method for finding explicit solutions to
boundary value problems. In this case, the solution $v$, which must also
have at most linear growth as $q\to\infty$, is given by a complicated
expression. To determine $v$, we construct a natural candidate by
recalling a known stochastic representation and heuristically rewriting
it in terms of a four-dimensional reflected Brownian motion. We then
verify rigorously that this candidate satisfies \eqref{paper:eq:gapHJBb},
together with the necessary regularity and growth conditions. While
this technique was used in \cite{BEZ} to solve the four-expert counterpart
of our problem, two features make the five-expert problem substantially
more difficult:
\begin{enumerate}
    \item The heuristic derivation of the natural candidate for $v$
    requires rigorously solving linear first-order hyperbolic systems
    with nonconstant coefficients. In the four-expert setting, the
    initial conditions of those systems were simple enough to allow one
    to guess their solutions. With five experts, the corresponding
    initial conditions involve long, complicated expressions, requiring
    us to develop analytic methods for finding the general solutions.

    \item In the four-expert setting, the most difficult part of
    verifying the counterpart of \eqref{paper:eq:gapHJBa} was the
    application of hyperbolic minimum principles arising naturally from
    the hyperbolic systems solved during the derivation of the candidate
    $v$. With five experts, the inequalities arising from
    \eqref{paper:eq:gapHJBa} are too long to fit on a single page, and
    proving them requires identifying hidden patterns. In particular,
    crucial hyperbolic minimum principles can be applied to convex
    combinations of \eqref{paper:eq:gapHJBa} over multiple choices of
    $\omega$. Wolfram Mathematica was used during the derivation to check
    lengthy algebraic and differentiation identities. Optional scripts
    in the accompanying repository provide reproducible checks of
    selected identities. The proofs presented in the paper are
    self-contained and can be verified without running these scripts.
\end{enumerate}

\subsection{The three regions, the auxiliary functions and the value formula}\label{subsec14}
The gap orthant $\mathbb{R}_+^4$, on which $v$ is defined, is partitioned
into Regions A, B, and C, described by
\begin{align}\label{pc:eq:regions}
 A&=\{q_i\geq0:\ q_4\geq q_2,\ 2q_1\geq q_4-q_2\},\notag\\
 B&=\{q_i\geq0:\ q_4\leq q_2\},\\
 C&=\{q_i\geq0:\ q_4\geq q_2,\ 2q_1\leq q_4-q_2\}.
 \notag
\end{align}
The formulas for $v$ in neighbouring regions agree on their shared
interfaces, so the use of closed regions causes no ambiguity. The values
of $v$ on lower-dimensional boundaries determine $v$ throughout Regions
A, B, and C through first-order hyperbolic systems with nonconstant
coefficients. The construction proceeds in three stages, beginning on a
one-dimensional set and then determining the values on higher-dimensional
sets. We first introduce the auxiliary functions from the first two
stages, followed by the propagation operators.
\begin{enumerate}
    \item First stage: the value on $q_1=q_2=q_4=0$.
We define
\begin{align}\label{paper:eq:Uintegral}
 H(t)&=\frac{3}{\sqrt2\sinh^5t\cosh^2t}
 \left(\frac t{16}-\frac{\sinh(2t)}{64}
       -\frac{\sinh(4t)}{64}+\frac{\sinh(6t)}{192}\right),\notag\\
 U(t)&=\cosh t\int_t^\infty H(r)\,dr
\end{align}
for $t>0$. Both functions extend continuously to $t=0$. Indeed, the
bracket in the definition of $H(t)$ equals $t^5/5+O(t^7)$, cancelling
the fifth-order zero of the denominator. The exponential decay of $H$
at infinity also ensures convergence of the integral defining $U$.
As shown below, $U(t)=v(0,0,t/\sqrt2,0)$. The limit
$U(0)=45\pi^2/(512\sqrt2)$, obtained by taking $t\downarrow0$ in
\eqref{paper:eq:Uintegral}, gives the leading constant in the minimax
expected regret from tied initial scores as the stopping probability
$\delta$ tends to zero.

\item Second stage: the values on $q_2=q_4=0$ and $q_1=q_2=0$.
For $s>0$ and $w\geq0$, define
\begin{align}\label{paper:eq:L}
 L(s,w)&=\int_1^{\coth s}
       (r\sinh s-\cosh s)e^{-2wr}g(r)\,dr,\notag\\
 \widetilde L(s,w)&=\frac1{\sqrt2}\arctan(e^{-s})\cosh s+L(s,w),
\end{align}
where the negative function $g$ is given by
\begin{equation}\label{paper:eq:g}
 g(r)=\frac{8-25r^2+15r^4
       -15r(r^2-1)^2\operatorname{arccoth}r}
                 {16\sqrt2\,r\sqrt{r^2-1}},
 \qquad
 \operatorname{arccoth}r=\frac12\log\frac{r+1}{r-1}
\end{equation}
for $r>1$. As shown below,
$\widetilde L(s,w)=v(w/\sqrt2,0,s/\sqrt2,0)$ provides the boundary
data for Regions A and B. Its first term,
$\frac1{\sqrt2}\arctan(e^{-s})\cosh s$, is the corresponding
four-expert boundary function~\cite[proof of Proposition~5.1]{BEZ};
$L(s,w)$ is the perturbation term.
For $y>0$ and $z\geq0$, define
\begin{align}\label{paper:eq:P}
 P(y,z)&=\frac{\sinh^3y}{\sinh^3\left(y+\frac{z}{3}\right)}U\left(y+\frac{z}{3}\right)\notag\\
 &\quad+\frac34e^{-y-\frac{4z}{3}}(e^{\frac{2z}{3}}-1)^2
  \int_{y+\frac{z}{3}}^\infty
  \frac{e^{-2r}(e^{\frac{2z}{3}}-e^{2r})^2U(r)}{\sinh^5r}\,dr\notag\\
 &\quad+\frac32\sinh y\,e^{-\frac{2z}{3}}(e^{\frac{2z}{3}}-1)
  \int_{y+\frac{z}{3}}^\infty
  \frac{e^{-2r}(e^{\frac{2z}{3}}-e^{2r})
       \bigl(e^{\frac{2z}{3}}-(e^{2r}+1)/2\bigr)U(r)}{\sinh^5r}\,dr
\end{align}
which equals $v(0,0,y/\sqrt2,z/\sqrt2)$ and provides the boundary
data for Region C, as established below. These traces agree with the
first stage: $P(t,0)=\widetilde L(t,0)=U(t)$ for $t>0$, with the
values at $t=0$ understood by continuity; see also \eqref{pc:eq:Lzero}.

\item Third stage: the values on three-dimensional boundary faces.
For a bounded and sufficiently smooth function $\phi$ and for $y>0$
and $p\geq0$, define
\begin{align}\label{paper:eq:K}
 (\mathcal K_p\phi)(y)
 &=\frac{\sinh^2y}{\sinh^2\left(y+\frac{p}{2}\right)}\phi\left(y+\frac{p}{2}\right)
   +4\sinh y\sinh \left(\frac{p}{2}\right)\,J_\phi\left(y+\frac{p}{2}\right)\notag\\
 &\quad+6\sinh^2\left(\frac{p}{2}\right)\left(\cosh \left(y+\frac{p}{2}\right)\,J_\phi\left(y+\frac{p}{2}\right)
                         -\sinh \left(y+\frac{p}{2}\right)\,M_\phi\left(y+\frac{p}{2}\right)\right),\\
 J_\phi(t)&=\int_t^\infty\frac{\phi(r)}{\sinh^3r}\,dr,
 \qquad
 M_\phi(t)=\int_t^\infty\frac{\cosh r\,\phi(r)}{\sinh^4r}\,dr.
 \notag
\end{align}
Thus $\mathcal K_0\phi=\phi$. We also define
\begin{equation}\label{paper:eq:T}
 (\mathcal T\phi)(y)=2\sinh^2y
                  \int_y^\infty\frac{\phi(r)}{\sinh^3r}\,dr.
\end{equation}

For the traces on $(A\cup C)\cap\{q_2=0\}$ and
$(A\cup C)\cap\{q_3=0\}$, use parameters $x,z\geq0$ and $y>0$,
and set
\[
 w=x-z/2,\qquad p=\min\{z,2x\}.
\]
We apply the operators $\mathcal K_p$ and $\mathcal T\circ\mathcal K_p$
to the datum
\begin{equation}\label{pc:eq:datum}
 \phi_w(s)=
 \begin{cases}
 \widetilde L(s,w),&w\geq0,\\
 P(s,-2w),&w<0,
 \end{cases}
\end{equation}
so we obtain the functions
\begin{equation}\label{pc:eq:ACpair}
 F(x,y,z)=(\mathcal K_p\phi_w)(y),\qquad
 R(x,y,z)=(\mathcal T\mathcal K_p\phi_w)(y),
\end{equation}
which give the boundary values
\[
 \begin{aligned}
 F(x,y,z)&=v(x/\sqrt2,0,y/\sqrt2,z/\sqrt2),\\
 R(x,y,z)&=v(x/\sqrt2,y/\sqrt2,0,(y+z)/\sqrt2).
 \end{aligned}
\]
For the traces on $B\cap\{q_4=0\}$ and $B\cap\{q_3=0\}$, set
\[
 b_0(s)=\frac1{2\sqrt2}\left(1+\frac{e^{-2s}}3\right),
 \qquad b_1(s)=\frac2{3\sqrt2}e^{-s},\qquad
 \psi_x(s)=\widetilde L(s,x)-b_0(s),
\]
and define the correction terms
\begin{align}\label{correctionterms}
 H_B(x,y,z)&=(\mathcal K_z\psi_x)(y),\qquad
 R_B(x,y,z)=(\mathcal T\mathcal K_z\psi_x)(y),
\end{align}
The full boundary values are
\[
 \begin{aligned}
 H_B(x,y,z)+b_0(y)&=v(x/\sqrt2,z/\sqrt2,y/\sqrt2,0),\\
 R_B(x,y,z)+b_1(y)&=v(x/\sqrt2,(y+z)/\sqrt2,0,y/\sqrt2).
 \end{aligned}
\]
The values on lower-dimensional edges are obtained by continuity.
\end{enumerate}
We can now give the value of $v(q)$ for every $q\in\mathbb{R}_+^4=A\cup B\cup C$:
\begin{enumerate}
    \item For $q\in A\cup C$, we put
\begin{equation}\label{pc:eq:ACcoordinates}
 x=\sqrt2q_1,\qquad y=\sqrt2(q_2+q_3),\qquad
 z=\sqrt2(q_4-q_2),
\end{equation}
and the value of $v(q)$ is given by:
\begin{equation}\label{pc:eq:vAC}
 \boxed{v(q)=F(x,y,z)\frac{\sinh(\sqrt2q_3)}{\sinh y}
             +R(x,y,z)\frac{\sinh(\sqrt2q_2)}{\sinh y}}.
\end{equation}
On $A$, the above formula uses $p=z$ and the interface datum $\widetilde L(\cdot,x-z/2)$; on $C$, it uses $p=2x$ and the face
datum $P(\cdot,z-2x)$. 

\item For $q\in B$, we put instead
\begin{equation}\label{pc:eq:Bcoordinates}
 x=\sqrt2q_1,\qquad y=\sqrt2(q_3+q_4),\qquad
 z=\sqrt2(q_2-q_4),
\end{equation} 
and the value of $v(q)$ is given by:
\begin{align}\label{pc:eq:vB}
\boxed{v(q)=\bigl(H_B(x,y,z)+b_0(y)\bigr)
                         \frac{\sinh(\sqrt2q_3)}{\sinh y}+\bigl(R_B(x,y,z)+b_1(y)\bigr)
                         \frac{\sinh(\sqrt2q_4)}{\sinh y}}.
\end{align}
There is no switch of datum inside $B$: its datum parameter is $x$ and its propagation distance is always $z$.
\end{enumerate}

\subsection{Classical verification}
\begin{theorem}[The five-expert value function]\label{paper:thm:verification}
The piecewise function \eqref{pc:eq:vAC}--\eqref{pc:eq:vB}
has a ranked extension \eqref{paper:eq:rankvalue} in $C^2(\mathbb R^5)$
which solves \eqref{paper:eq:HJB}. In particular,
\[
 \lim_{\delta\downarrow0}\sqrt\delta\,V_\delta(0)
 =u(0)=\frac{45\pi^2}{512\sqrt2},
\]
and $00101$ is a pointwise maximizing rank control in the limiting
Hamiltonian.
\end{theorem}
\begin{proof}
Theorem~\ref{sm:thm:globalC2} gives $u\in C^2(\R^5)$.
Theorems~\ref{ab:thm:complete} and~\ref{rc:thm:regionC} establish \eqref{paper:eq:gapHJBa} and \eqref{paper:eq:gapHJBb} in all three regions.
In the strict chamber $\{(x_1,x_2,x_3,x_4,x_5):x_1<x_2<x_3<x_4<x_5\}$,
away from the regional interfaces,
these relations are equivalent to \eqref{paper:eq:HJB}.
Continuity of $u,D^2u$, and $g$ passes the equation to the interfaces
of Regions A, B, and C, and to ordering walls where $x_i=x_j$ for
some $i\neq j$. Symmetry then passes the equation to every chamber
$\{(x_1,x_2,x_3,x_4,x_5):x_{\sigma(1)}\leq x_{\sigma(2)}\leq
x_{\sigma(3)}\leq x_{\sigma(4)}\leq x_{\sigma(5)}\}$, where
$\sigma:\{1,2,3,4,5\}\to\{1,2,3,4,5\}$ is a permutation. Boundedness of
$v$, established in Section~\ref{pc:sec:construction}, gives at most
linear growth of $u$; viscosity uniqueness and the
convergence theorem in~\cite{DK} identify $u$ with the limiting value.
The value at the origin is $\frac{45\pi^2}{512\sqrt{2}}$ as given in \eqref{pc:eq:Uendpoints}.  The equality
$\mathcal R_{00101}=0$ proves the assertion about the Hamiltonian.
\end{proof}

In the next section we construct $v$ from its values on lower-dimensional
boundaries, and in Section~\ref{sm:section} we prove its smoothness.
Sections~4--6 are devoted to proving the regional inequalities.

\section{Construction of the candidate}\label{pc:sec:construction}

Throughout this section, the coordinates are ranked in increasing order.
We write
\[
 q_i=x^{(i+1)}-x^{(i)},\qquad i=1,\ldots,4,
\]
for the unscaled gaps.  The candidate has the form
\begin{equation}\label{pc:eq:candidate}
 u(x)=x^{(5)}+v(q_1,q_2,q_3,q_4).
\end{equation}
The one-dimensional function $U$ introduced earlier is a boundary profile
of this five-expert problem; it is distinct from the full candidate $u$.
The construction proceeds from the smallest invariant face outward.
Reflection first determines an ordinary differential equation for $U$.
We then solve the interface system on $q_2=q_4$, obtaining its integral
density, and solve the hyperbolic propagation equations that extend
these data into the three regions. This order explains both the
formulas and the conditions that select them. The boundary
differentiations initially guide the construction; regularity of the
resulting explicit candidate is established independently in
Section~\ref{sm:section}.

\subsection{The reflected-process motivation}

By \cite[Proposition~2.2]{BEZ}, the limiting value function $u$ admits
the stochastic control representation
\begin{equation}\label{StConRep}
 u(x)=\sup_{\sigma}\mathbb E\left[\int_0^\infty e^{-t}g(X_t)\,dt\right],
\end{equation}
where $X_t=x+\int_0^t\sigma_s\,dW_s$, $g(X_t)=\max_i X_{i,t}$,
and $W$ is a standard one-dimensional Brownian motion. The supremum is
over progressively measurable controls taking values in $\{0,1\}^5$.
Write $X_t^{(1)}\leq\cdots\leq X_t^{(5)}$ for the ranked coordinates;
their gaps are $X_t^{(i+1)}-X_t^{(i)}$, $i=1,\ldots,4$.

The candidate rank control is
\begin{equation}
 \sigma_*=(0,0,1,0,1)^\top.
\end{equation}
When the scores are distinct, this selects exactly the third and fifth
ranks and gives the gap increment vector $a=(0,1,-1,1)^\top$.
At collisions, semimartingale decompositions of the ranked coordinates
involve local-time terms; see \cite[Theorem~2.3]{BAG08}. Such
decompositions motivate the reflected formulation, but do not themselves
specify a named-particle feedback law at ties. For the construction,
we define an auxiliary reflected gap process directly by the Skorokhod
problem
\begin{equation}\label{pc:eq:skorokhod}
 Y_t=q+aW_t+R\Lambda_t,\qquad
 R=\begin{pmatrix}
 1&-\frac12&0&0\\
 -\frac12&1&-\frac12&0\\
 0&-\frac12&1&-\frac12\\
 0&0&-\frac12&1
 \end{pmatrix},
\end{equation}
where $Y_t\in[0,\infty)^4$, each $\Lambda^i$ is continuous and
nondecreasing, $\Lambda^i_0=0$, and
\[
 \int_0^t\mathbf 1_{\{Y^i_s>0\}}\,d\Lambda^i_s=0.
\]
These conventions specify the normalization of the boundary regulators.
The matrix $R$ is a nonsingular $M$-matrix: the spectral radius of $I-R$
is $\cos(\pi/5)<1$.  By the orthant Skorokhod theorem~\cite{HR}, the Skorokhod map is
well defined and Lipschitz, including for the degenerate driving
covariance $aa^\top$.

Before the first gap reaches zero, $d\Lambda^1=0$ and
$dY^1=-\tfrac12d\Lambda^2$.  Once it reaches zero it remains there,
and $d\Lambda^1=\tfrac12d\Lambda^2$.  On this face, replacing
$\Lambda^2$ by $\widehat\Lambda^2=\tfrac34\Lambda^2$ gives
\begin{align}\label{pc:eq:collapsed}
 dY^2&=dW+d\widehat\Lambda^2-\tfrac12d\Lambda^3,\notag\\
 dY^3&=-dW+d\Lambda^3-\tfrac23d\widehat\Lambda^2
                         -\tfrac12d\Lambda^4,\\
 dY^4&=dW+d\Lambda^4-\tfrac12d\Lambda^3.\notag
\end{align}
The coefficient $2/3$ is therefore forced by the collapse of the first
gap; it is not a change of the control.

The weighted sum of the gaps satisfies
\begin{equation}\label{pc:eq:weighted}
 d\left(\sum_{i=1}^4 iY^i_t\right)
 =3\,dW_t+\frac52\,d\Lambda^4_t.
\end{equation}
Indeed, multiplication of each column of $R$ by $(1,2,3,4)$ gives
zero except in the fourth column, where it gives $5/2$.
To make the corresponding ordered-score interpretation explicit, set
$\Lambda^0=\Lambda^5=0$ and define
\[
 Z_t^i=x^{(i)}+\sigma_{*,i}W_t
                  +\tfrac12\Lambda_t^{i-1}-\tfrac12\Lambda_t^i,
 \qquad i=1,\ldots,5.
\]
Then $Z_t^{i+1}-Z_t^i=Y_t^i\geq0$ and
$\sum_i Z_t^i=\sum_i x_i+2W_t$. In particular,
$g(Z_t)=Z_t^5=x^{(5)}+W_t+\tfrac12\Lambda_t^4$, so
\[
 \mathbb E\!\left[\int_0^\infty e^{-t}g(Z_t)\,dt\right]
 =x^{(5)}+\frac12\mathbb E_q\!\left[\int_0^\infty
                                      e^{-t}\Lambda_t^4\,dt\right].
\]
This auxiliary reflected-process payoff motivates the representation
\begin{equation}\label{pc:eq:localtime}
 v(q)=\frac12\mathbb E_q\!\left[\int_0^\infty e^{-t}
                         \,d\Lambda^4_t\right]
     =\frac12\mathbb E_q\!\left[\int_0^\infty e^{-t}
                         \Lambda^4_t\,dt\right].
\end{equation}
The last equality is integration by parts; the exponential discount
and the growth estimate for the Skorokhod map remove the terminal term.
The auxiliary ordered process $Z$ is defined from the Skorokhod problem;
no equivalence with a named-particle feedback SDE at multiple collisions
is assumed. The representation above is verified using the linear PDE
and the reflection conditions, and identification with the game value
is made by the nonlinear PDE verification.

In the interior, the linear equation suggested by
\eqref{pc:eq:localtime} is
\begin{equation}\label{pc:eq:linearPDE}
 v-\frac12(\partial_{q_2}-\partial_{q_3}
                         +\partial_{q_4})^2v=0.
\end{equation}
On the relative interiors of the reflecting faces with $q_1>0$, the
boundary conditions are
\begin{align}\label{pc:eq:reflections}
 v_{q_2}-\tfrac12(v_{q_1}+v_{q_3})&=0 &&(q_2=0),\notag\\
 v_{q_3}-\tfrac12(v_{q_2}+v_{q_4})&=0 &&(q_3=0),\\
 v_{q_4}-\tfrac12v_{q_3}&=-\tfrac12 &&(q_4=0).\notag
\end{align}
On $q_1=0$, the first of these becomes
$v_{q_2}-\tfrac23v_{q_3}=0$ at $q_2=0$, as in
\eqref{pc:eq:collapsed}.  The smoothness argument will supply the
closed-chamber interpretation of these identities.

\subsection{From the exit decomposition to the boundary equations}

First, we adapt the stopping argument from \cite[Equations (5.7) - (5.8)]{BEZ} to our setting. Before any reflecting face is reached, the gaps evolve as
$(q_1,q_2+W_t,q_3-W_t,q_4+W_t)$. When $q_2\leq q_4$, the
first exit occurs when $W$ leaves $(-q_2,q_3)$; when $q_4\leq q_2$,
it leaves $(-q_4,q_3)$. The discounted exit probabilities solve
$h''/2=h$ and are ratios of hyperbolic sines. Consequently the
interior problem is determined by two boundary traces, not by an
arbitrary four-dimensional boundary function.

For $q_4\geq q_2$, use $x=\sqrt2q_1$,
$y=\sqrt2(q_2+q_3)$, and $z=\sqrt2(q_4-q_2)$, and write
\[
 F(x,y,z)=v(x/\sqrt2,0,y/\sqrt2,z/\sqrt2),\qquad
 R(x,y,z)=v(x/\sqrt2,y/\sqrt2,0,(y+z)/\sqrt2).
\]
The exit formula is
\begin{equation}\label{pc:eq:exitACderive}
 v(q)=\frac{\sinh(\sqrt2q_3)}{\sinh y}F(x,y,z)
       +\frac{\sinh(\sqrt2q_2)}{\sinh y}R(x,y,z).
\end{equation}
At $q_2=0$, its derivatives are
\[
 v_{q_2}=\sqrt2(F_y-F_z-\coth y\,F+\csch y\,R),\quad
 v_{q_1}=\sqrt2F_x,\quad v_{q_3}=\sqrt2F_y.
\]
The first reflection condition in \eqref{pc:eq:reflections} therefore
gives the first equation below, while at $q_3=0$, the identities
\[
 v_{q_3}=\sqrt2(R_y-\coth y\,R+\csch y\,F),\qquad
 v_{q_2}+v_{q_4}=\sqrt2R_y
\]
give the second:
\begin{align}\label{pc:eq:derivedACsystem}
 (-\partial_x+\partial_y-2\partial_z)F
       &=2\coth y\,F-2\csch y\,R,\notag\\
 R_y&=2\coth y\,R-2\csch y\,F.
\end{align}
The two traces propagate in different characteristic directions;
this coupling is the main boundary problem left by the exit formula.

For $q_2\geq q_4$, instead set $y=\sqrt2(q_3+q_4)$,
$z=\sqrt2(q_2-q_4)$ and define
\[
 \widehat H(x,y,z)=v(x/\sqrt2,z/\sqrt2,y/\sqrt2,0),\qquad
 \widehat R(x,y,z)=v(x/\sqrt2,(y+z)/\sqrt2,0,y/\sqrt2).
\]
Now the weights are $\sinh(\sqrt2q_3)/\sinh y$ and
$\sinh(\sqrt2q_4)/\sinh y$. At $q_4=0$,
\[
 v_{q_4}=\sqrt2(\widehat H_y-\widehat H_z
                      -\coth y\,\widehat H+\csch y\,\widehat R),
 \qquad v_{q_3}=\sqrt2\widehat H_y.
\]
The inhomogeneous reflection condition on this face gives
\begin{align}\label{pc:eq:derivedBsystem}
 (\partial_y-2\partial_z)\widehat H
   &=2\coth y\,\widehat H-2\csch y\,\widehat R-1/\sqrt2,\notag\\
 \widehat R_y&=2\coth y\,\widehat R-2\csch y\,\widehat H.
\end{align}
The elementary pair
\[
 b_0(y)=\frac{1+e^{-2y}/3}{2\sqrt2},\qquad
 b_1(y)=\frac{2e^{-y}}{3\sqrt2}
\]
is a particular solution independent of $z$. Subtracting this pair
leaves the same homogeneous $2$--$2$ propagation problem as in
\eqref{pc:eq:derivedACsystem}, but without the derivative in $x$. This explains the
elementary terms in the Region B formula.

The companion equation already determines its integral operator.
Indeed,
\[
 \partial_y\left(\frac{R}{\sinh^2y}\right)
       =-\frac{2F}{\sinh^3y}.
\]
Boundedness at infinity removes the homogeneous term proportional to
$\sinh^2y$ and gives $R=\mathcal TF$, with $\mathcal T$ defined in
\eqref{paper:eq:T}. It remains to determine the initial traces and solve
the first equation, rather than merely to verify a proposed formula.

\subsection{The diagonal profile}

The set $q_1=0$, $q_2=q_4$ is invariant for the reflected dynamics.
Write $V(a,b)=v(0,a,b,a)$ and define
\begin{equation}\label{pc:eq:Udef}
 U(t)=V(0,t/\sqrt2)=v(0,0,t/\sqrt2,0).
\end{equation}
On this invariant set the second and fourth gaps have the same
driving path and remain equal. Their normalized regulators satisfy
$\widehat\Lambda^2=\Lambda^4=\ell$. Writing their common gap as
$A$ and the third gap as $B$, \eqref{pc:eq:collapsed} becomes
\[
 dA=dW+d\ell-\tfrac12d\Lambda^3,\qquad
 dB=-dW+d\Lambda^3-\tfrac76d\ell.
\]
Thus the coefficient $7/6=2/3+1/2$ records both simultaneous
reflections. The boundary conditions for the local-time resolvent are
\[
 V-\tfrac12(\partial_a-\partial_b)^2V=0,\qquad
 V_a(0,b)-\tfrac76V_b(0,b)=-\tfrac12,
 \qquad \tfrac12V_a(a,0)-V_b(a,0)=0.
\]
For $V_2(t)=V(t,0)$ and $V_3(t)=V(0,t)$, the exit decomposition gives
\begin{equation}\label{pc:eq:diagonalexit}
 V(a,b)=\frac{\sinh(\sqrt2 b)}{\sinh(\sqrt2(a+b))}V_3(a+b)
       +\frac{\sinh(\sqrt2 a)}{\sinh(\sqrt2(a+b))}V_2(a+b).
\end{equation}
Differentiating on the two faces yields
\begin{align}\label{pc:eq:diagonalpair}
 -\tfrac12&=\sqrt2\operatorname{csch}(\sqrt2 t)V_2(t)
       -\sqrt2\coth(\sqrt2 t)V_3(t)-\tfrac16V_3'(t),\notag\\
 0&=-\sqrt2\coth(\sqrt2 t)V_2(t)
       +\sqrt2\operatorname{csch}(\sqrt2 t)V_3(t)+\tfrac12V_2'(t).
\end{align}
For clarity, put $Q(t)=V_2(t/\sqrt2)$. The first equation gives
\[
 Q=\cosh t\,U+\tfrac16\sinh t\,U'
                          -\tfrac1{2\sqrt2}\sinh t,
\]
whereas the second reads $Q'=2\coth t\,Q-2\csch t\,U$.
Differentiating the displayed expression for $Q$ and substituting it
in this last equation eliminates the companion trace and gives
\begin{equation}\label{pc:eq:Uode}
 -U''(t)-5\coth t\,U'(t)+6U(t)
             =\frac3{\sqrt2}\coth t.
\end{equation}

The real integral representation introduced in
\eqref{paper:eq:Uintegral} avoids branch choices for dilogarithms
and gives the first-stage formula described in Section~\ref{subsec14}:
\begin{align}\label{pc:eq:Uintegral}
 H(t)&=\frac{3}{\sqrt2\sinh^5t\cosh^2t}
 \left(\frac t{16}-\frac{\sinh(2t)}{64}
       -\frac{\sinh(4t)}{64}+\frac{\sinh(6t)}{192}\right),\notag\\
 U(t)&=\cosh t\int_t^\infty H(r)\,dr.
\end{align}
To derive it, put $U(t)=\cosh t\,\gamma(t)$ in
\eqref{pc:eq:Uode}.  The integrating factor gives
\[
 \bigl(\gamma'(t)\sinh^5t\cosh^2t\bigr)'
   =-\frac3{\sqrt2}\sinh^4t\cosh^2t.
\]
The primitive in parentheses in \eqref{pc:eq:Uintegral} vanishes at
zero and has derivative $\sinh^4t\cosh^2t$. A nonzero integration
constant would give $\gamma'(t)\sim Ct^{-5}$ near zero, contradicting
boundedness. Once this constant is removed, integration of $\gamma'$
leaves only an additive constant, which would contribute a multiple
of $\cosh t$ to $U$. Boundedness at infinity removes that mode and
gives $\gamma(t)=\int_t^\infty H(r)\,dr$.
In particular,
\begin{equation}\label{pc:eq:Uendpoints}
 U(0)=\frac{45\pi^2}{512\sqrt2},\qquad
 U'(0)=-\frac3{5\sqrt2},\qquad U''(0)=U(0),\qquad
 \lim_{t\to\infty}U(t)=\frac1{2\sqrt2}.
\end{equation}
The derivatives at zero follow directly by inserting a Taylor
expansion into \eqref{pc:eq:Uode}; the value $U(0)$ is the evaluation
of the convergent integral in \eqref{pc:eq:Uintegral}.

\subsection{Solving the interface problem on
\texorpdfstring{$q_2=q_4$}{q2 = q4}}\label{subsec24}

For $r>1$, define
\begin{equation}\label{pc:eq:g}
 g(r)=\frac{8-25r^2+15r^4
       -15r(r^2-1)^2\operatorname{arccoth}r}
                 {16\sqrt2\,r\sqrt{r^2-1}},
 \qquad
 \operatorname{arccoth}r=\frac12\log\frac{r+1}{r-1}.
\end{equation}
For $s>0$ and $w\geq0$, put
\begin{align}\label{pc:eq:L}
 L(s,w)&=\int_1^{\coth s}
       (r\sinh s-\cosh s)e^{-2wr}g(r)\,dr,\notag\\
 \widetilde L(s,w)&=\frac1{\sqrt2}\arctan(e^{-s})\cosh s+L(s,w).
\end{align}
These are the data on $q_2=q_4$; their compatibility with the
diagonal construction is
\begin{equation}\label{pc:eq:Lzero}
 \widetilde L(s,0)=U(s).
\end{equation}
One useful identity for checking the kernel is
\begin{equation}\label{pc:eq:Lidentity}
 L_{ss}(s,w)-L(s,w)
 =-e^{-2w\coth s}g(\coth s)\operatorname{csch}^3s.
\end{equation}
It follows by differentiating \eqref{pc:eq:L}: the integrand vanishes
at its upper endpoint in the first differentiation.  Substitution of
\eqref{pc:eq:g} gives \eqref{pc:eq:Lzero}, with the constants fixed by
\eqref{pc:eq:Uintegral}.

The datum on $q_2=q_4$ is itself determined by a two-dimensional
boundary problem.  Write
\[
 \ell(s,w)=v(w/\sqrt2,0,s/\sqrt2,0),\qquad
 \rho(s,w)=v(w/\sqrt2,s/\sqrt2,0,s/\sqrt2).
\]
The exit decomposition on this interface is
\begin{equation}\label{id:eq:interfaceexit}
 v(q_1,q_2,q_3,q_2)
 =\frac{\sinh(\sqrt2q_3)}{\sinh s}\ell(s,w)
  +\frac{\sinh(\sqrt2q_2)}{\sinh s}\rho(s,w),
 \qquad s=\sqrt2(q_2+q_3),\quad w=\sqrt2q_1.
\end{equation}
Indeed, in scaled variables
$\mathcal V(w,a,b)=v(w/\sqrt2,a/\sqrt2,b/\sqrt2,a/\sqrt2)$,
the interior equation and the two reflection conditions are
\[
 \mathcal V-(\partial_a-\partial_b)^2\mathcal V=0,\qquad
 \mathcal V_a(w,0,b)-\mathcal V_b(w,0,b)
           -\tfrac12\mathcal V_w(w,0,b)=-\frac1{2\sqrt2},
\]
\[
 \tfrac12\mathcal V_a(w,a,0)-\mathcal V_b(w,a,0)=0.
\]
The first reflection condition follows by adding the first and third
equations in \eqref{pc:eq:reflections} at $q_2=q_4=0$.
The interior equation yields \eqref{id:eq:interfaceexit}.
Differentiating that formula on its two boundary faces gives
\[
 \mathcal V_a(w,0,s)=\operatorname{csch}s\,\rho
                   -\coth s\,\ell+\ell_s,\qquad
 \mathcal V_b(w,s,0)=\operatorname{csch}s\,\ell
                   -\coth s\,\rho+\rho_s.
\]
Substitution in the reflection conditions gives
\begin{align}\label{id:eq:interfacepair}
 \ell_w&=2\operatorname{csch}s\,\rho
                   -2\coth s\,\ell+\frac1{\sqrt2},\notag\\
 \rho_s&=2\coth s\,\rho-2\operatorname{csch}s\,\ell.
\end{align}
The initial datum is $\ell(s,0)=U(s)$, with $U$ given by
\eqref{pc:eq:Uintegral}.  We solve these equations among functions
bounded as $s\to\infty$ for each fixed $w\geq0$.
This calculation determines the candidate; regularity needed to
justify the boundary differentiations is verified independently below.

We first look for a pair $(b_{*}, \rho_0)$ that is independent of $w$ and solves our system when it replaces $(\ell, \rho)$.
In that case, the first equation gives
\[
 \rho_0(s)=\cosh s\,b_*(s)-\frac1{2\sqrt2}\sinh s,
\]
and the second becomes
\[
 b_*'(s)-\tanh s\,b_*(s)=-\frac1{2\sqrt2}.
\]
Multiplication by $\operatorname{sech}s$ and boundedness at infinity
therefore yield
\begin{align}\label{id:eq:stationarypair}
 b_*(s)&=\frac{\cosh s}{2\sqrt2}
             \int_s^\infty\operatorname{sech}t\,dt
       =\frac1{\sqrt2}\cosh s\,\arctan(e^{-s}),\notag\\
 \rho_0(s)&=\frac1{\sqrt2}
   \left(\cosh^2s\,\arctan(e^{-s})-\frac12\sinh s\right).
\end{align}
In particular $b_*''-b_*=-\tanh s/(2\sqrt2)$.

To identify the remaining part, eliminate $\rho$ from
\eqref{id:eq:interfacepair}.  The first equation gives
\[
 \rho=\cosh s\,\ell+\frac12\sinh s\,\ell_w
                  -\frac1{2\sqrt2}\sinh s.
\]
Substitution in the second gives $\mathscr E[\ell]=0$, where
\begin{equation}\label{id:eq:eliminated}
 \mathscr E[\ell]
 =\ell_{sw}-\coth s\,\ell_w
      +2\coth s\,\ell_s-2\ell+\frac1{\sqrt2}\coth s.
\end{equation}
The useful factorization is
\begin{equation}\label{id:eq:factorization}
 (\partial_s+\coth s)\mathscr E[\ell]
 =(\partial_w+2\coth s)(\ell_{ss}-\ell)+\frac1{\sqrt2}.
\end{equation}
Consequently, if $L=\ell-b_*$ and $D=L_{ss}-L$, then
\begin{equation}\label{id:eq:curvaturetransport}
 D_w+2\coth s\,D=0,\qquad
 D(s,0)=D_0(s):=U''(s)-U(s)+\frac1{2\sqrt2}\tanh s.
\end{equation}
Thus
\begin{equation}\label{id:eq:transportsolution}
 D(s,w)=e^{-2w\coth s}D_0(s).
\end{equation}
This is the origin of the exponential factor in the interface kernel.

The diagonal integral makes $D_0$ elementary.  If
\[
 I(s)=\frac s{16}-\frac{\sinh(2s)}{64}
                     -\frac{\sinh(4s)}{64}
                     +\frac{\sinh(6s)}{192},
\]
then $I'(s)=\sinh^4s\cosh^2s$, and differentiation of
\eqref{pc:eq:Uintegral} gives
\begin{equation}\label{id:eq:initialcurvature}
 D_0(s)=\frac1{\sqrt2}
  \left(\frac{15I(s)}{\sinh^6s}-3\coth s+\frac12\tanh s\right).
\end{equation}
In particular $D_0(s)=O(e^{-2s})$ at infinity.  Inverting
$\partial_s^2-1$ with no $e^{-s}$ homogeneous component gives
\begin{equation}\label{id:eq:greeninterface}
 L(s,w)=\int_s^\infty\sinh(t-s)
                 e^{-2w\coth t}D_0(t)\,dt.
\end{equation}
The integral is absolutely convergent and is $O(e^{-2s})$.
With $r=\coth t$, define
\begin{equation}\label{id:eq:densityfromU}
 g(\coth t)=-\sinh^3t\,D_0(t).
\end{equation}
Since $dt=-dr/(r^2-1)$ and
$\sinh(t-s)=(\cosh s-r\sinh s)/\sqrt{r^2-1}$,
\eqref{id:eq:greeninterface} becomes
\begin{equation}\label{id:eq:kernelinterface}
 L(s,w)=\int_1^{\coth s}
            (r\sinh s-\cosh s)e^{-2wr}g(r)\,dr.
\end{equation}
Substituting \eqref{id:eq:initialcurvature} in
\eqref{id:eq:densityfromU} gives exactly
\[
 g(r)=\frac{8-25r^2+15r^4
       -15r(r^2-1)^2\operatorname{arccoth}r}
                 {16\sqrt2\,r\sqrt{r^2-1}}.
\]
Hence $\ell=\widetilde L$ as defined in \eqref{pc:eq:L}.

It remains to check that the scalar reduction has neither added nor
discarded a solution of the original system.  Define $M=\coth s$ and
\begin{equation}\label{id:eq:companioninterface}
 \rho(s,w)=\rho_0(s)-\sinh^2s
             \int_1^M(r-M)^2e^{-2wr}g(r)\,dr.
\end{equation}
Differentiation of \eqref{id:eq:kernelinterface} and
\eqref{id:eq:companioninterface} verifies both equations in
\eqref{id:eq:interfacepair}; the moving-endpoint terms vanish because
the integrands vanish at $r=M$.  Equivalently, boundedness of $\rho$
and the integrating factor for its equation give
\[
 \rho(s,w)=2\sinh^2s
              \int_s^\infty\frac{\ell(t,w)}{\sinh^3t}\,dt.
\]
At $w=0$, both \eqref{id:eq:greeninterface} and $U-b_*$ solve
$L_{ss}-L=D_0$.  The diagonal integral implies
\[
 U(s)=\frac1{2\sqrt2}+\frac{e^{-2s}}{2\sqrt2}
                     +O(e^{-4s}),\qquad
 b_*(s)=\frac1{2\sqrt2}+\frac{e^{-2s}}{3\sqrt2}
                     +O(e^{-4s}).
\]
Their difference is therefore $o(e^{-s})$, which excludes both
homogeneous modes and proves $\ell(s,0)=U(s)$.

Finally, suppose two bounded solutions of
\eqref{id:eq:interfacepair} have the same initial datum.
By \eqref{id:eq:curvaturetransport}, their first-component difference
has the form $A(w)\cosh s+B(w)\sinh s$.
Boundedness in $s$ reduces this to $C(w)e^{-s}$.
Equation \eqref{id:eq:eliminated} then gives $C'+2C=0$;
the initial condition gives $C(0)=0$, so $C=0$.
The first equation of \eqref{id:eq:interfacepair} determines the
second component uniquely.  In particular, the omitted $e^{-s}$
mode in \eqref{id:eq:greeninterface} is not a free boundary parameter:
the interface system would propagate it as a constant multiple of
$e^{-s-2w}$, and compatibility with $U$ forces that constant to vanish.

As an additional compatibility identity, differentiating
\eqref{pc:eq:Uode} shows that
\[
 (\partial_s^2-1)(U'/3)=-2\coth s\,D_0(s).
\]
Comparison with the $w$-derivative of
\eqref{id:eq:greeninterface}, again excluding the homogeneous modes
by decay, gives
\begin{equation}\label{id:eq:normalcompatibility}
 \ell_w(s,0)=\frac13U'(s),\qquad
 \rho(s,0)=\cosh s\,U(s)+\frac16\sinh s\,U'(s)
                             -\frac1{2\sqrt2}\sinh s.
\end{equation}
These are precisely the compatibility relations with the diagonal
construction.

There is no additional free datum on the collapsed face $s=0$.
The density satisfies $g(r)=O((r-1)^{-1/2})$ as $r\downarrow1$
and $g(r)=O(r^{-4})$ as $r\to\infty$.  Moreover,
$-1\leq r\sinh s-\cosh s\leq0$ for $1\leq r\leq\coth s$.
Dominated convergence in the two kernel formulas therefore gives
\begin{equation}\label{id:eq:collapsedinterface}
 \ell(0,w)=\rho(0,w)
 =\frac{\pi}{4\sqrt2}-\int_1^\infty e^{-2wr}g(r)\,dr.
\end{equation}
At $w=0$ this agrees with the already determined value $U(0)$.

\subsection{Solving the propagation systems}\label{secpropsystems}

For a bounded smooth datum $\phi$ and $y>0$, define
\begin{equation}\label{pc:eq:T}
 (\mathcal T\phi)(y)=2\sinh^2y
                  \int_y^\infty\frac{\phi(r)}{\sinh^3r}\,dr.
\end{equation}
For $p\geq0$, let $d=p/2$, $t=y+d$, and set
\begin{align}\label{pc:eq:K}
 (\mathcal K_p\phi)(y)
 &=\frac{\sinh^2y}{\sinh^2t}\phi(t)
   +4\sinh y\sinh d\,J_\phi(t)\notag\\
 &\quad+6\sinh^2d\bigl(\cosh t\,J_\phi(t)
                         -\sinh t\,M_\phi(t)\bigr),\\
 J_\phi(t)&=\int_t^\infty\frac{\phi(r)}{\sinh^3r}\,dr,
 \qquad
 M_\phi(t)=\int_t^\infty\frac{\cosh r\,\phi(r)}{\sinh^4r}\,dr.
 \notag
\end{align}
Thus $\mathcal K_0\phi=\phi$.  This is a propagation operator for the
two-by-two hyperbolic system.  More precisely, if
$F(y,p)=(\mathcal K_p\phi)(y)$ and
$R(y,p)=(\mathcal T\mathcal K_p\phi)(y)$, then
\begin{align}\label{pc:eq:22system}
 (\partial_y-2\partial_p)F&=2\coth y\,F
                              -2\operatorname{csch}y\,R,\notag\\
 \partial_yR&=-2\operatorname{csch}y\,F+2\coth y\,R.
\end{align}
The second identity is immediate from \eqref{pc:eq:T}; the first
follows by differentiating \eqref{pc:eq:K}, using
$J_\phi'=-\phi\operatorname{csch}^3$ and
$M_\phi'=-\cosh\,\phi\operatorname{csch}^4$.
Its infinitesimal form is
\begin{equation}\label{pc:eq:generator}
 \left.\partial_p\mathcal K_p\phi\right|_{p=0}
 =\frac12\phi'-\coth y\,\phi
                       +\operatorname{csch}y\,\mathcal T\phi.
\end{equation}

The interior and collapsed-face equations can be solved together.
For $m\in\{2,3\}$ consider
\begin{align}\label{pc:eq:m2system}
 F_y-mF_p&=m\coth y\,F-m\csch y\,R,\notag\\
 R_y&=2\coth y\,R-2\csch y\,F,
 \qquad F(y,0)=\phi(y).
\end{align}
We select the bounded solution, with
$R(y,p)/\sinh^2y\to0$ as $y\to\infty$.
The second equation has integrating factor $\sinh^{-2}y$ and therefore
\begin{equation}\label{pc:eq:Tderivation}
 \left(\frac{R}{\sinh^2y}\right)_y
 =-\frac{2F}{\sinh^3y},\qquad
 R(y,p)=2\sinh^2y\int_y^\infty\frac{F(r,p)}{\sinh^3r}\,dr.
\end{equation}
This derives the companion operator $\mathcal T$.

\paragraph{Elimination and transport of the curvature.}
The first equation of \eqref{pc:eq:m2system} gives
\[
 R=\cosh y\,F-\frac{\sinh y}{m}(F_y-mF_p).
\]
Substituting this expression into the second equation gives
\begin{equation}\label{pc:eq:Em}
 \mathscr E_mF:=F_{yy}-mF_{yp}-(m+1)\coth y\,F_y
                 +m\coth y\,F_p+mF=0.
\end{equation}
Indeed, the second-equation residual is
$-(\sinh y/m)\mathscr E_mF$.  The decisive factorization is
\begin{equation}\label{pc:eq:curvaturefactorization}
 (\partial_y+\coth y)\mathscr E_mF
  =(\partial_y-m\partial_p-m\coth y)(F_{yy}-F).
\end{equation}
Consequently, writing $H=F_{yy}-F$, the quantity
$H/\sinh^my$ is constant along the characteristics of
$\partial_y-m\partial_p$.  Such a characteristic meets $p=0$ at
$t=y+p/m$. Hence
\begin{equation}\label{pc:eq:curvaturetransport}
 F_{yy}(y,p)-F(y,p)
  =\frac{\sinh^my}{\sinh^m(y+p/m)}
        (\phi''-\phi)(y+p/m).
\end{equation}

\paragraph{Reconstruction from compactly supported data.}
First take a smooth compactly supported datum $\phi$. Set
$d=p/m$ and $t=y+d$. The solution with compact support in the $y$
variable of \eqref{pc:eq:curvaturetransport} is
\begin{equation}\label{pc:eq:compactGreen}
 F(y,p)=\int_t^\infty
       \sinh(r-t)\left(\frac{\sinh(r-d)}{\sinh r}\right)^m
       (\phi''-\phi)(r)\,dr.
\end{equation}
This is the elementary terminal Green formula for
$\partial_y^2-1$: applying that operator leaves the integrand at
the moving lower endpoint. At $p=0$, two integrations by parts give
$F(y,0)=\phi(y)$.

It is important to recover the original equation, not merely its
differentiated consequence. By \eqref{pc:eq:curvaturefactorization},
the function in \eqref{pc:eq:compactGreen} has
$\mathscr E_mF=c(p)/\sinh y$. It and all its derivatives vanish for
$y$ beyond the support of $\phi$, so $c(p)=0$. Defining $R$ by the
first equation now gives the second equation and the terminal
condition; thus $R=\mathcal TF$ by \eqref{pc:eq:Tderivation}.

\paragraph{Integration by parts and bounded data.}
Let
\[
 W(r)=\sinh(r-t)\left(\frac{\sinh(r-d)}{\sinh r}\right)^m.
\]
For $m\in\{2,3\}$, $p\geq0$, and $y>0$, set $d=p/m$ and $t=y+d$,
and define the propagation operator $\mathcal S^{(m)}_p$ by
\begin{equation}\label{pc:eq:generalpropagator}
 (\mathcal S^{(m)}_p\phi)(y)
 :=\left(\frac{\sinh y}{\sinh t}\right)^m\phi(t)
       +\int_t^\infty V_m(y,d,r)\phi(r)\,dr,
\end{equation}
where
\begin{align}\label{pc:eq:Vm}
 V_m(y,d,r)={}&
 \frac{2m\sinh d\sinh t\sinh^{m-1}(r-d)}{\sinh^{m+2}r}\notag\\
 &+\frac{m(m-1)\sinh^2d\sinh(r-t)\sinh^{m-2}(r-d)}
          {\sinh^{m+2}r}.
\end{align}
For smooth compactly supported $\phi$, the identities $W(t)=0$ and
$W'(t)=(\sinh y/\sinh t)^m$ show that two integrations by parts in
\eqref{pc:eq:compactGreen} give
$F(y,p)=(\mathcal S^{(m)}_p\phi)(y)$.
To obtain this kernel directly, use
\[
 A(r)=\frac{\sinh(r-d)}{\sinh r}
       =\cosh d-\coth r\sinh d,
 \qquad A'(r)=\frac{\sinh d}{\sinh^2r},
\]
and simplify $W''-W$; the result is exactly $V_m$.
Both terms of $V_m$ are nonnegative for $r\ge t\ge d$.
Moreover, for $(y,p)$ in a compact subset with $y>0$, the kernel
and every fixed finite number of its $(y,p)$ derivatives have
integrable exponential bounds as $r\to\infty$.
Thus \eqref{pc:eq:generalpropagator}, unlike the pre-integration
formula \eqref{pc:eq:compactGreen}, applies to bounded smooth data:
cut off $\phi$ at infinity, use dominated convergence in this
formula and its local derivatives, and recover
\eqref{pc:eq:m2system}. This applies in particular to the explicit
data used below, which are smooth for positive arguments and bounded.

\paragraph{Selection of the remaining bounded mode.}
The curvature equation alone admits the addition $c(p)e^{-y}$.
For a bounded solution of the full system this coefficient is not
arbitrary, because
\begin{equation}\label{pc:eq:modeODE}
 \mathscr E_m(c(p)e^{-y})
  =\frac{mc'(p)+(m+1)c(p)}{\sinh y}.
\end{equation}
Hence $c(p)=c(0)e^{-(m+1)p/m}$. Two bounded solutions with the same
initial datum have zero curvature difference by
\eqref{pc:eq:curvaturetransport}; their difference is $c(p)e^{-y}$,
since the other homogeneous mode $e^y$ is unbounded. Their common
initial datum gives $c(0)=0$, so $c(p)=c(0)e^{-(m+1)p/m}=0$,
proving uniqueness. In particular,
\[
 \mathcal S^{(m)}_p(e^{-\cdot})(y)
   =e^{-y-(m+1)p/m},\qquad
 \mathcal T(e^{-\cdot})(y)=e^{-2y}.
\]
The uniqueness argument also gives the semigroup property in $p$.

\paragraph{The interior propagator.}
For $m=2$, $d=p/2$, expansion of \eqref{pc:eq:Vm} gives
\[
 V_2=
 \frac{4\sinh y\sinh d+6\sinh^2d\cosh t}{\sinh^3r}
 -\frac{6\sinh^2d\sinh t\cosh r}{\sinh^4r}.
\]
Substitution in \eqref{pc:eq:generalpropagator} yields precisely
\eqref{pc:eq:K}; thus $\mathcal K_p=\mathcal S^{(2)}_p$.

\subsection{The new boundary datum in Region C}

On the face $q_1=0$, set
$P(y,z)=v(0,0,y/\sqrt2,z/\sqrt2)$ and
$R_0(y,z)=(\mathcal T[P(\cdot,z)])(y)$.
The collapsed reflection equations produce the system
\begin{align}\label{pc:eq:32system}
 (\partial_y-3\partial_z)P&=3\coth y\,P
                       -3\operatorname{csch}y\,R_0,\notag\\
 \partial_yR_0&=-2\operatorname{csch}y\,P+2\coth y\,R_0,
 \qquad P(y,0)=U(y).
\end{align}
The coefficient pair $(3,2)$ distinguishes this boundary system from
the interior propagation system \eqref{pc:eq:22system}.

For $\alpha=z/3$ and $\rho=y+\alpha$, the required datum is
\begin{align}\label{pc:eq:P}
 P(y,z)&=\frac{\sinh^3y}{\sinh^3\rho}U(\rho)\notag\\
 &\quad+\frac34e^{-y-4\alpha}(e^{2\alpha}-1)^2
  \int_\rho^\infty
  \frac{e^{-2r}(e^{2\alpha}-e^{2r})^2U(r)}{\sinh^5r}\,dr\notag\\
 &\quad+\frac32\sinh y\,e^{-2\alpha}(e^{2\alpha}-1)
  \int_\rho^\infty
  \frac{e^{-2r}(e^{2\alpha}-e^{2r})
       \bigl(e^{2\alpha}-(e^{2r}+1)/2\bigr)U(r)}{\sinh^5r}\,dr.
\end{align}
In particular $P(y,0)=U(y)$.  A useful scalar identity satisfied by
this formula is
\begin{equation}\label{pc:eq:Pode}
 P_{yy}(y,z)-P(y,z)
 =\frac{\sinh^3y}{\sinh^3(y+z/3)}
                       (U''-U)(y+z/3).
\end{equation}
\paragraph{The collapsed-face datum.}
For $m=3$, put $p=z$, $d=\alpha=z/3$, $t=\rho=y+\alpha$,
and $\phi=U$. Formula \eqref{pc:eq:generalpropagator} becomes
\begin{align}\label{pc:eq:Pderived}
 P(y,z)={}&\frac{\sinh^3y}{\sinh^3\rho}U(\rho)\notag\\
 &+6\sinh\alpha\int_\rho^\infty
 \frac{\sinh(r-\alpha)
  \{\sinh\rho\sinh(r-\alpha)
       +\sinh\alpha\sinh(r-\rho)\}}{\sinh^5r}
 U(r)\,dr.
\end{align}
This is the same datum as \eqref{pc:eq:P}. To see the equivalence,
put $b=\sinh\alpha$ and $S_r=\sinh(r-\alpha)$: the numerator of
its integral is
\begin{equation}\label{tedioustoverify}
6bS_r\{\sinh\rho S_r+b\sinh(r-\rho)\}
 =12e^{-y}b^2S_r^2
   +6\sinh y\,bS_r(\cosh r-e^{2\alpha-r}).
\end{equation}
For $z>0$, this identity follows by cancelling $bS_r$ and expressing
$b$, $S_r$, and the remaining hyperbolic functions in terms of
exponentials. At $z=0$, $b=0$ and both sides vanish.
Replacing the hyperbolic factors on the right of
\eqref{tedioustoverify} by exponentials gives the two integrands of
\eqref{pc:eq:P}. Together with the computations in
Section~\ref{subsec24}, this establishes the second-stage description
in Section~\ref{subsec14}.
Equation \eqref{pc:eq:Pode} is now derived from
\eqref{pc:eq:curvaturetransport}, rather than used as an isolated
verification. The missing bounded mode is fixed by
$3c'(z)+4c(z)=0$ and the initial datum $P(y,0)=U(y)$, as shown above.

\subsection{Why the three regions occur and the value in each of them}
All the auxiliary functions defined in Section~\ref{subsec14} have now been obtained in the preceding boundary calculations. The
characteristics explain precisely the partition in equations \eqref{pc:eq:vAC} and \eqref{pc:eq:vB}. For the pair
on $A\cup C$, the three-variable system is
\begin{align}\label{pc:eq:ACsystem}
 (-\partial_x+\partial_y-2\partial_z)F
       &=2\coth y\,F-2\operatorname{csch}y\,R,\notag\\
 \partial_yR&=-2\operatorname{csch}y\,F+2\coth y\,R.
\end{align}
The quantity $w=x-z/2$ is invariant under the first differential
operator. Writing $F$ and $R$ in the coordinates $(w,y,p)$, either
choice $p=z$ or $p=2x$ turns that operator into
$\partial_y-2\partial_p$. The discussion in
Section~\ref{secpropsystems} therefore gives the formulas for $F$
and $R$ in \eqref{pc:eq:ACpair}. The backward characteristic reaches
$z=0$ first when $z\leq2x$, and $x=0$ first when $z\geq2x$.
This gives $A$ and $C$, respectively.  In $B$ the first operator is
$\partial_y-2\partial_z$, leaving $x$ fixed, so the formulas in
\eqref{correctionterms} follow from the same propagation kernel.

The weights in \eqref{pc:eq:vAC} and \eqref{pc:eq:vB} are the
discounted exit probabilities for a one-dimensional Brownian motion
between its two nearest reflecting faces, as in
\eqref{pc:eq:diagonalexit}.  At $A\cap C$, compatibility follows from
$P(s,0)=\widetilde L(s,0)=U(s)$.  At $A\cap B$, it follows from
$\mathcal K_0=I$ and $\mathcal T b_0=b_1$.
At vanishing denominators in the exit weights, the expressions are
defined by their continuous limits.  The separate smoothness argument in the next section
establishes the matching of the full second-order derivatives,
including these collision points.  The nonlinear verification then
checks
\[
 v-\tfrac12\left(\sum_{i=1}^4
                    (\omega_{i+1}-\omega_i)\partial_{q_i}\right)^2v
 \geq0\qquad(\omega\in\{0,1\}^5),
\]
with equality for $\omega=\sigma_*$.

\subsection{Single-integral representation in Region C}
\label{pc:sec:singleintegral}

The composition $\mathcal K_{2x}[P(\cdot,z-2x)]$ contains nested
integrals when \eqref{pc:eq:P} is substituted.  They can all be
removed by evaluating elementary primitives.  We give the reduction
explicitly, since the resulting kernels are also useful for checking
the additional equality directions.

Continue to use the coordinates \eqref{pc:eq:ACcoordinates}, with
$z\geq2x$, and set
\[
 \zeta=z-2x,\qquad \alpha=\zeta/3,\qquad
 a=x+y,\qquad \rho=a+\alpha,\qquad D_t=e^{2t}-1.
\]
Here $\zeta$ is the nonnegative distance parameter, whereas the
characteristic invariant is $w=x-z/2=-\zeta/2$.
Let $p=F$ and $r_3=R$ denote the Region C pair in
\eqref{pc:eq:ACpair}. Applying $\mathcal K_{2x}$ in
\eqref{pc:eq:K}, with $d=x$ and $t=y+x=a$, gives
\begin{align}\label{pc:eq:nestedC}
 p(x,y,z)&=C_0P(a,\zeta)+C_1J_1(a,\zeta)+C_2J_2(a,\zeta),
\end{align}
where every prefactor is explicit:
\begin{align}\label{pc:eq:Cprefactors}
 C_0&=\frac{e^{2x}D_y^2}{D_a^2},&
 C_1&=8e^{-x-y}D_xD_y,&
 C_2&=12e^{-3x-y}D_x^2,
\end{align}
and the nested terms are
\begin{align}\label{pc:eq:Jnested}
J_j(a,\zeta)&=\int_a^\infty\Gamma_j^a(s)P(s,\zeta)\,ds,
 \qquad j=1,2,\notag\\
 \Gamma_1^a(s)&=\frac{e^{3s}}{D_s^3},\qquad
 \Gamma_2^a(s)=\frac{e^{3s}}{D_s^3}\frac{D_s-D_a}{D_s}.
\end{align}
Substituting $\mathcal K_{2x}[P(\cdot,\zeta)]$ from \eqref{pc:eq:K}
into the definition of $\mathcal T$ in \eqref{pc:eq:T}, and integrating
by parts to eliminate the double integrals, gives
\begin{align}\label{pc:eq:nestedC2}
 r_3(x,y,z)&=R_1J_1(a,\zeta)+R_2J_2(a,\zeta),
\end{align}
with
\begin{align}\label{pc:eq:Cprefactors2}
 R_1&=4e^{-x-2y}D_aD_y,&
 R_2&=12e^{-3x-2y}D_aD_x.
\end{align}
To shorten the boundary expression, introduce
\begin{align}\label{pc:eq:Fkernels}
 \phi_\alpha(s)&=\frac{\sinh^3s}{\sinh^3(s+\alpha)},\qquad
 c_1=\frac34e^{-4\alpha}(e^{2\alpha}-1)^2,
 \qquad c_2=\frac32e^{-2\alpha}(e^{2\alpha}-1),\notag\\
 \mathsf F_1(r)&=\frac{e^{-2r}(e^{2\alpha}-e^{2r})^2}{\sinh^5r},\notag\\
 \mathsf F_2(r)&=\frac{e^{-2r}(e^{2\alpha}-e^{2r})
             \bigl(e^{2\alpha}-(e^{2r}+1)/2\bigr)}{\sinh^5r}.
\end{align}
Thus \eqref{pc:eq:P} reads
\begin{equation}\label{pc:eq:Pshort}
 P(s,\zeta)=\phi_\alpha(s)U(s+\alpha)
  +c_1e^{-s}\int_{s+\alpha}^\infty\mathsf F_1(r)U(r)\,dr
  +c_2\sinh s\int_{s+\alpha}^\infty\mathsf F_2(r)U(r)\,dr.
\end{equation}
For $b\geq a>0$, define the four elementary primitives
\[
 I_{j,\pm}(a,b)=\int_a^b\Gamma_j^a(s)e^{\pm s}\,ds.
\]
The substitution $D=e^{2s}-1$ gives
\begin{align}\label{pc:eq:Iprimitives}
 I_{1,-}(a,b)&=\frac1{4D_a^2}-\frac1{4D_b^2},\notag\\
 I_{1,+}(a,b)&=\frac1{2D_a}+\frac1{4D_a^2}
                    -\frac1{2D_b}-\frac1{4D_b^2},\notag\\
 I_{2,-}(a,b)&=\frac1{12D_a^2}-\frac1{4D_b^2}
                                      +\frac{D_a}{6D_b^3},\\
 I_{2,+}(a,b)&=\frac1{4D_a}+\frac1{12D_a^2}-\frac1{2D_b}
                -\frac{1-D_a}{4D_b^2}+\frac{D_a}{6D_b^3}.\notag
\end{align}

For $r\geq\rho$, put $b=r-\alpha$ and define
\begin{align}\label{pc:eq:Hkernels}
 \mathsf H_j(a,r;\alpha)
 &=\Gamma_j^a(b)\phi_\alpha(b)
       +c_1\mathsf F_1(r)I_{j,-}(a,b)\notag\\
 &\quad+\frac{c_2}{2}\mathsf F_2(r)
             \bigl(I_{j,+}(a,b)-I_{j,-}(a,b)\bigr),\qquad j=1,2.
\end{align}
Then
\begin{equation}\label{pc:eq:Jsingle}
 J_j(a,\zeta)=\int_\rho^\infty\mathsf H_j(a,r;\alpha)U(r)\,dr.
\end{equation}
To prove this identity, the first term of \eqref{pc:eq:Pshort} uses
the substitution $r=s+\alpha$.  For either of the other two terms,
reverse the order on the triangular domain
$a\leq s<\infty$, $r\geq s+\alpha$:
\begin{align}\label{pc:eq:fubini}
 &\int_a^\infty\Gamma_j^a(s)A(s)
                  \int_{s+\alpha}^\infty f(r)U(r)\,dr\,ds\notag\\
 &\hspace{1cm}=\int_\rho^\infty f(r)U(r)
                    \left(\int_a^{r-\alpha}\Gamma_j^a(s)A(s)\,ds\right)dr.
\end{align}
Here $A(s)=e^{-s}$ or $\sinh s$.  Formula
\eqref{pc:eq:Iprimitives}, together with
$\sinh s=(e^s-e^{-s})/2$, gives \eqref{pc:eq:Hkernels}.
For $a>0$, boundedness of $U$ and exponential decay of the kernels
give absolute integrability, so Fubini is justified.  All limiting
face cases are subsequently obtained by continuity.

Set
\begin{align}\label{pc:eq:finalCkernels}
 B_0(a,r;\alpha)&=c_1e^{-a}\mathsf F_1(r)
                         +c_2\sinh a\,\mathsf F_2(r),\notag\\
 K_p(x,y,z;r)&=C_0B_0+C_1\mathsf H_1+C_2\mathsf H_2,\notag\\
 K_r(x,y,z;r)&=R_1\mathsf H_1+R_2\mathsf H_2,
\end{align}
where the parameters $a,\alpha$ in every kernel are as above.
The Region C pair now has the single-integral representation
\begin{align}\label{pc:eq:pairCsingle}
 p(x,y,z)&=A_0U(\rho)+\int_\rho^\infty K_p(x,y,z;r)U(r)\,dr,\notag\\
 r_3(x,y,z)&=\int_\rho^\infty K_r(x,y,z;r)U(r)\,dr,\qquad
 A_0=C_0\phi_\alpha(a)=\frac{\sinh^2y\sinh a}{\sinh^3\rho}.
\end{align}
In particular, with
$\lambda_p=\sinh(\sqrt2q_3)/\sinh y$ and
$\lambda_r=\sinh(\sqrt2q_2)/\sinh y$,
\begin{equation}\label{pc:eq:valueCsingle}
 v(q)=\lambda_pA_0U(\rho)+\int_\rho^\infty
                (\lambda_pK_p+\lambda_rK_r)U(r)\,dr.
\end{equation}
This is exactly the original boundary expression propagated into
Region C, with the two nested integrations replaced by elementary
functions.  No modification of the candidate is involved.

For use in checking the equality directions, put
$\eta=\sqrt2q_2$, so that
$\lambda_p=\sinh(y-\eta)/\sinh y$ and
$\lambda_r=\sinh\eta/\sinh y$, and regard the expressions above as
functions of the independent variables $(x,y,\alpha,\eta)$.
Define
\[
 Q_B=\partial_x-\partial_y-\partial_\eta,
 \qquad Q_C=\partial_x-\partial_\alpha.
\]
Substitution of the elementary primitives gives the kernel identities
\begin{align}\label{pc:eq:kernelzeros}
 (Q_B^2-1)(\lambda_pA_0)&=0,&
 (Q_B^2-1)(\lambda_pK_p+\lambda_rK_r)&=0,\notag\\
 (Q_C^2-1)A_0&=0,&
 (Q_C^2-1)K_p=(Q_C^2-1)K_r&=0.
\end{align}
These are identities of elementary functions, independent of $U$.
For the local term, they follow immediately from the simplified
formula for $A_0$ in \eqref{pc:eq:pairCsingle}; the kernel checks use
the same differentiation and \eqref{pc:eq:Iprimitives}.
Both operators leave $\rho=x+y+\alpha$ fixed.  Hence the identities
pass through the integrals without moving-endpoint terms, supplying
the two additional linear equalities used in the verification.

\subsection{Boundedness and the fixed-policy expectation}
\label{pc:sec:ito}

The boundedness condition follows directly from the integral kernels.
For $t>0$, the substitution $s=\sinh r$ gives
\[
 \int_t^\infty\frac{dr}{\sinh^3r}
 =\int_{\sinh t}^\infty\frac{ds}{s^3\sqrt{1+s^2}}
 \leq\frac1{2\sinh^2t}.
\]
Consequently, the positive operator $\mathcal T$ satisfies
\begin{equation}\label{pc:eq:Tcontraction}
 0\leq(\mathcal T1)(y)
 =2\sinh^2y\int_y^\infty\frac{dr}{\sinh^3r}\leq1.
\end{equation}
Thus $\|\mathcal T\phi\|_\infty\leq\|\phi\|_\infty$ for every
bounded datum $\phi$.

For the propagation operator in \eqref{pc:eq:generalpropagator},
let $m\in\{2,3\}$, $p\geq0$, $d=p/m$, and $t=y+d$ with $y>0$.
For $r\geq t$,
\[
 0\leq\frac{\sinh(r-d)}{\sinh r}\leq1,\qquad
 \sinh(r-t)\leq\sinh r,\qquad \sinh d\leq\sinh t.
\]
Applying these inequalities to the two nonnegative terms of
$V_m$ in \eqref{pc:eq:Vm}, and using the integral estimate above,
gives
\[
 \int_t^\infty V_m(y,d,r)\,dr
 \leq m\frac{\sinh d}{\sinh t}
       +\frac{m(m-1)}2\frac{\sinh^2d}{\sinh^2t}
 \leq m+\frac{m(m-1)}2.
\]
The coefficient $(\sinh y/\sinh t)^m$ of the local term is at most
one. Hence, for every bounded datum $\phi$,
\begin{equation}\label{pc:eq:propagatorbound}
 \sup_{p\geq0,\,y>0}|(\mathcal S_p^{(m)}\phi)(y)|
 \leq\left(1+m+\frac{m(m-1)}2\right)\|\phi\|_\infty.
\end{equation}
In particular, the propagation bounds are $4\|\phi\|_\infty$ for
$m=2$ and $7\|\phi\|_\infty$ for $m=3$, uniformly in both the
spatial variable and the propagation distance.

The continuous function $U$ is bounded because its limits at zero
and infinity are finite; see \eqref{pc:eq:Uendpoints}. For the
interface datum, $\arctan(e^{-s})\cosh s\leq e^{-s}\cosh s\leq1$,
and $|r\sinh s-\cosh s|\leq e^{-s}\leq1$ on the integration range
of \eqref{pc:eq:L}. The function $g$ is absolutely integrable on
$(1,\infty)$, since its endpoint orders are $O((r-1)^{-1/2})$ at
$1$ and $O(r^{-4})$ at infinity. Thus \eqref{pc:eq:L} gives
\[
 \sup_{s>0,\,w\geq0}|\widetilde L(s,w)|
 \leq\frac1{\sqrt2}+\int_1^\infty|g(r)|\,dr<\infty.
\]
Since $P(y,z)=(\mathcal S_z^{(3)}U)(y)$, estimate
\eqref{pc:eq:propagatorbound} first gives a uniform bound for $P$.
The data families $\phi_w$ and $\psi_x$ are therefore uniformly
bounded as well. Using $\mathcal K_p=\mathcal S_p^{(2)}$ gives
uniform bounds for $\mathcal K_p\phi_w$ and $\mathcal K_z\psi_x$;
\eqref{pc:eq:Tcontraction} bounds their companions. The elementary
terms $b_0,b_1$ added in Region B are also bounded. Finally, the
nonnegative exit weights in each value formula have sum at most one.
Thus $v$ is uniformly bounded wherever these formulas have positive
denominators. The bound also holds on the remaining faces by the
continuous extension established in Section~\ref{sm:section}, so
$v$ is bounded on the whole orthant.

Once the closed-chamber regularity and reflection
identities are established, its interpretation
\eqref{pc:eq:localtime} follows from a short verification argument.
More explicitly, the ordering-wall identities give
\begin{equation}\label{pc:eq:allreflection}
 (R^\top\nabla v)_i=-\tfrac12\mathbf 1_{\{i=4\}}
 \quad\text{on }\{q_i=0\},\qquad i=1,\ldots,4.
\end{equation}
This includes $v_{q_1}-\tfrac12v_{q_2}=0$ on the first face.
Continuity of the gradient makes the same identities valid at
intersections of faces.

Apply It\^o's formula, with localization, to $e^{-t}v(Y_t)$ for the
process \eqref{pc:eq:skorokhod}.  The interior drift vanishes by
\eqref{pc:eq:linearPDE}; continuity extends that identity to all
points of the closed orthant.  The regulator terms are
$-\tfrac12e^{-t}d\Lambda^4_t$ by \eqref{pc:eq:allreflection}.
Taking expectations and removing the localization therefore gives
\begin{equation}\label{pc:eq:itoverification}
 v(q)=\mathbb E_q[e^{-T}v(Y_T)]
            +\frac12\mathbb E_q\!\left[\int_0^T e^{-t}
                                      \,d\Lambda^4_t\right].
\end{equation}
Boundedness of the chosen branch sends the first term to zero as
$T\to\infty$, proving \eqref{pc:eq:localtime}.  In particular $v\geq0$.
This argument uses the linear equation and the independently checked
smoothness and boundary conditions; it does not use the remaining
nonlinear HJB inequalities.

\section{Smoothness across interfaces and ordering walls}\label{sm:section}

Let
\[
 q_i=x^{(i+1)}-x^{(i)},\qquad i=1,\ldots,4,
\]
and write the chamber candidate as
\[
 \mathcal U(x_1,\ldots,x_5)=x_5+v(q_1,q_2,q_3,q_4),
 \qquad x_1\leq\cdots\leq x_5.
\]
The three formula regions are
\[
 \begin{aligned}
 A&=\{q_4\geq q_2,\ 2q_1\geq q_4-q_2\},\\
 B&=\{q_4\leq q_2\},\\
 C&=\{q_4\geq q_2,\ 2q_1\leq q_4-q_2\}.
 \end{aligned}
\]
There are only two codimension-one internal interfaces:
\[
 \Sigma_{AB}=\{q_4=q_2\},\qquad
 \Sigma_{AC}=\{q_4-q_2=2q_1\}.
\]
The closures of $B$ and $C$ meet only on the codimension-two set
$\{q_1=0,q_4=q_2\}$.

We prove that the explicit piecewise candidate extends twice continuously
differentiably through every internal interface and every ordering
collision.  There are two distinct issues.  First, the two-jets of the
regional formulas must agree on $\Sigma_{AB}$ and $\Sigma_{AC}$.  Second,
the one-sided Hessian must have a continuous limit on every stratum of the
closed Weyl chamber, including the full collision.  Cubic contact across
the internal interfaces, weighted third-order endpoint estimates, and the
four ordering-wall smooth-fit identities resolve these two issues.

\begin{theorem}[Global smoothness]\label{sm:thm:globalC2}
The regional formulas define a function
$v\in C^2(\overline{\mathbb R_+^4})$.  The chamber function
$\mathcal U\in C^2(\overline{\mathcal W}_5)$, and its symmetric ranked
extension
\[
 u(x)=\mathcal U(x^{(1)},\ldots,x^{(5)})
\]
belongs to $C^2(\mathbb R^5)$.
\end{theorem}

\subsection{The propagation operators}

For the bounded data families of Section~\ref{pc:sec:construction},
the operators \eqref{pc:eq:T}--\eqref{pc:eq:K} are well defined.
We recall
\begin{equation}\label{sm:eq:Tdef}
 (\Top\phi)(y)=2\sinh^2y\int_y^\infty
 \frac{\phi(r)}{\sinh^3r}\,dr
\end{equation}
and
\begin{equation}\label{sm:eq:Adef}
 \Aop\phi=\frac12\phi'-\coth y\,\phi
 +\csch y\,\Top\phi.
\end{equation}
The explicit operator $\Kop_p$ in \eqref{pc:eq:K} is the propagation
semigroup generated by $\Aop$:
\begin{equation}\label{sm:eq:semigroup}
 \Kop_{p+r}=\Kop_p\Kop_r,
 \qquad \left.\partial_p\Kop_p\right|_{p=0}=\Aop,
 \qquad \Kop_p^*=\Top\Kop_p.
\end{equation}
These identities follow either by differentiating the displayed kernel or
by uniqueness for the corresponding homogeneous hyperbolic system.
Whenever a negative propagation parameter occurs below, $\Kop_p$ denotes
the local analytic continuation of this kernel near $p=0$; only its Taylor
jet through order two is used.

We shall also use
\begin{align}\label{sm:eq:QP}
 Q(F,R)&=2\coth y\,F-2\csch y\,R,\\
 P(F,R)&=-2\csch y\,F+2\coth y\,R.\notag
\end{align}

\begin{lemma}[Companion identity]\label{sm:lem:companion}
If $R=\Top F$, then
\begin{equation}\label{sm:eq:companionidentity}
 \Top\!\left[F'-\frac12Q(F,R)\right]=P(F,R).
\end{equation}
\end{lemma}

\begin{proof}
Differentiate \eqref{sm:eq:Tdef} to obtain $(\Top F)'=P(F,\Top F)$.
Substitution in the left-hand side of \eqref{sm:eq:companionidentity}, followed
by one integration by parts, gives the same expression.  The boundary term
at infinity is zero because the data are bounded and the weight decays exponentially.
\end{proof}

\subsection{The A/B interface}

Introduce midpoint and signed-normal variables
\begin{equation}\label{sm:eq:ABcoords}
 x=\sqrt2q_1,\quad a=\sqrt2q_3,\quad
 b=\frac{\sqrt2}{2}(q_2+q_4),\quad
 s=a+b,\quad d=\sqrt2(q_4-q_2).
\end{equation}
Thus $\sqrt2q_2=b-d/2$, $\sqrt2q_4=b+d/2$, and $d=0$ is
$\Sigma_{AB}$.

Write
\[
 \ell(s,x)=\widetilde L(s,x),\qquad R(s,x)=\Top\ell(s,x),
 \qquad c=\frac1{\sqrt2}.
\]
The identity that drives the matching is
\begin{equation}\label{sm:eq:diagonalidentity}
 \ell_x+Q(\ell,R)=c.
\end{equation}
Here $\ell_x$ is differentiation in the second argument of $\ell$.
This identity is exact.  Indeed, split $\ell=\ell_0+L$, where
\[
 \ell_0(s)=\frac1{\sqrt2}\arctan(e^{-s})\cosh s.
\]
The elementary companion of $\ell_0$ gives
$Q(\ell_0,\Top\ell_0)=c$.  For the integral part, put $C=\coth s$ and
$h_x(r)=e^{-2xr}g(r)$.  Fubini's theorem gives
\[
 \Top L(s,x)=-\sinh^2s\int_1^C(r-C)^2h_x(r)\,dr,
\]
and hence
\[
 Q(L,\Top L)=2\sinh s\int_1^Cr(r-C)h_x(r)\,dr=-L_x.
\]
This proves \eqref{sm:eq:diagonalidentity}.

For Region $B$, define the corrected pair
\[
 A_0(s)=\frac{1+e^{-2s}/3}{2\sqrt2},\qquad
 B_0(s)=\frac{2e^{-s}}{3\sqrt2},\qquad
 H=h+A_0,\quad S=r_2+B_0.
\]
Direct differentiation gives
\[
 A_0'-Q(A_0,B_0)=-c,\qquad
 B_0'-P(A_0,B_0)=0,\qquad B_0=\Top A_0.
\]
Consequently, the two pairs satisfy
\begin{align}\label{sm:eq:ABsystems}
 -f_x+f_y-2f_z&=Q(f,r),& r_y&=P(f,r),& r&=\Top f,\\
 H_y-2H_z&=Q(H,S)-c,& S_y&=P(H,S),& S&=\Top H.\notag
\end{align}
At $z=0$ their traces agree:
\begin{equation}\label{sm:eq:ABtrace}
 (f,r)=(H,S)=(\ell,R).
\end{equation}

The two analytic continuations along the normal $d$ are
\begin{align*}
 V_A(d)&=\alpha_A(d)f(x,s-d/2,d)
       +\beta_A(d)r(x,s-d/2,d),\\
 V_B(d)&=\alpha_B(d)H(x,s+d/2,-d)
       +\beta_B(d)S(x,s+d/2,-d),
\end{align*}
where
\begin{align*}
 \alpha_A(d)&=\frac{\sinh a}{\sinh(s-d/2)},&
 \beta_A(d)&=\frac{\sinh(b-d/2)}{\sinh(s-d/2)},\\
 \alpha_B(d)&=\frac{\sinh a}{\sinh(s+d/2)},&
 \beta_B(d)&=\frac{\sinh(b+d/2)}{\sinh(s+d/2)}.
\end{align*}
Put $\alpha=\sinh a/\sinh s$ and $\beta=\sinh b/\sinh s$.
Zeroth-order matching follows from \eqref{sm:eq:ABtrace}.

A dot denotes the derivative at $d=0$ along the corresponding displayed
curve.  From \eqref{sm:eq:diagonalidentity} and \eqref{sm:eq:ABsystems},
\begin{equation}\label{sm:eq:ABfirstcomponents}
 \dot f_A=-\frac c2,\qquad
 \dot H_B=-\frac{\ell_x}{2}.
\end{equation}
If $u=f_z=(\ell_y-c)/2$ and
$w=H_z=(\ell_y+\ell_x)/2$, then
$u+w=\ell_y-Q(\ell,R)/2$.  Lemma~\ref{sm:lem:companion} therefore gives
\begin{equation}\label{sm:eq:ABfirstr}
 \dot r_A=\dot S_B=: \rho.
\end{equation}
Using
\[
 \beta\bigl(\coth s-\coth b\bigr)=-\alpha\csch s,
\]
the product rule now yields $V_A'(0)=V_B'(0)$.

For the second derivative, differentiating the systems once more gives
\begin{align}
 \ddot f_A-\ddot H_B
 &=-\coth s\,(\dot f_A+\dot H_B)+2\csch s\,\rho,
 \label{sm:eq:ABsecondf}\\
 \ddot r_A&=\ddot S_B.\label{sm:eq:ABsecondr}
\end{align}
For completeness, \eqref{sm:eq:ABsecondr} follows by applying
Lemma~\ref{sm:lem:companion} to $f_z-H_z$ and differentiating
\eqref{sm:eq:diagonalidentity} in $s$.  Since
$\alpha_B(d)=\alpha_A(-d)$ and $\beta_B(d)=\beta_A(-d)$,
\eqref{sm:eq:ABsecondf}--\eqref{sm:eq:ABsecondr} reduce the remaining difference
to
\[
 V_A''(0)-V_B''(0)
 =2\rho\{\alpha\csch s+
 \beta(\coth s-\coth b)\}=0.
\]
We have proved
\begin{equation}\label{sm:eq:ABcubic}
 V_A(d)-V_B(d)=O(d^3).
\end{equation}
For $s>0$ the estimate is smooth in the tangential variables $(x,a,b)$.
The endpoint $s=0$ is treated separately below.  Therefore, away from that
endpoint,
the values, all first derivatives, and every Hessian entry agree on
$\Sigma_{AB}$.

\subsection{The A/C interface}

Use the linear coordinates
\begin{equation}\label{sm:eq:ACcoords}
 x=\sqrt2q_1,\qquad \eta=\sqrt2q_2,\qquad
 y=\sqrt2(q_2+q_3),\qquad
 a=\sqrt2(q_4-q_2-2q_1).
\end{equation}
Then $a=0$ is $\Sigma_{AC}$, Region $A$ has $a\leq0$, and Region $C$
has $a\geq0$.

Let $U(y)=\widetilde L(y,0)$ and define
\begin{equation}\label{sm:eq:JK}
 J(y)=\int_y^\infty\frac{U(r)}{\sinh^3r}\,dr,
 \qquad
 K(y)=\int_y^\infty\frac{\cosh r\,U(r)}{\sinh^4r}\,dr.
\end{equation}
To calculate the normal derivatives of the Region $A$ datum, introduce
\[
 M_k(y)=\int_1^{\coth y}r^k(r\sinh y-\cosh y)g(r)\,dr.
\]
Differentiation at the moving endpoint gives the exact recurrence
\begin{equation}\label{sm:eq:momentrecurrence}
 \sinh y\,M_{k+1}'-\cosh y\,M_{k+1}
 =\cosh y\,M_k'-\sinh y\,M_k.
\end{equation}
Combining \eqref{sm:eq:momentrecurrence} with
\[
 U''+5\coth y\,U'-6U=-\frac3{\sqrt2}\coth y
\]
and the weighted integral limits at infinity yield
\begin{equation}\label{sm:eq:M12}
 M_1=-\frac16U',\qquad
 M_2=\frac1{36}U''-\frac{U}{6\sinh^2y}
      -\frac{\cosh y}{3}J+\sinh y\,K.
\end{equation}
Since
\[
 \widetilde L(y,n)=U(y)-2nM_1(y)+2n^2M_2(y)+O(n^3),
\]
we obtain
\begin{align}
 \widetilde L(y,0)&=U,\notag\\
 \partial_n\widetilde L(y,0)&=\frac13U',\label{sm:eq:Ljets}\\
 \partial_{nn}\widetilde L(y,0)
 &=\frac19U''-\frac23\csch^2y\,U
   -\frac43\cosh y\,J+4\sinh y\,K.\notag
\end{align}

Now set
\[
 B_a=\Kop_a[\widetilde L(\cdot,-a/2)],
 \qquad P_a(y)=p(0,y,a).
\]
The semigroup Taylor expansion and \eqref{sm:eq:Ljets} give
\begin{align}
 B_0&=U,\notag\\
 B_0'&=\frac13U'-\coth y\,U+2\sinh y\,J,\label{sm:eq:Bjets}\\
 B_0''&=\frac19U''-\frac23\coth y\,U'+U
       +\frac83\cosh y\,J-4\sinh y\,K.\notag
\end{align}

Expanding the Region $C$ boundary formula $P_a$ gives exactly the same
three expressions.  To make this algebra transparent, write $\gamma=a/3$;
the two integrals in that formula satisfy at $\gamma=0$
\[
 I_1(0)=4J,\qquad I_2(0)=2J,
\]
and differentiation of their moving lower endpoints gives the last two
terms in \eqref{sm:eq:Bjets}.  Hence
\begin{equation}\label{sm:eq:PminusB}
 P_a-B_a=O(a^3).
\end{equation}

The propagated first components in the two regions are
\begin{align*}
 F_A&=\Kop_{2x+a}[\widetilde L(\cdot,-a/2)]
     =\Kop_{2x}B_a,\\
 F_C&=\Kop_{2x}P_a.
\end{align*}
Equation \eqref{sm:eq:PminusB} and the semigroup property imply
$F_A-F_C=O(a^3)$, smoothly in $(x,\eta,y)$ for $y>0$.  The companions have the same
contact because $\Kop_p^*=\Top\Kop_p$.  The hyperbolic weights depend only
on $(\eta,y)$, so
\begin{equation}\label{sm:eq:ACcubic}
 v_A-v_C=O(a^3).
\end{equation}
Finally,
\begin{align}\label{sm:eq:derivativemap}
 \partial_{q_1}&=\sqrt2(\partial_x-2\partial_a),&
 \partial_{q_2}&=\sqrt2(\partial_\eta+\partial_y-\partial_a),\notag\\
 \partial_{q_3}&=\sqrt2\partial_y,&
 \partial_{q_4}&=\sqrt2\partial_a.
\end{align}
Thus \eqref{sm:eq:ACcubic} proves equality of every $q$-coordinate Hessian
entry on $\Sigma_{AC}$.  The first possible discrepancy between the two
data is cubic; equality of the data away from the interface is not needed.

There is no third codimension-one calculation.  At the B/C triple junction,
the A/B and A/C two-jets agree, so B/C compatibility follows by transitivity.

\subsection{One-sided regularity on the closed chamber}

We next establish the one-sided endpoint estimates needed at the ordering
faces, including their intersections.  Matching values and gradients alone
would not provide these estimates for the Hessian.

For $m\in\mathbb N_0$, write
$F=P_2+\mathscr R_{3,m}$ at a finite-dimensional cone vertex if $P_2$ is
a polynomial of degree at most two and, with
$\rho=|\xi|_1$ and $R=F-P_2$,
\begin{equation}\label{sm:eq:E3class}
 \left|\partial^\beta R(\xi)\right|
 \leq C_\beta\rho^{3-|\beta|}
       \bigl(1+|\log\rho|\bigr)^m,
 \qquad |\beta|\leq2.
\end{equation}
Every finite $m$ gives a $C^2$ one-sided extension with two-jet $P_2$.
The exponent is retained because one propagation can increase it by one.

\Needspace{13\baselineskip}
\begin{proposition}[Closed-chamber regularity]\label{sm:prop:closedchamber}
The piecewise function $v$ has a $C^2$ one-sided extension to
$\overline{\mathbb R_+^4}$.  Consequently,
\[
 \mathcal U(x_1,\ldots,x_5)=x_5+v(q_1,q_2,q_3,q_4)
\]
belongs to $C^2(\overline{\mathcal W}_5)$.

More precisely, if
\[
 u_0=U(0)=\frac{45\pi^2}{512\sqrt2},
\]
then at the full collision
\begin{equation}\label{sm:eq:fullcollisionjet}
 v(q)=u_0-\frac{q_1+2q_2+3q_3+4q_4}{5}
      +\frac{u_0}{2}q^{\mathsf T}Bq+\mathscr R_{3,2},
\end{equation}
where
\[
 B=\begin{pmatrix}
 \frac43&1&\frac23&\frac13\\
 1&2&\frac43&\frac23\\
 \frac23&\frac43&2&1\\
 \frac13&\frac23&1&\frac43
 \end{pmatrix}.
\]
In particular, $D^2v(q)\to u_0B$ as $q\to0$ inside the closed orthant.
\end{proposition}

\begin{proof}
We divide the endpoint calculation into the data, propagation, and
interpolation steps.

\emph{Step 1: the two data families.}
The profile equation at the origin and the large-$r$ expansion of the
kernel give
\begin{equation}\label{sm:eq:Uandgtails}
 U'(0)=-\frac{3}{5\sqrt2},\qquad U''(0)=U(0)=u_0,
 \qquad
 g(r)=-\frac{1}{14\sqrt2}\,r^{-4}+O(r^{-6}).
\end{equation}
The last estimate, together with its differentiated versions, is obtained by inserting
$\operatorname{arccoth}r=r^{-1}+\frac13r^{-3}+\frac15r^{-5}+\cdots$
in the definition of $g$.

Split the integral defining $\widetilde L$ into a local part and a tail,
subtract the quadratic Taylor polynomial, and use
\eqref{sm:eq:Uandgtails}.  Cutting the tail at
$r\asymp |(s,w)|^{-1}$, the borderline third moment
$r^3g(r)=O(r^{-1})$ contributes at most
$|(s,w)|^3(1+|\log|(s,w)||)$.  Differentiating $k\leq3$ times gives the
corresponding bound with power $3-k$ (third derivatives are taken in
the open parameter cone).  For $r\leq(s+w)^{-1}$ one uses the cubic
Taylor remainder, and on the remaining tail one uses the exponential
bound and $r\leq\coth s$; the differentiated moving-endpoint terms obey
the same estimates.  This gives,
uniformly for $s,w\geq0$ near the origin,
\begin{equation}\label{sm:eq:Lendpointjet}
 \widetilde L(s,w)
 =u_0-\frac{3s+w}{5\sqrt2}
  +u_0\left(\frac{s^2}{2}+\frac{sw}{3}+\frac{w^2}{3}\right)
  +\mathscr R_{3,1}.
\end{equation}
For completeness, the coefficients can be obtained without evaluating any
improper moment: $\widetilde L(s,0)=U(s)$ and
$\partial_w\widetilde L(s,0)=U'(s)/3$.  The recurrence
\eqref{sm:eq:momentrecurrence}, through the $M_2$ identity
\eqref{sm:eq:M12}, gives
\[
 \partial_{ww}\widetilde L(0,0)
 =-4\int_1^\infty r^2g(r)\,dr=\frac{2u_0}{3}.
\]
The moving-endpoint identity
\begin{equation}\label{sm:eq:Lcompatibility}
 \partial_{ss}L(s,w)-L(s,w)
 =-e^{-2w\coth s}g(\coth s)\csch^3s
\end{equation}
shows
$\partial_{ss}\widetilde L(0,w)=\widetilde L(0,w)$.

The Region C datum requires a joint limit.  Put
$P(s,a)=p_0(s,a)$, $\alpha=a/3$, and $\rho=s+\alpha$.  Its displayed
formula is equivalently
\begin{align}\label{sm:eq:Pendpointform}
 P(s,a)
 &=\frac{\sinh^3s}{\sinh^3\rho}U(\rho)
 +12e^{-s}\sinh^2\alpha
   \int_\rho^\infty
   \frac{\sinh^2(r-\alpha)U(r)}{\sinh^5r}\,dr\notag\\
 &\quad
 +6\sinh s\sinh\alpha
   \int_\rho^\infty
   \frac{\sinh(r-\alpha)
   \bigl(\cosh r-e^{2\alpha-r}\bigr)U(r)}{\sinh^5r}\,dr.
\end{align}
Set
\[
 P_2(s,a)=u_0-\frac{3s+4a}{5\sqrt2}
 +u_0\left(\frac{s^2}{2}+\frac{sa}{2}+\frac{a^2}{3}\right),
 \qquad \lambda=\frac\alpha\rho.
\]
The profile equation gives, for $0\leq k\leq2$,
\[
 \left|\partial_r^k\left(U(r)-u_0+\frac{3r}{5\sqrt2}
 -\frac{u_0r^2}{2}\right)\right|\leq Cr^{3-k}.
\]
Here is an explicit expansion of the two integrals, including the
finite part that contributes to the quadratic terms.  In addition to
$J$ and $K$ in \eqref{sm:eq:JK}, put
\[
 N_5(\rho)=\int_\rho^\infty\frac{U(r)}{\sinh^5r}\,dr,
 \qquad u_1=-\frac3{5\sqrt2}.
\]
Their Laurent expansions are
\begin{align}
 J(\rho)&=\frac{u_0}{2\rho^2}+\frac{u_1}{\rho}
                         +\frac{5u_0}{12}+O(\rho),\notag\\
 K(\rho)&=\frac{u_0}{3\rho^3}+\frac{u_1}{2\rho^2}
                         +\frac{u_0}{3\rho}+O(1+|\log\rho|),\notag\\
 N_5(\rho)&=\frac{u_0}{4\rho^4}+\frac{u_1}{3\rho^3}
                         -\frac{u_0}{6\rho^2}+O(\rho^{-1}).
 \label{sm:eq:momentLaurent}
\end{align}
These expansions may be differentiated three times, reducing the
remainder power by the number of derivatives.  The singular terms follow
by expanding the integrands at zero.  The finite part $5u_0/12$ in $J$
follows from \eqref{sm:eq:TUidentity}, which is an identity on $s>0$
and uses only the profile equation and the condition at infinity.

Writing $s_\alpha=\sinh\alpha$ and $c_\alpha=\cosh\alpha$, the two
integrals in \eqref{sm:eq:Pendpointform}, without their exterior
prefactors, are exactly
\begin{align*}
 I_1&=\cosh(2\alpha)J-\sinh(2\alpha)K+s_\alpha^2N_5,\\
 I_2&=(e^{2\alpha}c_\alpha+2e^\alpha s_\alpha^2)J
       -s_\alpha(e^{2\alpha}+2e^\alpha c_\alpha)K
       +2e^\alpha s_\alpha^2N_5,
\end{align*}
where the three moments are evaluated at $\rho$.  Substitute
\eqref{sm:eq:momentLaurent}, $s=(1-\lambda)\rho$ and
$\alpha=\lambda\rho$.  The coefficients of $1,\rho,\rho^2$ become,
respectively,
\[
 u_0,\qquad -\frac{3(1+3\lambda)}{5\sqrt2},\qquad
 \frac{u_0}{2}(1+\lambda+4\lambda^2).
\]
They are precisely the coefficients of $P_2((1-\lambda)\rho,3\lambda\rho)$.
All the coefficients multiplying $J,K,N_5$ vanish to orders at least
$2,3,4$, respectively.  Their omitted terms therefore give
\begin{equation}\label{sm:eq:Prescaledremainder}
 \left|\partial_\rho^i\partial_\lambda^j
  \{P((1-\lambda)\rho,3\lambda\rho)
             -P_2((1-\lambda)\rho,3\lambda\rho)\}\right|
 \leq C\rho^{3-i}(1+|\log\rho|),\qquad i+j\leq3,
\end{equation}
uniformly for $0\leq\lambda\leq1$.  Since
\[
 \partial_s=\partial_\rho-\frac\lambda\rho\partial_\lambda,
 \qquad
 \partial_a=\frac13\partial_\rho+
 \frac{1-\lambda}{3\rho}\partial_\lambda,
\]
\eqref{sm:eq:Prescaledremainder} implies
\[
 |\partial_s^i\partial_a^j(P-P_2)|
\leq C\rho^{3-i-j}(1+|\log\rho|),\qquad i+j\leq3,
\]
Thus, jointly at $(0,0)$,
\begin{equation}\label{sm:eq:Pendpointjet}
 P(s,a)
 =u_0-\frac{3s+4a}{5\sqrt2}
  +u_0\left(\frac{s^2}{2}+\frac{sa}{2}+\frac{a^2}{3}\right)
  +\mathscr R_{3,1}.
\end{equation}
The same coefficients are independently identified by the boundary
system.  Indeed,
\begin{equation}\label{sm:eq:TUidentity}
 \Top U=\cosh s\,U+\frac{\sinh s}{6}U'
       -\frac{\sinh s}{2\sqrt2},
\end{equation}
which follows from the profile ODE, boundedness of $U$, and the exponential integration weights.  At $a=0$, the first boundary
equation gives
\[
 P_a(s,0)=\frac12U'(s)-\frac1{2\sqrt2}.
\]
The regular limit of the same equation at $s=0$ is
$4P_s(0,a)=3P_a(0,a)$.  Together with $P(s,0)=U(s)$, these identities give
\[
 P_s(0,0)=-\frac3{5\sqrt2},\quad
 P_a(0,0)=-\frac4{5\sqrt2},\quad
 P_{ss}(0,0)=u_0,\quad
 P_{sa}(0,0)=\frac{u_0}{2},\quad
 P_{aa}(0,0)=\frac{2u_0}{3}.
\]
Here $P_{sa}$ follows by differentiating the first boundary identity at
$s=0$, and $P_{aa}$ by differentiating the second at $a=0$; these
operations are licensed by the preceding $\mathscr R_{3,1}$ estimate.

\emph{Step 2: propagation and the companion.}
The only possible loss of two derivatives in $\Top$ is explicit.  If
$\phi(s)=a_0+a_1s+a_2s^2/2+O(s^3)$, then
\begin{equation}\label{sm:eq:Tlogterm}
 \Top\phi(s)
 =a_0+2a_1s+O(s^2)
  -(a_2-a_0)s^2\log s.
\end{equation}
Thus the logarithmic Hessian term vanishes precisely when
$\phi''(0)=\phi(0)$.

The same cancellation controls $\Kop$.  If $d=p/2$ and $t=y+d$, its
kernel can be written exactly as
\begin{align}\label{sm:eq:Kendpointform}
 \Kop_{2d}[\phi](y)
 &=\frac{\sinh^2y}{\sinh^2t}\phi(t)
 +4\sinh y\sinh d\,\mathfrak J_\phi(t)\notag\\
 &\quad+6\sinh^2d
 \left(\cosh t\,\mathfrak J_\phi(t)
       -\sinh t\,\mathfrak M_\phi(t)\right),
\end{align}
where
\[
 \mathfrak J_\phi(t)=\int_t^\infty\frac{\phi(r)}{\sinh^3r}\,dr,
 \qquad
 \mathfrak M_\phi(t)=\int_t^\infty
       \frac{\cosh r\,\phi(r)}{\sinh^4r}\,dr.
\]
\medskip
\noindent\textbf{Endpoint propagation estimate for the explicit data.}
We use the estimate only for
\[
 \phi(s,\vartheta)=\widetilde L(s,\vartheta),\qquad
 \widetilde L(s,\vartheta)-\frac{1+e^{-2s}/3}{2\sqrt2},
 \qquad P(s,\vartheta).
\]
The preceding local calculations give
$\phi=P_2+\mathscr R_{3,m}$ with $m\leq1$; moreover their remainders
satisfy the same weighted bound through derivative order three in the
open parameter cone.
The tail hypotheses needed in addition to the local expansion are,
for every fixed small $\vartheta_0>0$ and $i+j\leq4$,
\begin{equation}\label{sm:eq:tailbounds}
 \sup_{s\geq1,\ 0\leq\vartheta\leq\vartheta_0}
 |\partial_s^i\partial_\vartheta^j\phi(s,\vartheta)|<\infty.
\end{equation}
For $\widetilde L$, formula \eqref{pc:eq:L} restricts its inner
variable to $[1,\coth1]$; differentiating its exponential inserts
only bounded powers of that variable, and the moving-endpoint terms
have the same bound.  For $P$, substitute $r=s+\vartheta/3+t$
in \eqref{pc:eq:P}: the denominators stay separated from zero,
and each required differentiated integrand has an integrable
exponential bound in $t$.  The required derivatives of $U$ are
bounded by \eqref{pc:eq:Uode}--\eqref{pc:eq:Uintegral}.
The elementary subtraction preserves these estimates.
The exact radial compatibility is
\[
 \phi_{ss}(0,\vartheta)=\phi(0,\vartheta)
\]
for every nearby $\vartheta$.  Then, jointly in $(y,p,\vartheta)$,
\[
 \Kop_p[\phi_\vartheta](y)=Q_2(y,p,\vartheta)+\mathscr R_{3,m+1},
 \qquad
 \Kop_p^*[\phi_\vartheta](y)=Q_2^*(y,p,\vartheta)+\mathscr R_{3,m+1}
\]
for quadratic polynomials $Q_2,Q_2^*$.  Moreover the propagated first
component satisfies
$\partial_{yy}\Kop_p[\phi](0)=\Kop_p[\phi](0)$.
The propagated remainders also satisfy the third-order bound
\begin{equation}\label{sm:eq:thirdremainder}
 |\partial^\beta\mathscr R_{3,m+1}|
       \leq C(1+|\log\delta|)^{m+1},\qquad |\beta|=3,
 \quad \delta=y+p/2+|\vartheta|_1>0.
\end{equation}

\emph{Proof.}
Use \eqref{sm:eq:Kendpointform}, split $\mathfrak J_\phi$ and
$\mathfrak M_\phi$ at $r=1$, and subtract the local Taylor terms
of the explicit data.  Bound \eqref{sm:eq:tailbounds} licenses
differentiation of the far parts, which contribute ordinary
Taylor coefficients and third-order remainders.  In the local parts set $r=\delta z$, where
$\delta=y+d+|\vartheta|_1$.  For the full local combination in
\eqref{sm:eq:Kendpointform}, including its hyperbolic prefactors,
differentiation of total order $k\leq3$ gives
\[
 C\delta^{3-k}(1+|\log\delta|)^{m+1}.
\]
For uniformity when $y+d$ is smaller than $|\vartheta|_1$, first
subtract the radial Taylor terms at $r=0$ with the parameter fixed,
and split the local integral also at $r=|\vartheta|_1$.  The third-order
bound for the explicit data controls this subtraction on the lower
interval; on the upper interval $r+|\vartheta|_1\leq2r$ gives the same
weighted estimate directly.  This proves the estimate also when the
ratios of the small parameters tend to zero.
The extra logarithm can occur only in $\mathfrak M_\phi$, whose singular
weight is $r^{-4}$; $\mathfrak J_\phi$ has weight $r^{-3}$ and does not
increase the logarithmic power.  For the quadratic part, the only
potential logarithmic term of order two is proportional to
$\phi_{ss}(0,\vartheta)-\phi(0,\vartheta)$, with a homogeneous
quadratic prefactor in $(y,d)$ multiplying $\log t$.
It vanishes by compatibility.  Direct differentiation of the same kernel gives
$(\Kop_p\phi)_{yy}(0)=\Kop_p\phi(0)$.  Finally apply \eqref{sm:eq:Tdef}.
Its only order-two logarithm is
$-(\psi_{yy}(0)-\psi(0))y^2\log y$, hence it vanishes for the propagated
first component; its $r^{-3}$ weight does not further increase $m$.
This proves the asserted estimate for the three displayed families.
In particular these estimates control the differentiated remainders;
they are stronger than a Taylor expansion of function values alone.

The compatibility required here is uniform in every nearby data parameter.
It holds for $\widetilde L$ for every $w\geq0$ by
\eqref{sm:eq:Lcompatibility}, for
$\Psi=\widetilde L-(1+e^{-2s}/3)/(2\sqrt2)$ by direct differentiation,
and for $p_0$ because its reduced ODE reads
\[
 P_{ss}-P=\frac{\sinh^3s}{\sinh^3(s+a/3)}
          (U''-U)(s+a/3),
\]
For each $a>0$ its right-hand side vanishes at $s=0$ because of the
$\sinh^3s$ factor; at $a=0$ the continuous limit is
$U''(0)-U(0)=0$ by \eqref{sm:eq:Uandgtails}.
Thus $\widetilde L$ and $p_0$ start in
$P_2+\mathscr R_{3,1}$ (the elementary subtraction defining $\Psi$ is
analytic), and one propagation produces at worst a
$\mathscr R_{3,2}$ remainder.

\Needspace{6\baselineskip}
\emph{Step 3: the collapsing hyperbolic weights.}
Let a propagated pair satisfy
\begin{align*}
 F_y-DF&=Q(F,R)+\gamma,\\
 R_y&=P(F,R),\qquad F(0)=R(0)=A,
\end{align*}
where $D$ differentiates only tangential parameters and $\gamma$ is
independent of $y$.  Its regular Frobenius jet is
\begin{equation}\label{sm:eq:regularpairjet}
 F_y(0)=\frac{DA+\gamma}{3}=:f_1,\qquad
 R_y(0)=2f_1,\qquad
 F_{yy}(0)=A,\qquad R_{yy}(0)=A+Df_1.
\end{equation}
This follows by inserting Taylor expansions in the two equations and
equating the coefficients of $y^{-1},1,y$.

Using \eqref{sm:eq:regularpairjet} in
\begin{equation}\label{sm:eq:collapsedinterpolation}
 \frac{\sinh b\,F(a+b,\theta-a)
       +\sinh a\,R(a+b,\theta-a)}{\sinh(a+b)}
\end{equation}
shows that its polynomial terms through degree three are divisible by
$a+b$.  The numerator remainder has an additional factor of order
$a+b$ from its hyperbolic weights.  When differentiating the quotient,
use $R(0,\theta)=F(0,\theta)$ to write
\[
 R(y,\theta)-F(y,\theta)
 =y\int_0^1(R_y-F_y)(ty,\theta)\,dt.
\]
The regular pair identities give the polynomial part of this expression;
the remainder estimates through order three in
\eqref{sm:eq:thirdremainder} control its derivatives through order two.
Thus division preserves the $\mathscr R_{3,m}$ estimate, uniformly
also when $a+b$ is smaller than a tangential parameter.  Hence
\eqref{sm:eq:collapsedinterpolation} has a $C^2$ extension at $a=b=0$.
If some tangential gap remains positive, the functions are analytic within
each formula region and the same conclusion is simply Hadamard division:
the numerator vanishes when $a+b=0$ because $F(0)=R(0)$.  At a point that
also lies on an internal interface, the regional two-jets are identified
by the interface calculations above.

\medskip
\noindent\emph{Identification of the collapsing variables.}
For the face $q_2=q_3=0$ (Regions $A$ and $C$), put
\[
 a=\sqrt2q_2,\qquad b=\sqrt2q_3,\qquad
 x=\sqrt2q_1,\qquad \theta=\sqrt2q_4,
\]
so that $y=a+b$ and $z=\theta-a$.  In Region $A$ take
$(F,R)=(f,r_1)$, and in Region $C$ take $(F,R)=(p,r_3)$.  In either case
\[
 v=\frac{\sinh b\,F(x,a+b,\theta-a)
              +\sinh a\,R(x,a+b,\theta-a)}{\sinh(a+b)}.
\]
The governing system is
\[
 F_y-(\partial_x+2\partial_z)F=Q(F,R),\qquad R_y=P(F,R),
\]
so in \eqref{sm:eq:regularpairjet} one has
$D=\partial_x+2\partial_z$ and $\gamma=0$.

For the face $q_3=q_4=0$ (Region $B$), put
\[
 a=\sqrt2q_4,\qquad b=\sqrt2q_3,\qquad
 x=\sqrt2q_1,\qquad \theta=\sqrt2q_2,
\]
again giving $y=a+b$ and $z=\theta-a$.  With
\[
 A_0(y)=\frac{1+e^{-2y}/3}{2\sqrt2},\qquad
 B_0(y)=\frac{2e^{-y}}{3\sqrt2},\qquad
 (F,R)=(h+A_0,r_2+B_0),
\]
one has
\[
 v=\frac{\sinh b\,F(x,a+b,\theta-a)
              +\sinh a\,R(x,a+b,\theta-a)}{\sinh(a+b)}.
\]
Here
\[
 F_y-2F_z=Q(F,R)-\frac1{\sqrt2},\qquad R_y=P(F,R),
\]
and hence $D=2\partial_z$ and $\gamma=-1/\sqrt2$.  These are precisely
the two instances of \eqref{sm:eq:collapsedinterpolation}.

\medskip
\noindent\emph{The non-full endpoint of the $A/B$ interface.}
It remains to justify the endpoint
$q_2=q_3=q_4=0$ with $q_1>0$, where the estimate
\eqref{sm:eq:ABcubic} is not automatically uniform in the variables used in
its interior proof.  Set
\[
 X=\sqrt2q_1,
 \quad \alpha=\sqrt2q_2,
 \quad \beta=\sqrt2q_3,
 \quad \gamma=\sqrt2q_4,
 \quad m=\frac{\alpha+\gamma}{2},
 \quad \delta=\frac{\gamma-\alpha}{2},
 \quad t=\beta+m,
\]
and write $H_\lambda(r)=\cosh\lambda-r\sinh\lambda$ and
$c=1/\sqrt2$.  Elementary hyperbolic identities reduce the part not
containing $L$ in both regions to the single expression
\[
 c\arctan(e^{-t})\cosh t\cosh\delta\cosh m
 +c\operatorname{arctanh}(e^{-t})\sinh t\sinh\delta\sinh m
 -\frac c2\sinh\gamma .
\]
Its only non-analytic term is
$O(t^3(1+|\log t|))$, because $|\delta|\leq m\leq t$, and its
derivatives through order two tend to zero.

Fubini's theorem applied to the $L$-terms gives exactly
\begin{align*}
 V_{L,A}
 &=-\int_1^{\coth t}H_t(r)H_\delta(r)H_m(r)
       e^{-2(X-\delta)r}g(r)\,dr,
 &&\delta\geq0,\\
 V_{L,B}
 &=-\int_1^{\coth t}H_t(r)H_{-\delta}(r)H_m(r)
       e^{-2Xr}g(r)\,dr,
 &&\delta\leq0.
\end{align*}
When $m,\beta,X$ are held fixed, the only distinct normal factors are
\[
 e^{2\delta r}H_\delta(r)\quad\hbox{and}\quad H_{-\delta}(r),
\]
and their value, first derivative, and second derivative at $\delta=0$
are respectively
\[
 (1,r,1)\quad\hbox{and}\quad(1,r,1).
\]
Thus the two integral formulas have the same normal two-jet.  Tangential
and mixed derivatives are also legitimate: two differentiations insert
at most two powers of $r$, while
$g(r)=-(14\sqrt2)^{-1}r^{-4}+O(r^{-6})$, so $r^2g(r)$ is integrable.
Moreover $H_t(\coth t)=0$ kills the first moving-endpoint term, and the
only second-order endpoint contribution is $O(t)$ as $t\downarrow0$.
Consequently the complete $A$ and $B$ two-jets agree at
$q_2=q_3=q_4=0$, for every fixed $q_1>0$.

\medskip
\noindent\emph{The non-full endpoint of the $A/C$ interface.}
Now fix $x=\sqrt2q_1>0$ and consider
$q_2=q_3=0$, $q_4=2q_1$.  Use
\[
 \eta=\sqrt2q_2,\qquad y=\sqrt2(q_2+q_3),\qquad
 \sigma=\sqrt2(q_4-q_2-2q_1).
\]
Thus $\sigma=0$ is the interface.  With
\[
 B_\sigma=\Kop_\sigma[\widetilde L(\cdot,-\sigma/2)],
 \qquad P_\sigma=p_0(\cdot,\sigma),
\]
the data calculation in the preceding section gives
\[
 P_\sigma-B_\sigma=\sigma^3E_0,
\]
where $E_0$ is $C^2$, uniformly on compact subsets of the positive data
variable.  Equivalently, if $\tau$ denotes tangential variables, then
\[
 \left|\partial_\sigma^i\partial_\tau^\beta
 (P_\sigma-B_\sigma)\right|
 \leq C|\sigma|^{3-i},\qquad i+|\beta|\leq2.
\]
After propagation the two first components are
\[
 F_A=\Kop_{2x}B_\sigma,\qquad F_C=\Kop_{2x}P_\sigma.
\]
Since $x>0$, the kernel $\Kop_{2x}$ samples its datum only at arguments
$r\geq x+y\geq x$.  Hence the $O(\sigma^3)$ data contact is uniform as
$y\downarrow0$, and it is preserved by both $\Kop_{2x}$ and its companion
$\Kop_{2x}^*=\Top\Kop_{2x}$.

Writing the propagated differences as
$F_A-F_C=\sigma^3\widehat F$ and
$R_A-R_C=\sigma^3\widehat R$, one has
$\widehat F(0)=\widehat R(0)$.  Therefore
\[
 v_A-v_C
 =\sigma^3
 \frac{\sinh(y-\eta)\widehat F(y,\eta,\sigma)
       +\sinh\eta\widehat R(y,\eta,\sigma)}{\sinh y}
\]
extends smoothly to $y=\eta=0$ by Hadamard division.  In particular,
$v_A-v_C=\sigma^3\widetilde E$ with $\widetilde E\in C^2$ there, and
\[
 \partial_\sigma^i\partial_\tau^\beta(v_A-v_C)
 =O(|\sigma|^{3-i}),\qquad i+|\beta|\leq2.
\]
Thus the $A/C$ two-jets agree at this endpoint.  The remaining case $x=0$
is the full collision and is covered by \eqref{sm:eq:fullcollisionjet}.

For clarity, we record the substitution which produces the matrix in
\eqref{sm:eq:fullcollisionjet}.  Use Region $C$ and set
\[
 X=\sqrt2q_1,\qquad A=\sqrt2q_2,\qquad C=\sqrt2q_3,
 \qquad D_4=\sqrt2q_4,\qquad y=A+C,\qquad z=D_4-A.
\]
Write $p=p(x,y,z)$ and $r=\Top p$.  From \eqref{sm:eq:Pendpointjet},
\[
 p_x(0,y,z)=\frac13P_y(y,z),\qquad
 p_x=\Kop_{2x}\left[\frac13P_y(\cdot,z-2x)\right],
\]
and the regular pair identities, one obtains at $(x,y,z)=(0,0,0)$
\begin{align*}
 \nabla p&=-\frac1{5\sqrt2}(1,3,4),&
 D^2p&=u_0
 \begin{pmatrix}
  \frac23&\frac13&\frac16\\
  \frac13&1&\frac12\\
  \frac16&\frac12&\frac23
 \end{pmatrix},\\[2mm]
 \nabla r&=-\frac1{5\sqrt2}(1,6,4),&
 D^2r&=u_0
 \begin{pmatrix}
  \frac23&\frac23&\frac16\\
  \frac23&\frac73&1\\
  \frac16&1&\frac23
 \end{pmatrix}.
\end{align*}
For example, the only entry in the first matrix not obtained immediately
by differentiating $p_x=P_y/3$ is
\[
 p_{xx}(0)=\frac23\bigl(\Aop P_y(0)-P_{yz}(0)\bigr)
 =\frac23\left(\frac32u_0-\frac12u_0\right)=\frac23u_0,
\]
where $\Aop\phi(0)=3\phi_y(0)/2$ for a regular companion pair.  Since
$r=p$ at $y=0$, their tangential derivatives agree; the remaining entries
follow from \eqref{sm:eq:regularpairjet}.

Expanding the Region $C$ interpolation now gives, uniformly with a
$\mathscr R_{3,2}$ remainder,
\begin{align*}
 v={}&p(X,0,z)+(C+2A)p_y(X,0,z)\\
 &+\frac{A+C}{2}\{C p_{yy}(0)+A r_{yy}(0)\}
   -\frac{u_0}{2}AC+\mathscr R_{3,2}.
\end{align*}
The last displayed quadratic term uses
$(\sinh C+\sinh A)/\sinh(A+C)=1-AC/2+O((A+C)^4)$.
Substitution of the two derivative matrices yields
\begin{align*}
 v={}&u_0-\frac{X+2A+3C+4D_4}{5\sqrt2}\\
 &+u_0\bigg\{
 \frac{X^2}{3}+\frac{A^2}{2}+\frac{C^2}{2}+\frac{D_4^2}{3}
 +\frac{XA}{2}+\frac{XC}{3}+\frac{XD_4}{6}\\
 &\hspace{35mm}
 +\frac{2AC}{3}+\frac{AD_4}{3}+\frac{CD_4}{2}\bigg\}
 +\mathscr R_{3,2}.
\end{align*}
Replacing $(X,A,C,D_4)$ by $\sqrt2(q_1,q_2,q_3,q_4)$ gives precisely
\eqref{sm:eq:fullcollisionjet} and the derivative bound
\eqref{sm:eq:E3class} with $m=2$.

The same quadratic form is obtained directly from the Region $A$ kernel.
Indeed, put
\[
 \alpha=\frac{q_2+2q_3+q_4}{\sqrt2},\quad
 \delta=\frac{q_4-q_2}{\sqrt2},\quad
 \mu=\frac{q_4+q_2}{\sqrt2},\quad
 \lambda=\sqrt2q_1-\delta.
\]
The three convergent moments needed for its two-jet are
\[
 \int_1^\infty g(r)\,dr=\frac{\pi}{4\sqrt2}-u_0,
 \quad \int_1^\infty rg(r)\,dr=-\frac1{10\sqrt2},
 \quad \int_1^\infty r^2g(r)\,dr=-\frac{u_0}{6}.
\]
They follow by setting $s=w=0$ in \eqref{sm:eq:Lendpointjet} and its
first two $w$-derivatives.  Expansion of
$-e^{-2\lambda r}H_\alpha(r)H_\delta(r)H_\mu(r)$, together with the
elementary four-expert term, gives the quadratic part
\[
 \frac{u_0}{2}(\alpha^2+\delta^2+\mu^2)
 +\frac{u_0}{6}\{\alpha\delta+\alpha\mu+\delta\mu
                 +2\lambda(\alpha+\delta+\mu)+2\lambda^2\},
\]
which is $u_0q^{\mathsf T}Bq/2$.
For Region $B$, replace $\delta$ by $-\delta$ and $\lambda$ by
$\sqrt2q_1$ in its integral term; the resulting polynomial is identical.
The remainder estimate follows by splitting the integral at the reciprocal
of $|q|_1$: the third moment has logarithmic growth, and all second-order
moving-endpoint terms are $O(|q|_1)$.  Thus the direct $A/B$ calculation
also gives the same full-collision Hessian.

Away from these collapsing sets, ordinary differentiation under the
integrals is valid.  Finally, \eqref{sm:eq:ABcubic} and \eqref{sm:eq:ACcubic}
identify the one-sided two-jets on the two internal interfaces; at their
intersection the agreement follows by transitivity.  Therefore all
derivatives through order two extend continuously to the entire closed
orthant, proving the proposition.
\end{proof}

\subsection{Ordering walls}

\begin{lemma}[Symmetric $C^2$ extension]\label{sm:lem:ordering}
Let
\[
 \overline{\mathcal W}_5=\{x_1\leq\cdots\leq x_5\}
\]
and suppose $\mathcal U\in C^2(\overline{\mathcal W}_5)$, with one-sided derivatives
through order two continuous at every corner.  If
\begin{equation}\label{sm:eq:smoothfitx}
 (\partial_i-\partial_{i+1})\mathcal U=0
 \quad\hbox{on }\{x_i=x_{i+1}\},\qquad i=1,\ldots,4,
\end{equation}
then
\[
 u(x)=\mathcal U(x^{(1)},\ldots,x^{(5)})
\]
belongs to $C^2(\R^5)$.
\end{lemma}

\begin{proof}
Across the wall $x_i=x_{i+1}$, the two formulas are $\mathcal U$ and
$\mathcal U\circ s_i$,
where $s_i$ swaps the two coordinates.  Condition \eqref{sm:eq:smoothfitx}
is exactly gradient matching.  Put
$F_i=\mathcal U_i-\mathcal U_{i+1}$.  Since $F_i=0$ on
the wall, tangential differentiation in $e_k$, $k\neq i,i+1$, gives
$\mathcal U_{ik}=\mathcal U_{i+1,k}$.  Differentiation along the common wall parameter
$e_i+e_{i+1}$ gives
\[
 \mathcal U_{ii}+\mathcal U_{i,i+1}
 =\mathcal U_{i+1,i}+\mathcal U_{i+1,i+1}.
\]
Schwarz symmetry cancels the middle terms, hence
$\mathcal U_{ii}=\mathcal U_{i+1,i+1}$.  Therefore
\[
 D^2\mathcal U=s_iD^2\mathcal U s_i
\]
on the wall, which is precisely Hessian matching.

At a multiple collision, the adjacent swaps inside each equal block generate
the full permutation group of that block.  The gradient and Hessian two-jet
is therefore independent of the order chosen inside the block.  No new
codimension-two or higher matching condition is required.
\end{proof}

\begin{remark}
The hypothesis $\mathcal U\in C^2(\overline{\mathcal W}_5)$ is crucial in the
abstract lemma; Proposition~\ref{sm:prop:closedchamber} verifies it for the
present value function.  If one assumes only
\[
 \mathcal U\in C^2(\mathcal W_5)
 =C^2\!\left(\{x_1<\cdots<x_5\}\right),
\]
then the first-order compatibility conditions
\eqref{sm:eq:smoothfitx} do not imply continuity of the second derivatives
up to, or across, the ordering walls.  The proof above uses the existence
and continuity of the one-sided second-order traces on the closed chamber
when it differentiates the smooth-fit identity tangentially along a wall.
\end{remark}

For $\mathcal U=x_5+v(q)$, the chamber derivatives are
\[
 \mathcal U_1=-v_1,\quad \mathcal U_2=v_1-v_2,\quad
 \mathcal U_3=v_2-v_3,\quad
 \mathcal U_4=v_3-v_4,\quad \mathcal U_5=1+v_4.
\]
Thus \eqref{sm:eq:smoothfitx} is equivalent to the four identities
\begin{align}
 2v_1-v_2&=0 &&\text{on }q_1=0,\label{sm:eq:wall1}\\
 v_1-2v_2+v_3&=0 &&\text{on }q_2=0,\label{sm:eq:wall2}\\
 v_2-2v_3+v_4&=0 &&\text{on }q_3=0,\label{sm:eq:wall3}\\
 v_3-2v_4-1&=0 &&\text{on }q_4=0.\label{sm:eq:wall4}
\end{align}
The identities \eqref{sm:eq:wall2}--\eqref{sm:eq:wall4} are exactly
the corresponding reflected boundary equations
\eqref{pc:eq:reflections}, after multiplication by two and adjustment
of signs.  They first hold on relative interiors of the faces and
extend to their intersections by
Proposition~\ref{sm:prop:closedchamber}.  The first identity requires a
separate verification; it is not a consequence of the other three at a
generic point of $q_1=0$.

Here is a direct verification of \eqref{sm:eq:wall1}.  In Region $B$, the
kernel representation is
\[
 H(x,y,z)=\Kop_z[\widetilde L(\cdot,x)-A_0](y)+A_0(y),
 \qquad S=\Top(H-A_0)+B_0.
\]
At $x=0$, \eqref{sm:eq:diagonalidentity}, \eqref{sm:eq:Ljets}, and the elementary
pair $(A_0,B_0)$ give
\[
 \Aop[\widetilde L(\cdot,0)-A_0]
 =2\,\partial_x\widetilde L(\cdot,0).
\]
Consequently $H_z=2H_x$ and $S_z=2S_x$ on $x=0$, and the Region $B$
formula gives \eqref{sm:eq:wall1}, since there
$\partial_{q_1}=\sqrt2\partial_x$ and
$\partial_{q_2}=\sqrt2\partial_z$.

For Region $C$, write at $x=0$
\[
 P(y,z)=p(0,y,z),\qquad R(y,z)=\Top P(y,z).
\]
The boundary hyperbolic system is
\begin{align*}
 P_y-3P_z&=3\coth y\,P-3\csch y\,R,\\
 R_y&=-2\csch y\,P+2\coth y\,R.
\end{align*}
Because $p(x,y,z)=\Kop_{2x}P(\cdot,z-2x)(y)$,
\[
 p_x(0,y,z)=\frac13P_y,\qquad
 r_x(0,y,z)=\frac13\Top(P_y).
\]
The elementary operator identity
\begin{equation}\label{sm:eq:Tderivative}
 \Top(\phi')=(\Top\phi)'
 +\Top\bigl(\coth y\,\phi-\csch y\,\Top\phi\bigr)
\end{equation}
and the boundary system imply
\[
 R_y-R_z=\frac23\Top(P_y).
\]
\Needspace{9\baselineskip}
If
\[
 \lambda=\frac{\sinh(y-\eta)}{\sinh y},\qquad
 \mu=\frac{\sinh\eta}{\sinh y},
\]
the Region $C$ boundary value is $\lambda P+\mu R$.  A direct product-rule
calculation, using
\[
 \partial_{q_1}=\sqrt2\partial_x,\qquad
 \partial_{q_2}=\sqrt2(\partial_\eta+\partial_y-\partial_z)
 \quad\text{on }q_1=0,
\]
now gives
\[
 2v_1-v_2=0.
\]
Thus all four ordering-wall identities hold.

\subsection{Conclusion}

\begin{proof}[Proof of Theorem~\ref{sm:thm:globalC2}]
The $A/B$ and $A/C$ calculations give cubic contact in the respective
normal variables, so the values, gradients, and Hessians of the regional
formulas agree on both codimension-one interfaces.  The closures of
Regions $B$ and $C$ meet only at their common junction with Region $A$;
there the same conclusion follows by transitivity.

Proposition~\ref{sm:prop:closedchamber} proves that all derivatives through order two have continuous one-sided
limits on every collision stratum, including the full collision.  Hence
$v\in C^2(\overline{\mathbb R_+^4})$ and
$\mathcal U\in C^2(\overline{\mathcal W}_5)$.

Finally, the four identities \eqref{sm:eq:wall1}--\eqref{sm:eq:wall4} are
exactly the adjacent-wall smooth-fit conditions for $\mathcal U$.
Lemma~\ref{sm:lem:ordering} therefore gives $u\in C^2(\mathbb R^5)$.
\end{proof}

\section{Hyperbolic comparison through the companion kernel}
\label{paper:sec:comparison}
The verification uses the positivity of the companion kernel.
We first treat a known nonnegative companion, then give a single
comparison argument for both coupled propagation systems.
\begin{lemma}[Comparison with a nonnegative companion]
\label{paper:lem:onecomponent}
Let $m>0$.  Suppose $F,R$ are continuous for $y,z\geq0$ and are
$C^1$ where $y>0$ and $z>0$.  Assume $R\geq0$ and
\[
(\partial_y-m\partial_z)F
\leq m\coth y\,F-m\csch y\,R,
\qquad y>0,\ z>0.
\]
If $F(y,0)\geq0$ for all $y\geq0$, then $F\geq0$ throughout
the quadrant.  No separate sign assumption at $y=0$ is needed.
\end{lemma}
\begin{proof}
Fix $y>0,z\geq0$ and put $T=z/m$,
$f(t)=F(y+t,z-mt)$ for $0\leq t\leq T$.
Multiplication by the positive integrating factor gives
\[
\frac{d}{dt}\left(\frac{f(t)}{\sinh^m(y+t)}\right)
\leq-\frac{mR(y+t,z-mt)}{\sinh^{m+1}(y+t)}.
\]
Integrating from $0$ to $T$ yields the quantitative lower bound
\begin{equation}\label{paper:eq:onecomponentbound}
\frac{F(y,z)}{\sinh^m y}
\geq\frac{F(y+z/m,0)}{\sinh^m(y+z/m)}
+m\int_0^{z/m}
\frac{R(y+t,z-mt)}{\sinh^{m+1}(y+t)}\,dt\geq0.
\end{equation}
Continuity gives the conclusion also at $y=0$.
\end{proof}
\begin{remark}\label{lem41rmk}There is a complementary implication.  If $F\geq0$ and
\[
R_y-2\coth y\,R+2\csch y\,F\leq0,
\qquad
\liminf_{b\to\infty}\frac{R(b,z)}{\sinh^2b}\geq0,
\]
then integration after multiplication by $\sinh^{-2}y$ gives
\[
R(y,z)\geq2\sinh^2y\int_y^\infty
\frac{F(r,z)}{\sinh^3r}\,dr\geq0,
\]
whenever the displayed integral converges.
Thus one component inequality suffices if the other component's sign
has already been obtained.
\end{remark}
The specializations $m=2$ and $m=3$
are exactly the characteristic slopes in the two systems of the
construction. The next lemma obtains both signs from a lower bound
by the positive companion operator, without prescribing a sign at $y=0$.
The same argument is compatible with averaging controls.  If two
controls have the same first trace and opposite errors in its
hyperbolic equation, their averaged equation has zero error.  A
nonnegative average companion then suffices in
Lemma~\ref{paper:lem:onecomponent}.  No sign assertion about either
error separately is required.
\begin{lemma}[Coupled kernel comparison]\label{ab:lem:two-two}
Let $m>0$ and $0\leq\theta\leq1$.  Suppose $F,G$ are continuous on
the closed quadrant, $F$ is $C^1$ where $y>0$ and $z > 0$, and $F$ is bounded
from below.  Assume that the following integral converges absolutely for $y, z >0$:
\[
(\mathcal TF)(y,z)
=2\sinh^2y\int_y^\infty\frac{F(r,z)}{\sinh^3r}\,dr.
\]
If for $y > 0$ and $z > 0$ we have
\[
F_y-mF_z\leq m\coth y\,F-m\csch y\,G,
\qquad G\geq\theta\mathcal TF,
\qquad F(y,0)\geq0,
\]
then $F,G\geq0$.  No boundary derivatives or sign condition at
$y=0$ are required.
\end{lemma}
\begin{proof}
Fix $\varepsilon>0$ and for $z > 0$ put
\[
n(z)=\sup_{r\geq\varepsilon}(-F(r,z))_+.
\]
This is a bounded measurable function.  Positivity of $\mathcal T$
and $\mathcal T1\leq1$, proved in \eqref{pc:eq:Tcontraction}, give
$G(r,z)\geq-n(z)$ for $r\geq\varepsilon$.
The characteristic integration in
\eqref{paper:eq:onecomponentbound}, before imposing a sign on $G$,
therefore gives, for $y\geq\varepsilon$,
\[
F(y,z)\geq
-m\int_0^{z/m}
\left(\frac{\sinh y}{\sinh(y+t)}\right)^m
\frac{n(z-mt)}{\sinh(y+t)}\,dt
\geq-\csch\varepsilon\int_0^z n(s)\,ds.
\]
Taking negative parts and the supremum over $y\geq\varepsilon$
yields
\begin{equation}\label{ab:eq:negativepart}
n(z)\leq\csch\varepsilon\int_0^z n(s)\,ds.
\end{equation}
Gr\"onwall's inequality gives $n=0$.  Since $\varepsilon$ was
arbitrary, $F\geq0$ for $y, z>0$, and $G\geq\theta\mathcal TF\geq0$.
Continuity supplies both signs at $y=0$ and/or $z=0$.
\end{proof}
\begin{remark}\label{lem42remark}In  particular, the above lemma applies with $\theta=1$ whenever
\[
G_y=2\coth y\,G-2\csch y\,F,
\qquad
\liminf_{b\to\infty}\frac{G(b,z)}{\sinh^2b}\geq0,
\]
because the integrating-factor calculation above gives
$G\geq\mathcal TF$.  This covers the homogeneous $2$--$2$ and
$3$--$2$ systems with the same proof.  The nonnegative lower limits
at infinity established for the residuals imply the required lower
bound for $F$.
For two controls with the same first trace, a signed diagonal
companion can also be allowed.  The following argument reduces its
sign to a scalar equation on the common edge.
\end{remark}
\begin{lemma}[A paired boundary principle]\label{bo:lem:paired}
Let $F,R,S$ be continuous on the closed quadrant and continuously
differentiable locally up to $y=0$, away from the origin. Suppose that,
for $y,z>0$,
\begin{align}
F_y-2F_z&=2\coth y\,F-\csch y\,(R+S)+e,
& e&\leq0,\label{bo:eq:paired1}\\
S_y&=2\coth y\,S-2\csch y\,F,\label{bo:eq:paired2}
\end{align}
where $e$ has a finite nonpositive limit at $y=0$.
With $t=y+z/2$, assume the representations
\begin{align}
R(y,z)&=\frac{\sinh(z/2)}{\sinh t}F(0,2t)
+\frac{\sinh y}{\sinh t}r(t),\notag\\
S(y,z)&=\frac{\sinh(z/2)}{\sinh t}F(0,2t)
+\frac{\sinh y}{\sinh t}s(t),
\label{bo:eq:companions}
\end{align}
and the signs $r+s\geq0$, $2r+s\geq0$ for $t > 0$.
If
\[
F(y,0)\geq0,\qquad
\liminf_{z\to\infty}\frac{F(0,z)}{\sinh^3(z/2)}\geq0,
\]
then $F\geq0$ and $R+S\geq0$.
If also $r,s\geq0$, then $R,S\geq0$ individually.
\end{lemma}
\begin{proof}
Put $f(z)=F(0,z)$ and $e_0(z)=\lim_{y\downarrow0}e(y,z)$.
Fix $z>0$, so $t > 0$ as well. The representations \eqref{bo:eq:companions} give
$R(0,z)=S(0,z)=f(z)$. By the assumed one-sided $C^1$ regularity
and $\sinh y/y\to1$, we have the following limits
\begin{align*}
\frac{S(y,z)-F(y,z)}{\sinh y}
&\longrightarrow S_y(0,z)-F_y(0,z),\\
\frac{2F(y,z)-R(y,z)-S(y,z)}{\sinh y}
&\longrightarrow 2F_y(0,z)-R_y(0,z)-S_y(0,z)
\end{align*}
as $y\downarrow0$. Rewrite \eqref{bo:eq:paired1}--\eqref{bo:eq:paired2}
as
\begin{align*}
F_y-2F_z&=\frac{2F-R-S}{\sinh y}
       +2\frac{\cosh y-1}{\sinh y}F+e,\\
S_y&=2\frac{S-F}{\sinh y}
       +2\frac{\cosh y-1}{\sinh y}S.
\end{align*}
Since $(\cosh y-1)/\sinh y\to0$ and $F_z(0,z)=f'(z)$,
passing to the one-sided limit yields
\begin{align*}
F_y(0,z)-2f'(z)
 &=2F_y(0,z)-R_y(0,z)-S_y(0,z)+e_0(z),\\
S_y(0,z)&=2\bigl(S_y(0,z)-F_y(0,z)\bigr).
\end{align*}
Combining these two identities gives
\[
S_y(0,z)=2F_y(0,z),\qquad
R_y(0,z)+F_y(0,z)-2f'(z)=e_0(z).
\]
No analyticity or second derivatives of $F,R,S$ are needed for these
limits. Differentiating \eqref{bo:eq:companions} at $y=0$ gives
\begin{align*}
R_y(0,z)&=2f'(z)-\coth(z/2)f(z)+\frac{r(z/2)}{\sinh(z/2)},\\
S_y(0,z)&=2f'(z)-\coth(z/2)f(z)+\frac{s(z/2)}{\sinh(z/2)}.
\end{align*}
Substituting these identities in the previous ones and eliminating $F_y(0,z)$ yields
\begin{equation}\label{bo:eq:edge-ode}
f'(z)=\frac32\coth(z/2)f(z)
-\frac{2r(z/2)+s(z/2)}{2\sinh(z/2)}+e_0(z).
\end{equation}
Therefore
\[
\left(\frac{f(z)}{\sinh^3(z/2)}\right)'
=-\frac{2r(z/2)+s(z/2)}{2\sinh^4(z/2)}
+\frac{e_0(z)}{\sinh^3(z/2)}\leq0.
\]
The terminal condition gives $f\geq0$ for $z>0$.
Interpolation and $r+s\geq0$ now give $R+S\geq0$ for $y, z > 0$, and continuity extends this to all $y, z \geq 0$.
Lemma~\ref{paper:lem:onecomponent}, with $m=2$ and companion
$(R+S)/2$, proves $F\geq0$.
When $r,s\geq0$, interpolation gives the individual companion signs.
\end{proof}

\section{Verification in regions A and B}
\label{sec:ab-verification}

Write $q_i=x_{i+1}-x_i$ for the consecutive gaps, and set
\[
 A=\{q\in[0,\infty)^4:q_4\ge q_2,\ 2q_1\ge q_4-q_2\},
 \qquad
 B=\{q\in[0,\infty)^4:q_2\ge q_4\}.
\]
For $w\in\{0,1\}^5$, let $d(w)_i=w_{i+1}-w_i$ and define
\[
 R_w(q)=v(q)-\frac12\bigl(d(w)\cdot\nabla_q\bigr)^2v(q).
\]
The verification in these regions combines the exact equality directions,
the four-expert boundary values, and one-dimensional sign certificates.
Table~\ref{ab:case-numbering} records sixteen control representatives.
The complementary control $1-w$ gives the same residual, because its
gap direction is $-d(w)$. Thus these sixteen representatives account for
all thirty-two controls.
\begin{table}[htbp]
\centering
\begin{tabular}{ccl@{\qquad}ccl}
\hline
Case & $w$ & $d(w)$ & Case & $w$ & $d(w)$\\
\hline
1 & $00000$ & $(0,0,0,0)$ & 9 & $10010$ & $(-1,0,1,-1)$\\
2 & $10000$ & $(-1,0,0,0)$ & 10 & $10001$ & $(-1,0,0,1)$\\
3 & $01000$ & $(1,-1,0,0)$ & 11 & $01100$ & $(1,0,-1,0)$\\
4 & $00100$ & $(0,1,-1,0)$ & 12 & $01010$ & $(1,-1,1,-1)$\\
5 & $00010$ & $(0,0,1,-1)$ & 13 & $01001$ & $(1,-1,0,1)$\\
6 & $00001$ & $(0,0,0,1)$ & 14 & $00110$ & $(0,1,0,-1)$\\
7 & $11000$ & $(0,-1,0,0)$ & 15 & $00101$ & $(0,1,-1,1)$\\
8 & $10100$ & $(-1,1,-1,0)$ & 16 & $00011$ & $(0,0,1,0)$\\
\hline
\end{tabular}
\caption{Control numbering: A$j$ and B$j$ mean case $j$ in the
respective region.}
\label{ab:case-numbering}
\end{table}

\subsection{Reduction along equality directions}

The common equality direction is
\[
 a=(0,1,-1,1)=d(00101).
\]
The additional equality directions are
\[
 b=(1,-1,0,1)=d(01001)\quad\hbox{in }A,
 \qquad
 c=(0,1,0,-1)=d(00110)\quad\hbox{in }B.
\]
Thus the linear equations supplied by the construction are
\begin{equation}\label{ab:eq-directions}
 (a\cdot\nabla_q)^2v=2v,
 \qquad
 (b\cdot\nabla_q)^2v=2v\ \hbox{in }A,
 \qquad
 (c\cdot\nabla_q)^2v=2v\ \hbox{in }B.
\end{equation}
The first identity can be observed directly from the
interpolation formulas \eqref{pc:eq:vAC}--\eqref{pc:eq:vB}: along an $a$-characteristic its two face
functions and their common denominator are constant, while each numerator
is a hyperbolic sine with argument varying at speed $\sqrt2$.

\begin{proposition}[Boundary reduction]\label{ab:boundary-reduction}
Suppose the residuals extend continuously to the indicated faces, and the
candidate is smooth in each open region. To prove $R_w\ge0$ in $A$, it
suffices to prove it on
\[
 \{q_2=0\}\cap A
 \quad\hbox{and}\quad
 \{q_3=0,\ q_2=q_4\}.
\]
To prove $R_w\ge0$ in $B$, it suffices to prove it on
\[
 \{q_4=0\}\cap B
 \quad\hbox{and}\quad
 \{q_3=0,\ q_2=q_4\}.
\]
\end{proposition}

\begin{proof}
The operators have constant coefficients. Hence,
$1-\tfrac12(d(w)\cdot\nabla_q)^2$ commutes with each differential
operator in \eqref{ab:eq-directions}. Whenever $v$ satisfies any one
of these equations in a region, each residual $R_w$ satisfies that
same equation. For the corresponding direction $\xi\in\{a,b,c\}$,
the restriction $r(t)=R_w(q+t\xi)$ therefore satisfies $r''(t)=2r(t)$.
If the characteristic segment has endpoints at $t=t_-$ and $t=t_+$, then
\[
 r(t)=
 \frac{\sinh(\sqrt2(t_+-t))}{\sinh(\sqrt2(t_+-t_-))}r(t_-)
 +\frac{\sinh(\sqrt2(t-t_-))}{\sinh(\sqrt2(t_+-t_-))}r(t_+).
\]
Both coefficients are nonnegative. Thus, to prove $r(t)\geq0$ for
$t\in(t_-,t_+)$, it suffices to prove $r(t_-)\geq0$ and $r(t_+)\geq0$.

In $A$, the $a$-characteristic through $q$ is parametrized by
$q+(0,t,-t,t)$ for $-q_2\leq t\leq q_3$ and remains in $A$.
Its endpoints lie on $q_2=0$ and $q_3=0$.
Now let $q$ be a point on the latter face of $A$. The $b$-characteristic
through this point is $q+(t,-t,0,t)$, with endpoints at
$t=-(q_4-q_2)/2\leq0$ and $t=q_2\geq0$. Along this segment,
$q_1+t\geq q_1-(q_4-q_2)/2\geq0$ by the definition of $A$.
The endpoints lie on $q_2=q_4$ and $q_2=0$, respectively.
This proves the first assertion.

In $B$, the $a$-characteristic has endpoints at $t=-q_4$ and $t=q_3$,
on $q_4=0$ and $q_3=0$. Now let $q$ be a point on the latter face
of $B$. The $c$-characteristic through this point is $q+(0,t,0,-t)$,
with $-(q_2-q_4)/2\leq t\leq q_4$. Its endpoints lie on $q_2=q_4$
and $q_4=0$, respectively. The same interpolation proves the second assertion.
Boundary cases follow by continuity.
\end{proof}

This reduction uses the linear equality equations and does not presume the
nonlinear HJB inequality that is being proved. In particular, continuity
across an interface supplies only the values on that interface; the
characteristic argument supplies the interior conclusion.

\subsection{The sign of the common kernel}

The integral representations in $A$ and $B$ use
\begin{equation}\label{ab:kernel}
 g(r)=\frac{8-25r^2+15r^4
       -15r(r^2-1)^2\operatorname{arccoth}r}
 {16\sqrt2\,r\sqrt{r^2-1}},\qquad r>1.
\end{equation}
The following series gives the kernel signs needed in the polynomial and
monotonicity arguments.

\begin{lemma}\label{ab:kernel-signs}
The function $G=-g$ is completely monotone on $(1,\infty)$:
$(-1)^kG^{(k)}(r)>0$ for every integer $k\ge0$. In particular,
$g<0$, $g'>0$, and $g''<0$.
\end{lemma}

\begin{proof}
The expansion $\operatorname{arccoth}r=\sum_{n\ge0}r^{-2n-1}/(2n+1)$
gives
\begin{equation}\label{ab:kernel-series}
 G(r)=\frac{15}{2\sqrt2}\sum_{n=3}^{\infty}
 \frac{r^{3-2n}}{(2n-3)(2n-1)(2n+1)\sqrt{r^2-1}}.
\end{equation}
Indeed, if $S(t)=\sum_{n\ge0}t^n/(2n+1)$, then
\[
 (1-t)^2S(t)=1-\frac53t+\frac8{15}t^2+
 \sum_{n=3}^{\infty}
 \frac{8t^n}{(2n-3)(2n-1)(2n+1)}.
\]
Substituting $t=r^{-2}$ in \eqref{ab:kernel} proves
\eqref{ab:kernel-series}. Furthermore,
\[
 \frac{r^{3-2n}}{\sqrt{r^2-1}}
 =\sum_{k=0}^{\infty}\frac{\binom{2k}{k}}{4^k}
   r^{-(2n+2k-2)}.
\]
All coefficients are positive, all exponents are strictly negative, and
these series and their derivatives converge locally uniformly for $r>1$.
Termwise differentiation proves the assertion.
\end{proof}

For example, if $P\le0$ on $[1,M]$, $E\ge0$, and
$\int_1^M(-P(r))\,dr\ge E$, then, for every $\lambda\ge0$,
\begin{align*}
 &\int_1^M e^{-2\lambda r}g(r)P(r)\,dr
       +e^{-2\lambda M}g(M)E\\
 &\qquad\ge
 e^{-2\lambda M}G(M)
 \left(\int_1^M(-P(r))\,dr-E\right)\ge0.
\end{align*}
Here both $G$ and the exponential factor decrease. This is an exact
way to control a negative endpoint term; knowing the sign of the integral
alone would not suffice.

\subsection{Integral representation of the residuals}
\label{ad:sec:residuals}

Write $V_4(a,b,c)$ for the four-expert correction in consecutive gaps.
An explicit expression is
\begin{align}
 V_4(a,b,c)={}&\frac1{\sqrt2}\arctan(e^{-t})
       \cosh\frac{c-a}{\sqrt2}\cosh t\cosh\frac{c+a}{\sqrt2}
       -\frac{\sqrt2}{4}\sinh(\sqrt2c)\notag\\
 &+\frac1{\sqrt2}\operatorname{arctanh}(e^{-t})
       \sinh\frac{c-a}{\sqrt2}\sinh t\sinh\frac{c+a}{\sqrt2},
 \qquad t=\frac{a+2b+c}{\sqrt2}.
 \label{ad:eq:four-expert}
\end{align}
At the origin the right-hand side is interpreted by continuity.
Its four-expert residuals are nonnegative by~\cite{BEZ}.

For $q\in A$, put
\[
 x=\sqrt2q_1,\quad Y=\sqrt2(q_2+q_3),\quad
 z=\sqrt2(q_4-q_2),\quad
 \alpha=Y+\frac z2,\quad\lambda=x-\frac z2,
 \quad M=\coth\alpha.
\]
For $n\geq1$, define
\begin{equation}\label{ad:eq:moments}
 I_n(\lambda,\alpha)=\int_1^M
       \frac{(M-r)^{n-1}}{(n-1)!}e^{-2\lambda r}g(r)\,dr,
 \qquad I_0(\lambda,\alpha)=e^{-2\lambda M}g(M).
\end{equation}
Here $\lambda\geq0$, and $g$ is given by \eqref{ab:kernel}. In view of \eqref{pc:eq:ACpair} and \eqref{pc:eq:datum}, substituting $\phi(s) = \phi_w(s) = \widetilde L(s,w)$ which is given by \eqref{paper:eq:L} in \eqref{paper:eq:K} and then performing some tedious computations that involve integrations by parts and the substitution $r=\coth s$, we can write $F(x,y,z)=(\mathcal K_p\phi_w)(y)$ in terms of the iterated integrals $I_n(\lambda, \alpha)$ given by \eqref{ad:eq:moments}. Moreover, taking $\phi$ in \eqref{paper:eq:T} to be the expression we just obtained for $\mathcal K_p\phi_w$ and performing additional integrations by parts, we get a similar expression for $R(x,y,z)=(\mathcal T\mathcal K_p\phi_w)(y)$. Finally, plugging the expressions we just obtained for $F(x,y,z)$ and $R(x,y,z)$ in \eqref{pc:eq:vAC} allows us to write $v(q)$ in region A in terms of the iterated integrals $I_n(\lambda, \alpha)$. Specifically,
writing $u=\sinh Y$, $h=\sinh(z/2)$, $S=\sinh\alpha$, the
candidate in region A takes the form
\begin{align}
 v_A(q)&=V_4(q_2,q_3,q_4)+C(q),\notag\\
 C(q)&=\frac{\sinh(\sqrt2q_3)}{\sinh Y}
       \left(-\frac{u^2}{S}I_2-4uhI_3-6Sh^2I_4\right)\notag\\
 &\quad+\frac{\sinh(\sqrt2q_2)}{\sinh Y}
       \left(-2SuI_3-6S^2hI_4\right).
 \label{ad:eq:correction}
\end{align}
The quotients at $Y=0$ are understood through the continuous extension
proved in Section~\ref{sm:section}. To derive a similar formula for $v(q)$ in region $B$, we work in the same way but we use equations \eqref{pc:eq:vB} and \eqref{correctionterms} instead of \eqref{pc:eq:vAC} and \eqref{pc:eq:ACpair} respectively, and we find that
\begin{equation}\label{ad:eq:Btransform}
 v_B(q)=V_4(q_2,q_3,q_4)+C(Tq),
 \qquad Tq=\left(q_1+\frac{q_2-q_4}{2},q_4,q_3,q_2\right).
\end{equation}
Thus $Tq\in A$ whenever $q\in B$.

All directional
derivatives required below can be evaluated using only
\begin{equation}\label{ad:eq:recurrence}
 \partial_\lambda I_n=-2(MI_n-nI_{n+1})\quad(n\geq1),
 \qquad
 \partial_\alpha I_n=-\csch^2\alpha\,I_{n-1}\quad(n\geq1).
\end{equation}
In particular, no antiderivative of $g$ is needed.

For completeness, there is an equivalent single-kernel formula. Let
$D=M-r$ and put
\begin{align*}
 K_f&=-\frac{u^2}{S}D-2uhD^2-Sh^2D^3,
 &K_r&=-SuD^2-S^2hD^3,\\
 K&=\frac{\sinh(\sqrt2q_3)}{\sinh Y}K_f
   +\frac{\sinh(\sqrt2q_2)}{\sinh Y}K_r.
\end{align*}
Then $C=\int_1^M e^{-2\lambda r}g(r)K(q,r)\,dr$.
For a constant direction $d$, write $L=d\cdot\nabla_q$ and let $L$
act on $K$ while holding $r$ fixed. Since $K(q,M)=0$,
Leibniz' rule gives
\begin{align}
 C-\frac12L^2C
  ={}&\int_1^M e^{-2\lambda r}g(r)
       \left[K-\frac12(L-2rL\lambda)^2K\right]dr\notag\\
    &-\frac12e^{-2\lambda M}g(M)(LK)(q,M)\,LM,
 \label{ad:eq:residualformula}
\end{align}
where $(L-2rL\lambda)^2K$ stands for $(L^2 - 4rL\lambda\cdot L + 4r^2(L\lambda)^2)K$, which equals $(L-2rL\lambda)((L-2rL\lambda)K)$ because $L^2\lambda = 0$. The last non-integral term in the previous display is called the endpoint term and must be retained when differentiating a residual.

Every coefficient of $K_f$ and $K_r$ is nonpositive on $A$. Since
$g<0$, differentiation twice in $q_1$ gives
\begin{equation}\label{ad:eq:first-convexity}
 C_{q_1q_1}=8\int_1^M r^2 e^{-2\lambda r}g(r)K(q,r)\,dr\geq0.
\end{equation}
Consequently the residual of $C$ is concave as a function of the first
component of its direction.

The integral kernels also give the uniform limits needed for comparison
on unbounded faces. Set
\[
 \delta=\frac{q_4-q_2}{\sqrt2},\qquad
 \mu=\frac{q_4+q_2}{\sqrt2},\qquad
 H_l(r)=\cosh l-r\sinh l.
\]
The kernel factors as $K=-H_\alpha H_\delta H_\mu$, with
$0\leq\delta\leq\mu\leq\alpha$. For $\alpha\geq1$ and
$1\leq r\leq\coth\alpha$,
\[
 |H_l(r)|+|\partial_lH_l(r)|\leq C e^{-l}
 \quad(0\leq l\leq\alpha),\qquad
 \int_1^{\coth\alpha}|g(r)|\,dr\leq C e^{-\alpha}.
\]
The constants are independent of $\alpha\geq1$. For the first bound,
$H_l(r)\geq0$ and $\partial_lH_l(r)\leq0$ throughout the stated range.
Indeed, for $l>0$ we have $r\leq\coth\alpha\leq\coth l$, while
$r\geq1>\tanh l$; at $l=0$ the signs follow directly. Thus
\[
 |H_l(r)|+|\partial_lH_l(r)|
 =H_l(r)-\partial_lH_l(r)
 =(1+r)e^{-l}\leq(1+\coth1)e^{-l}.
\]
In particular, $C=1+\coth1$ is admissible for the first bound.
The restriction $\alpha\geq1$ is essential: at $l=\alpha$ and
$r=\coth\alpha$, the derivative equals $-1/\sinh\alpha$, so no such
constant is uniform as $\alpha\downarrow0$. Only the tail estimate is
needed here. For the second bound, the estimate
$g(r)=O((r-1)^{-1/2})$ as $r\downarrow1$ and continuity away from
$r=1$ give $|g(r)|\leq C_0(r-1)^{-1/2}$ for
$1<r\leq\coth1$. Hence
\[
 \int_1^{\coth\alpha}|g(r)|\,dr
 \leq C_0\int_1^{\coth\alpha}(r-1)^{-1/2}\,dr
 =2C_0\sqrt{\coth\alpha-1}
 \leq C_1e^{-\alpha},\qquad \alpha\geq1,
\]
where the last inequality follows from
$\coth\alpha-1=2/(e^{2\alpha}-1)$.
Since $H_l''=H_l$, the same bounds control the first two directional
derivatives of $K=-H_\alpha H_\delta H_\mu$. For the endpoint term, we have
$H_\alpha(M)=0$, $\partial_\alpha H_\alpha(M)=-\csch\alpha$,
and $LM=O(e^{-2\alpha})$. Formula~\eqref{ad:eq:residualformula}
therefore gives, for each fixed direction $d$,
\begin{equation}\label{ad:eq:uniform-tail}
 \left|C-\tfrac12(d\cdot\nabla)^2C\right|
 \leq C_d e^{-2\alpha-\delta-\mu},\qquad \alpha\geq1,
 \quad\lambda\geq0.
\end{equation}
In particular, the correction residuals tend to zero as
$q_2+q_3+q_4\to\infty$, uniformly in the admissible values of $q_1$.
The nonnegative four-expert
residuals supply the corresponding lower limits for the full residuals.

\subsection{The comb inequality in region A}

The comb residual admits a signed-kernel factorization on both
interpolation faces.

Use the scaled variables
\[
 x=\sqrt2q_1,\qquad Y=\sqrt2(q_2+q_3),\qquad
 z=\sqrt2(q_4-q_2),\qquad \lambda=x-z/2\ge0,
\]
and abbreviate
\[
 \alpha=Y+z/2,\quad S=\sinh\alpha,\quad M=\coth\alpha,
 \quad h=\sinh(z/2),\quad u=\sinh Y.
\]
The region-A formula is
\[
 v=f(x,Y,z)\frac{\sinh(\sqrt2q_3)}{\sinh Y}
    +r_1(x,Y,z)\frac{\sinh(\sqrt2q_2)}{\sinh Y}.
\]
Only the following parts of $f$ and $r_1$ depend on $x$:
\begin{align}
 f_{\rm dep}&=-\frac{u^2}{S}I_2-4uh I_3-6Sh^2I_4,\notag\\
 (r_1)_{\rm dep}&=-2Su I_3-6S^2h I_4,\label{ab:comb-dependent}
\end{align}
where
\[
 I_n=\int_1^M\frac{(M-\rho)^{n-1}}{(n-1)!}
                  e^{-2\lambda\rho}g(\rho)\,d\rho.
\]
These formulas follow by substituting $\rho=\coth s$ in the nested
integrals of the original representation and reversing their order.
The remaining terms are the $x$-independent four-expert terms.

\begin{proposition}[Case A12]\label{ab:comb-positive}
The residual of the comb control $w=(0,1,0,1,0)$ is nonnegative
throughout region $A$.
\end{proposition}
\begin{proof}
We first work where $Y>0$. Recall that $a=(0,1,-1,1)$ is the common
equality direction, and let $e_1=(1,0,0,0)$ be the first coordinate
vector in the gap space $\mathbb{R}^4$. For the comb control
$w=(0,1,0,1,0)$, we have $d(w)=(1,-1,1,-1)=e_1-a$.
Thus the common equality equation in \eqref{ab:eq-directions} gives
\[
 R_w=-\tfrac12v_{q_1q_1}+\partial_{q_1}(a\cdot\nabla_qv).
\]
Hence, directly differentiating the interpolation formula \eqref{pc:eq:vAC} on its two faces gives
\begin{align*}
 R_w\big|_{q_2=0}
   &=-2\coth Y\,f_x+2\operatorname{csch}Y\,(r_1)_x-f_{xx},\\
 R_w\big|_{q_3=0}
   &=-2\operatorname{csch}Y\,f_x+2\coth Y\,(r_1)_x-(r_1)_{xx}.
\end{align*}
Substitution of \eqref{ab:comb-dependent} yields the factorizations
\begin{align}
 R_w\big|_{q_2=0}
  &=-4\int_1^M e^{-2\lambda\rho}g(\rho)\rho(\rho^2-1)(M-\rho)
       \bigl[Sh^2(M-\rho)+uh\bigr]\,d\rho,\notag\\
 R_w\big|_{q_3=0}
  &=-4\int_1^M e^{-2\lambda\rho}g(\rho)\rho(\rho^2-1)(M-\rho)
       \bigl[S^2h(M-\rho)+Su\bigr]\,d\rho.
       \label{ab:comb-factorizations}
\end{align}
For example, the kernels of $f_{\rm dep}$ and $(r_1)_{\rm dep}$,
apart from $e^{-2\lambda\rho}g(\rho)$, are
\[
 K_f=-\frac{u^2}{S}D-2uhD^2-Sh^2D^3,
 \qquad K_r=-SuD^2-S^2hD^3,\qquad D=M-\rho.
\]
The first factorization is the polynomial identity
\[
 4\rho\coth Y K_f-4\rho\operatorname{csch}Y K_r-4\rho^2K_f
   =-4\rho(\rho^2-1)D(Sh^2D+uh);
\]
the second follows by exchanging the two face coefficients, giving
$4\rho\operatorname{csch}Y K_f-4\rho\coth Y K_r-4\rho^2K_r$.
These identities use only $S=\sinh(Y+z/2)$ and $M=\coth(Y+z/2)$.

All factors in the brackets of \eqref{ab:comb-factorizations} are
nonnegative, and $g<0$ by Lemma~\ref{ab:kernel-signs}. Thus both face
residuals are nonnegative, so Proposition~\ref{ab:boundary-reduction} propagates this
sign to all of $A$. Degenerate face points follow by the continuous
extension of the residual established in the regularity analysis.
\end{proof}

\subsection{A one-crossing integral argument}

The following principle reduces a minimization problem to its endpoints without requiring
the derivative to have a fixed sign.

\begin{lemma}\label{ab:single-crossing}
Suppose $H$ is continuously differentiable on $[x_0,\infty)$,
has a finite limit at infinity, and
\[
 H'(x)=\int_1^M e^{-2x\rho}k(\rho)\,d\rho,
\]
where the integrals are absolutely convergent. If, for some
$\rho_0\in[1,M]$, $k\le0$ on $(1,\rho_0)$ and $k\ge0$ on
$(\rho_0,M)$, then
\[
 H(x)\ge\min\{H(x_0),\lim_{t\to\infty}H(t)\}
 \qquad (x\ge x_0).
\]
\end{lemma}
\begin{proof}
The function $e^{2\rho_0x}H'(x)$ is nonincreasing, since its
derivative is
\[
 -2\int_1^M(\rho-\rho_0)e^{-2x(\rho-\rho_0)}k(\rho)\,d\rho
 \le0.
\]
Consequently $H'$ can change sign only from positive to negative.
The function $H$ is first nondecreasing and then nonincreasing,
with either part possibly absent. Its infimum is therefore attained
at an endpoint, with infinity understood through the stated limit.
\end{proof}

\subsection{The first five nonequality controls in region A}
\label{a26:section}

We give the verification for controls A2--A6. The order of the proof
is A2, A3, A6, followed by the comparisons establishing A4 and A5.
Throughout, the four-expert verification is used for the nonnegativity
of its individual residuals. For A4 and A5 we also use the residual
comparisons established in \cite[Section~6]{BEZ}.

Put
\[
 \alpha=\frac{q_2+2q_3+q_4}{\sqrt2},\quad
 \delta=\frac{q_4-q_2}{\sqrt2},\quad
 \mu=\frac{q_4+q_2}{\sqrt2},\quad
 \lambda=\sqrt2q_1-\delta,\quad M=\coth\alpha.
\]
Thus $\lambda,\delta\ge0$ in $A$. For an integration variable $\rho$,
recall that
\[
 H_l(\rho)=\cosh l-\rho\sinh l,
 \qquad K(q,\rho)=-H_\alpha(\rho)H_\delta(\rho)H_\mu(\rho).
\]
The candidate solution $v$ in region $A$ is given by
\begin{equation}\label{a26:kernel-representation}
 v_A(q)=V_4(q_2,q_3,q_4)
       +\int_1^M e^{-2\lambda\rho}g(\rho)K(q,\rho)\,d\rho.
\end{equation}
Here $V_4$ is the four-expert correction, with the normalization used
in the construction. For $d_j=d(w_j)$ and $D_j=d_j\cdot\nabla_q$, set
\begin{align}
 P_j(q,\rho)&=K(q,\rho)
       -\tfrac12(D_j-2\rho D_j\lambda)^2K(q,\rho),\notag\\
 E_j(q)&=-\tfrac12(D_jM)(D_jK)(q,M),\notag\\
 N_j(q)&=V_4(q_2,q_3,q_4)-\tfrac12D_j^2V_4(q_2,q_3,q_4).
 \label{a26:coefficients}
\end{align}
In differentiating $K$, the variable $\rho$ is held fixed. Denote by $R_j$ the residual $R_{w_j}$ for $j \in \{1, \ldots, 16\}$. Recalling \eqref{ad:eq:correction} and \eqref{ad:eq:residualformula} with $d=d_j$, so that $L=d_j\cdot\nabla_q=D_j$, we obtain
\begin{equation}\label{a26:residual}
 R_j(q)=N_j(q)+\int_1^M e^{-2\lambda\rho}g(\rho)P_j(q,\rho)\,d\rho
              +e^{-2\lambda M}g(M)E_j(q),
\end{equation}
where we recall that $(D_j-2\rho D_j\lambda)^2K$ stands for $(D_j^2 - 4\rho D_j\lambda\cdot D_j + 4\rho^2(D_j\lambda)^2)K$. The last formula also fixes all normalizations in the calculations below.

We write the two characteristic face traces as
\begin{align*}
 F_j(\lambda,Y,z)
 &=R_j\left(\frac{\lambda+z/2}{\sqrt2},0,
                         \frac Y{\sqrt2},\frac z{\sqrt2}\right),\\
 G_j(\lambda,Y,z)
 &=R_j\left(\frac{\lambda+z/2}{\sqrt2},\frac Y{\sqrt2},0,
                         \frac{Y+z}{\sqrt2}\right).
\end{align*}
The variables $\lambda,Y,z$ are nonnegative. Proposition~\ref{ab:boundary-reduction} reduces each inequality to these two faces, with
the second face further reduced to the case where $z=0$ when
needed.

\begin{proposition}\label{a26:case2}
The residual $R_2$ of the control $(1,0,0,0,0)$ is nonnegative in $A$.
\end{proposition}
\begin{proof}
Here $E_2=0$ and $P_2=(1-4\rho^2)K\ge0$: indeed,
$0\le\delta\le\mu\le\alpha$ implies $H_l(\rho)\ge0$ for
$l\in\{\alpha,\delta,\mu\}$ and $1\le\rho\le\coth\alpha$.
Since $g<0$, \eqref{a26:residual} is nondecreasing in $\lambda$.
It suffices to set $\lambda=0$.

Applying $1-\tfrac12D_2^2$ to \eqref{pc:eq:vAC}, whose weights are constant in the $D_2$ direction, and restricting separately to the faces $q_2=0$ and $q_3=0$, respectively, we obtain
\[
 F_2(\lambda,Y,z)=(1-\partial_\lambda^2)F(\lambda+z/2,Y,z),
 \qquad
 G_2(\lambda,Y,z)=(1-\partial_\lambda^2)R(\lambda+z/2,Y,z).
\]
Thus, setting $x=\lambda+z/2$ in \eqref{pc:eq:ACsystem}, applying
$1-\partial_\lambda^2$, and then setting $\lambda=0$, we obtain the
following system, in which the fixed first argument of $F_2$ and $G_2$
is suppressed:
\begin{equation}\label{a26:homogeneous-system}
 (\partial_Y-2\partial_z)F_2
       =2\coth Y F_2-2\operatorname{csch}Y G_2,
 \qquad
 \partial_YG_2=2\coth Y G_2-2\operatorname{csch}Y F_2.
\end{equation}
Lemma~\ref{ab:lem:two-two}, together with Remark~\ref{lem42remark},
reduces the question to $z=0$. We check its hypotheses explicitly.
Proposition~\ref{sm:prop:closedchamber} gives continuity of the residuals
on the closed faces, while the explicit integral formulas give $C^1$
regularity for $Y,z>0$. Since $D_2=-\partial_{q_1}$ and $V_4$ is
independent of $q_1$, we have $N_2=V_4\geq0$ by the four-expert
verification~\cite{BEZ}. The correction residual is uniformly bounded
on $Y+z/2\geq1$ by \eqref{ad:eq:uniform-tail}; continuity bounds the
full residual on the remaining compact triangle. Thus $F_2(0,Y,z)$
is bounded below on the whole quadrant. Next, for each fixed $z\geq0$, \eqref{ad:eq:four-expert} and
\eqref{ad:eq:uniform-tail} give
\[
 \lim_{Y\to\infty}F_2(0,Y,z)=\frac{e^{-z}}{2\sqrt2},
 \qquad
 \lim_{Y\to\infty}G_2(0,Y,z)=0.
\]
The first limit ensures absolute convergence of the integral defining
$\mathcal T[F_2(0,\cdot,\cdot)]$ for $Y>0$. The second gives the
companion condition required in Remark~\ref{lem42remark}, namely
\[
 \lim_{Y\to\infty}\frac{G_2(0,Y,z)}{\sinh^2Y}=0.
\]

Set $t=Y$, and let $W_F(t)=F_2(0,t,0)$ and $W_G(t)=G_2(0,t,0)$.
Formula \eqref{a26:residual} yields
\begin{align*}
 W_F(t)&=\frac{\arctan(e^{-t})\cosh t}{\sqrt2}
       +\int_1^{\coth t}(4\rho^2-1)(\cosh t-\rho\sinh t)g(\rho)\,d\rho,\\
 W_G(t)&=\frac{\arctan(e^{-t})\cosh^2t-\tfrac12\sinh t}{\sqrt2}
       +\int_1^{\coth t}(4\rho^2-1)(\cosh t-\rho\sinh t)^2g(\rho)\,d\rho.
\end{align*}
They satisfy $W_G'=2\coth t W_G-2\operatorname{csch}t W_F$ and
$W_G(t)/\sinh^2t\to0$ as $t\to\infty$. Thus $W_F\ge0$ implies
\[
 W_G(t)=2\sinh^2t\int_t^\infty\frac{W_F(s)}{\sinh^3s}\,ds\ge0.
\]
It remains to prove that $W_F(t) \geq 0$, or equivalently $J(t):=W_F(t)/\cosh t\ge0$. Differentiation gives
\[
 J'(t)=\operatorname{sech}^2t\,Q(t),\qquad
 Q(t)=\int_1^{\coth t}\rho(1-4\rho^2)g(\rho)\,d\rho
                         -\frac{\cosh t}{2\sqrt2}.
\]
$Q'$ is a strictly negative function since the integral in $Q$ has a positive integrand and a decreasing upper endpoint, while $\cosh t$ increases. Consequently $J$ can change monotonicity only from increasing to decreasing, so its infimum is at an endpoint.

We have $J(+\infty)=0$. To compute the other endpoint, recall that
$d(10000)=(-1,0,0,0)$ and hence
\[
J(t)=\frac{F_2(0,t,0)}{\cosh t}
=\frac{1}{\cosh t}
\left(v-\frac12v_{q_1q_1}\right)
\left(0,0,\frac{t}{\sqrt2},0\right).
\]
The closed-chamber regularity and the full-collision jet
\eqref{sm:eq:fullcollisionjet} give
$v(0)=u_0$ and $v_{q_1q_1}(0)=4u_0/3$. Therefore
\begin{equation}\label{a26:A2-origin}
J(0+)=u_0-\frac12\frac{4u_0}{3}
=\frac{u_0}{3}
=\frac{15\pi^2}{512\sqrt2}>0.
\end{equation}
Since the infimum of $J$ occurs at an endpoint, $J\geq0$.
This completes the proof.
\end{proof}

\begin{proposition}\label{a26:case3}
The residual $R_3$ of the control $(0,1,0,0,0)$ is nonnegative in $A$.
\end{proposition}
\begin{proof}
On the first face ($q_2=0$) write $\alpha=Y+z/2$ and introduce
\[
 a=e^z-1,\qquad b=e^{2Y+z}-e^z,\qquad
 c=\frac{\rho+1}{\rho-1}-e^{2Y+z},\qquad
 L=\frac{e^{-Y-3z/2}}{32}.
\]
Thus $a,b,c\ge0$ on the integration interval. Expanding
\eqref{a26:coefficients} gives
\[
 P_3=-\frac{8L\,\mathcal P(a,b,c)}{(a+b+c)^5},\qquad
 E_3=\frac{64L b^2(1+a+b)^2}{(a+b)^5},
\]
where the following polynomial has only nonnegative coefficients:
\begin{align*}
 \mathcal P={}&8a^4c
 +(b+c)^2\{8b(1+b)(2+b)+(8+32b+19b^2)c
                                  +2(4+7b)c^2+3c^3\}\\
 &+2a(b+c)\{4b(2+6b+3b^2)+(16+52b+35b^2)c
                                  +2(14+17b)c^2+11c^3\}\\
 &+8a^3\{2c+(b+c)(b+4c)\}\\
 &+a^2\{8c+(b+c)(24b(1+b)+(64+67b)c+43c^2)\}.
\end{align*}
The sign $P_3\le0$ makes the integral positive, whereas the endpoint
term is negative. The required domination is the exact identity
\begin{equation}\label{a26:A3-domination}
 \int_1^{\coth\alpha}(-P_3(\rho))\,d\rho-E_3
       =\frac{4L\,\mathcal Q(a,b)}{15(a+b)^5}\ge0,
\end{equation}
where
\begin{align*}
 \mathcal Q(a,b)={}&165a^4+6a^3(28+115b)
       +10b^2(4+8b+9b^2)\\
 &+4ab(44+130b+135b^2)+a^2(56+608b+975b^2).
\end{align*}
To verify \eqref{a26:A3-domination}, use
$d\rho=-2(a+b+c)^{-2}dc$ and, for $0\le k\le5$,
\[
 \int_0^\infty\frac{c^k\,dc}{(a+b+c)^7}
 =\frac{k!(5-k)!}{6!(a+b)^{6-k}}.
\]
No estimate is involved in this polynomial integration.

The function $e^{-2\lambda\rho}(-g(\rho))$ is decreasing by
Lemma~\ref{ab:kernel-signs}. Thus \eqref{a26:A3-domination} implies
that the sum of the integral and endpoint terms in
\eqref{a26:residual} is nonnegative. The four-expert term $N_3$ is
nonnegative by its verification in \cite{BEZ}. Hence $F_3\ge0$.

The second face satisfies the exact equation
\[
 \partial_YG_3=2\coth Y G_3-2\operatorname{csch}Y F_3.
\]
For fixed $\lambda,z\geq0$, the explicit four-expert formula
\eqref{ad:eq:four-expert} gives $N_3=O_z(e^{-Y})$ on this face as
$Y\to\infty$. The correction residual tends to zero by
\eqref{ad:eq:uniform-tail}. Thus $G_3/\sinh^2Y\to0$.
Multiplying the preceding equation by $\sinh^{-2}Y$ and integrating
from $Y$ to infinity, as in Remark~\ref{lem41rmk}, gives
\[
 G_3(\lambda,Y,z)
   =2\sinh^2Y\int_Y^\infty
             \frac{F_3(\lambda,s,z)}{\sinh^3s}\,ds\ge0.
\]
The equality-direction interpolation of Proposition~\ref{ab:boundary-reduction} completes the proof. Cases with
$a=0$ or $b=0$ occur when $z=0$ or $Y=0$, so they follow from the continuity of second-order derivatives given in Proposition~\ref{sm:prop:closedchamber}.
\end{proof}

\begin{proposition}\label{a26:case6}
The residual $R_6$ of the control $(0,0,0,0,1)$ is nonnegative in $A$.
\end{proposition}
\begin{proof}
By Proposition~\ref{ab:boundary-reduction}, it suffices to prove
nonnegativity of $F_6$ and $G_6$. Fix an arbitrary $\lambda\geq0$.
The pair $(F_6(\lambda,\cdot,\cdot),G_6(\lambda,\cdot,\cdot))$
satisfies \eqref{a26:homogeneous-system}, with the subscripts changed
to $6$. As in Proposition~\ref{a26:case2}, closed-face continuity
follows from Proposition~\ref{sm:prop:closedchamber}, and the integral
formulas give $C^1$ regularity for $Y,z>0$. The four-expert residual
$N_6$ is nonnegative by~\cite{BEZ}; together with
\eqref{ad:eq:uniform-tail} and continuity on the compact triangle
$Y+z/2\leq1$, this gives a lower bound for $F_6$ on the quadrant.
The explicit formula \eqref{ad:eq:four-expert} and the same tail
estimate give absolute convergence of $\mathcal T F_6$ and the terminal
condition $G_6/\sinh^2Y\to0$ as $Y\to\infty$ at fixed $z$.
Lemma~\ref{ab:lem:two-two}, together with Remark~\ref{lem42remark},
therefore reduces the proof to $z=0$ for each fixed $\lambda$.

For this boundary calculation use unscaled variables: write $x=q_1$
on both faces, and $y=q_3$ on the first face and $y=q_2=q_4$ on the
second. We need the nonnegativity of
\[
 W(x,y)=8R_6(x,0,y,0),\qquad \widetilde{W}(x,y)=8R_6(x,y,0,y),\qquad t=\sqrt2y.
\]
Writing $M=\coth t$ and $\chi=\sqrt2x$, formula
\eqref{a26:residual} gives
\begin{align}
 W(x,y)={}&2\int_1^M e^{-2\chi\rho}(1+2\rho^2)
                        (-\cosh t+\rho\sinh t)g(\rho)\,d\rho\notag\\
 &+2e^{-2\chi M}\operatorname{csch}^3t\,g(M)
 +\frac{2\arctan(e^{-t})\cosh t
              -4\operatorname{arctanh}(e^{-t})\sinh t+\tanh t}{\sqrt2},
 \label{a26:U6}\\
 \widetilde{W}(x,y)={}&\int_1^M e^{-2\chi\rho}g(\rho)
 \{3(\rho^2-1)+(1-7\rho^2)\cosh(2t)
                           +2(\rho+2\rho^3)\sinh(2t)\}\,d\rho\notag\\
 &+\frac{\arctan(e^{-t})(3-\cosh(2t))+5\sinh t
                   -4\operatorname{arctanh}(e^{-t})\sinh(2t)}{\sqrt2}.
 \label{a26:W6}
\end{align}
The equation
$\widetilde{W}_y=2\sqrt2\coth t\,\widetilde{W}-2\sqrt2\operatorname{csch}t\,W$
and $\widetilde{W}/\sinh^2t\to0$ show that $W\ge0$ suffices.

Define the second-order differential expression and its scaled derivative
\begin{align}
 B(x,y)&=\sinh t\{-W_{yy}-5\sqrt2\coth t\,W_y+12W\},\notag\\
 J(x,y)&=e^{2\chi M}\partial_y B(x,y).
 \label{a26:BJ}
\end{align}
The second-order operator in $B$ is suggested by the profile equation
\eqref{pc:eq:Uode}. Indeed, for
$h(y)=U(\sqrt2y)=v(0,0,y,0)$, that equation becomes
\[
 -h''(y)-5\sqrt2\coth(\sqrt2y)h'(y)+12h(y)
 =3\sqrt2\coth(\sqrt2y).
\]
This motivates applying the same operator to $W$ and multiplying by
$\sinh t$. We shall prove $J<0$ and $B(x,\infty)=0$, which imply
$B>0$. A one-dimensional minimum principle then gives $W\ge0$.

The following explicit form of $B$ verifies the elimination of the
integrals in $J$ and will also give its limit at infinity:
\begin{align}
 B(x,y)={}&-20\int_1^M e^{-2\chi\rho}\rho(1+2\rho^2)g(\rho)\,d\rho\notag\\
 &-2\sqrt2\{5\log\tanh(t/2)+5\operatorname{sech}t
                                      +\operatorname{sech}^3t\}\notag\\
 &+\frac{e^{-2\chi M}}{2\sinh^6t}
 \Bigl[(25-32\chi^2-40\cosh(2t)+15\cosh(4t)
                    +24\chi\sinh(2t))g(M)\notag\\
 &\hspace{39mm}+4(8\chi-3\sinh(2t))g'(M)-8g''(M)\Bigr].
 \label{a26:B-expanded}
\end{align}
It follows by differentiating \eqref{a26:U6} twice. The identities
$\partial_tM=-\operatorname{csch}^2t$ and
$\operatorname{arccoth}(\coth t)=t$ make all its terms elementary.

To give a finite sign certificate for $J$, introduce
\begin{align*}
 p_0(t)&=120t\cosh t-80\sinh t-15\sinh(3t)+\sinh(5t),\\
 p_1(t)&=120t(26+15\cosh(2t))-2415\sinh(2t)
                              -12\sinh(4t)-7\sinh(6t),\\
 p_2(t)&=3240t+4320t\cosh(2t)+1080t\cosh(4t)\\
       &+256\cosh t(4+\cosh(2t))\sinh^7t
                +3\{-584\sinh(2t)-394\sinh(4t)
                                  -24\sinh(6t)+\sinh(8t)\},\\
 p_3(t)&=-360t+256\sinh(2t)-40\sinh(4t)+\sinh(8t).
\end{align*}
Differentiating \eqref{a26:B-expanded} gives the identities
\begin{align}
 J(0,y)&=-\frac{5p_0(t)}{2\sinh^8t},
 &J_x(0,y)&=\frac{p_1(t)}{4\sqrt2\sinh^9t},\notag\\
 J_{xx}(0,y)&=-\frac{p_2(t)}{8\sinh^{10}t\cosh^3t},
 &J_{xxx}(0,y)&=-\frac{\sqrt2p_3(t)}{\sinh^{11}t},\notag\\
 J_{xxxx}(x,y)&=-256e^{2\chi M}(4+\cosh(2t))\operatorname{csch}^5t.
 &&\label{a26:J-certificates}
\end{align}
Here every identity is in the original variable $x$, not in $\chi$.

For clarity, all the one-variable signs needed in
\eqref{a26:J-certificates} are established next. They follow from
\begin{align*}
 p_0^{(5)}(t)&=10\sinh t\{12t-104\sinh(2t)+625\sinh(4t)\}>0,\\
 p_1^{(5)}(t)&=-192\sinh t
       \{-600t\cosh t+695\sinh(3t)+567\sinh(5t)\}<0,\\
 p_2^{(8)}(t)&=1280\Bigl[864t\{\cosh(2t)+64\cosh(4t)\}
             +3086\sinh(2t)+54784\sinh(4t)\\
 &\hspace{30mm}-137781\sinh(6t)+65536\sinh(8t)
                                      +78125\sinh(10t)\Bigr]>0,\\
 p_3'(t)&=512(3+\cosh(2t))\sinh^6t>0.
\end{align*}
The second sign uses $t\le\sinh t$, hence
$600t\cosh t\le300\sinh(2t)$.
For the third, $65536+78125=143661>137781$ and
$\sinh(10t),\sinh(8t)\ge\sinh(6t)$.
The endpoint data are
\[
 p_i^{(k)}(0)=0\quad(i=0,1,\ 0\le k\le4),\qquad
 p_2^{(k)}(0)=0\ (0\le k\le6),\quad
 p_2^{(7)}(0)=5898240,\quad p_3(0)=0.
\]
Taylor's formula with integral remainder proves
$p_0,p_2,p_3>0$ and $p_1<0$ on $(0,\infty)$.
Thus the first four Taylor coefficients of $J(\cdot,y)$ at $0$
are negative, and its fourth derivative is negative everywhere.
A final application of Taylor's formula gives $J(x,y)<0$ for $x\ge0$.

Every term in \eqref{a26:B-expanded} tends to zero as $y\to\infty$.
For example, $M-1=O(e^{-2t})$ and
$g^{(k)}(M)=O(e^{(2k+1)t})$ for $k=0,1,2$; the endpoint term is
$O((1+x^2)e^{-t})$ at fixed $x$.
Consequently $B(x,y)>0$ and
\begin{equation}\label{minprincstrat6}
-W_{yy}-5\sqrt2\coth(\sqrt2y)W_y+12W\ge0.
\end{equation}
The limits of $W$ at the two endpoints are
\[
 W(x,\infty)=0,\qquad
 W(x,0+)=\frac{\pi}{2\sqrt2}
          -2\int_1^\infty e^{-2\sqrt2x\rho}(1+2\rho^2)g(\rho)\,d\rho>0.
\]
The integral converges even for $x=0$, since $g(\rho)=O(\rho^{-4})$.
A negative value of $W$ would therefore give an interior negative
minimum, contradicting \eqref{minprincstrat6}.
Hence $W\ge0$, and the preceding first-order equation and terminal
condition give $\widetilde W\ge0$. The homogeneous comparison
principle and Proposition~\ref{ab:boundary-reduction} then prove the
assertion throughout $A$.
\end{proof}

\begin{proposition}\label{a26:case4}
The residual $R_4$ of the control $(0,0,1,0,0)$ is nonnegative in $A$.
\end{proposition}
\begin{proof}
On $q_2=0$, the exact coefficient identities are $P_4=P_3$ and
$E_4=E_3$. The argument in Proposition~\ref{a26:case3} proves that
their combined integral and endpoint contribution is nonnegative.
The individual four-expert residual $N_4$ is nonnegative as well.

On $q_3=0$, $q_2=q_4=y$, write $t=\sqrt2y$. The difference from
case A6 has no endpoint term, and its integral polynomial is
\begin{equation}\label{a26:diagonal-comparison}
 P_4-P_6=-\tfrac12(\rho^2-1)
                      \{1-\cosh(2t)+\rho\sinh(2t)\}\le0.
\end{equation}
Indeed, the expression in braces is at least $1-e^{-2t}\ge0$.
In the notation of \cite[Section~6]{BEZ}, the four-expert residuals
at the state with consecutive gaps $(q_2,q_3,q_4)$ are identified by
$(N_4,N_5,N_6)=(U_5,U_3,U_2)$.
The inequalities $U_5\ge U_3\ge U_2$ proved there therefore give
$N_4\ge N_6$.
Thus $R_4\ge R_6\ge0$ on this second reduced face, by
Proposition~\ref{a26:case6}. The boundary reduction in
Proposition~\ref{ab:boundary-reduction} proves the claim.
\end{proof}

\begin{proposition}\label{a26:case5}
The residual $R_5$ of the control $(0,0,0,1,0)$ is nonnegative in $A$.
\end{proposition}
\begin{proof}
We compare with case A6 on both reduced faces. On $q_2=0$, use
$Y=\sqrt2q_3$, $z=\sqrt2q_4$, $\alpha=Y+z/2$ and $M=\coth\alpha$.
There is no endpoint difference, $E_5-E_6=0$, and
\[
 P_5-P_6=(\rho^2-1)D(\rho),
\]
where $D$ is the quadratic determined by the following exact data:
\begin{align*}
 D(1)&=-e^{-Y-3z/2}(e^z-1),\\
 D(M)&=-2\operatorname{csch}^2\alpha\sinh Y\sinh(z/2),\\
 D''(\rho)&=\tfrac12e^{-Y-3z/2}(e^z-1)^2(1+e^{2Y+z})\ge0.
\end{align*}
These identities are obtained by expanding
\eqref{a26:coefficients}; they specify every coefficient of $D$.
Convexity and its nonpositive endpoint values imply $D\le0$ on
$[1,M]$. Since $g<0$, the integral contribution to $R_5-R_6$ is
nonnegative. The comparison $U_3\ge U_2$ from
\cite[Section~6]{BEZ} gives $N_5\ge N_6$ on both reduced faces.

On $q_3=0$, $q_2=q_4$, the difference $P_5-P_6$ is exactly the
right side of \eqref{a26:diagonal-comparison}, and $E_5-E_6=0$.
Thus $R_5\ge R_6\ge0$ on both reduced faces. The boundary reduction
again proves the inequality everywhere in $A$.
\end{proof}

\subsection{Three elementary cases and two polynomial factorizations}

The residual in case A1 is $v_A\geq0$, by the fixed-policy
representation of the candidate as a discounted local-time expectation.
Cases A13 and A15 are equalities by \eqref{ab:eq-directions}.
We next treat A11, A14 and A16. In the following formulas put
\[
 S=\sinh\alpha,\quad h=\sinh(z/2),\quad u=\sinh Y,
 \quad M=\coth\alpha,\quad D=M-\rho,
 \quad \beta=-\frac{u}{Sh}.
\]
Expressions containing $\beta$ are used initially for $h>0$;
their products have continuous limits at $h=0$.

\begin{proposition}[Cases A11 and A16]\label{a716:prop:A11A16}
The residuals in cases A11 and A16 are nonnegative in $A$.
\end{proposition}
\begin{proof}
Write $P_j^F(D)$ and $P_j^R(D)$ for $P_j(q,\rho)$ seen as a polynomial in $D=M-\rho$ on $q_2=0$ and
$q_3=0$, respectively. Formula \eqref{ad:eq:residualformula} gives
\begin{align}
 P_{16}^F=P_{16}^R&=0,\notag\\
 P_{11}^F(D)&=4Sh^2(D-M)(D-M+1)(D-M-1)(D-\beta)^2,\notag\\
 P_{11}^R(D)&=4S^2hD(D-M)(D-M+1)(D-M-1)(D-\beta).
 \label{a716:eq:11-polynomials}
\end{align}
For both controls the corresponding endpoint coefficients are
\[
 E^F=\frac{u^2}{S^5},\qquad E^R=0.
\]
These identities can be checked without any integration: insert the
two kernels in \eqref{ad:eq:residualformula}, differentiate, and collect
powers of $D$. In particular, the last quadratic factor in $P_{11}^F$
is the square $(D+u/(Sh))^2$.

On $0\leq D\leq M-1$ the two polynomials in
\eqref{a716:eq:11-polynomials} are nonpositive. Their integrals against
$g<0$ are consequently nonnegative. The four-expert parts of A11 and
A16 agree, since their last three gap directions differ only by sign.
It therefore suffices to prove A16 on the two faces.
On $q_3=0$ its correction residual vanishes, so the desired sign follows from the four-expert verification \cite{BEZ}. On $q_2=0$ its full
residual is exactly
\begin{equation}\label{a716:eq:16-face}
 \frac{u^2}{S^5}
 \left[e^{-2\lambda M}g(M)
       +\frac{\sinh^3\alpha\tanh\alpha}{2\sqrt2}\right].
\end{equation}
Because $g(M)<0$, the bracket is bounded below by $B(\alpha)$, where
\[
 B(t)=g(\coth t)+\frac{\sinh^3t\tanh t}{2\sqrt2}.
\]
Here we have
\begin{align*}
 B'(t)&=\frac{3}{128\sqrt2}\csch^4t\,B_1(t),\\
 B_1(t)&=120t\cosh t-115\sinh t+6\sinh3t-6\sinh5t+\sinh7t,\\
 B_1'(t)&=2\sinh t\,B_2(t),\\
 B_2(t)&=60t-5\sinh2t-23\sinh4t+7\sinh6t,\\
 B_2'(t)&=32(19+21\cosh2t)\sinh^4t.
\end{align*}
The function $B_2'$ is clearly nonnegative, so $B_2$ is increasing. Since $B_2$ vanishes at $0$, it is also nonnegative, making $B_1'$ nonnegative and thus $B_1$ increasing. Then, since $B_1$ also vanishes at $0$, it is nonnegative as well. Thus $B'$ is also nonnegative so $B$ is increasing. Since the limit of $B$ at $0+$ is $0$, it is also a nonnegative function, so the nonnegativity of \eqref{a716:eq:16-face} follows; the boundary reduction completes both cases. All degenerate faces follow by continuity.
\end{proof}

\begin{proposition}[Case A14]\label{a716:prop:A14}
The residual in case A14 is nonnegative in $A$.
\end{proposition}
\begin{proof}
There is no endpoint term. The two correction polynomials are
\begin{align*}
 P_{14}^F(D)&=4Sh^2D(D-M)(D-M+1)(D-M-1)(D-\beta),\\
 P_{14}^R(D)&=4S^2hD^2(D-M)(D-M+1)(D-M-1).
\end{align*}
As in \eqref{a716:eq:11-polynomials}, both are nonpositive for
$0\leq D\leq M-1$. Since $g<0$, the correction residual is therefore
nonnegative on both faces. The four-expert residual is nonnegative
by the verification in \cite{BEZ}. Proposition~\ref{ab:boundary-reduction}
then proves the claim.
\end{proof}

\subsection{A common one-dimensional certificate for A9 and A10}

For $\xi\geq0$ and $t>0$, write
\[
 m=\coth t,\qquad a(t)=\arctan(e^{-t}),\qquad
 b(t)=\operatorname{arctanh}(e^{-t}).
\]
Define
\begin{align}
 F(\xi,t)={}&2\int_1^m e^{-2\xi r}g(r)
  \bigl[(-1+6r^2)\cosh t+r(9-14r^2)\sinh t\bigr]dr\notag\\
 &+2e^{-2\xi m}\csch^3t\,g(m)
 +\frac{2a(t)\cosh t-4b(t)\sinh t+\tanh t}{\sqrt2},
 \label{a716:eq:F910}\\
 G_{10}(\xi,t)={}&\int_1^m e^{-2\xi r}g(r)
  \bigl[-3(8r^4-9r^2+1)+(24r^4-15r^2+1)\cosh2t\notag\\
 &\hspace{44mm}+2r(9-14r^2)\sinh2t\bigr]dr\notag\\
 &+\frac{-a(t)(-3+\cosh2t)+5\sinh t-4b(t)\sinh2t}{\sqrt2},
 \label{a716:eq:G10}\\
 G_9(\xi,t)={}&\int_1^m e^{-2\xi r}g(r)
 \bigl[1+7r^2-8r^4
  +(-3+8r^2)((1+r^2)\cosh2t-2r\sinh2t)\bigr]dr\notag\\
 &+\frac{a(t)(-1+3\cosh2t)-3\sinh t}{\sqrt2}.
 \label{a716:eq:G9}
\end{align}
These are actual residuals, with a common positive normalization:
\begin{align}
 F(\xi,t)&=8R_{10}(\xi/\sqrt2,0,t/\sqrt2,0)
          =8R_9(\xi/\sqrt2,0,t/\sqrt2,0),\notag\\
 G_j(\xi,t)&=8R_j(\xi/\sqrt2,t/\sqrt2,0,t/\sqrt2),
 \qquad j=9,10.
 \label{a716:eq:910-residual-link}
\end{align}
In particular, the same function $F$ suffices for the two cases.

\begin{lemma}\label{a716:lem:F-positive}
For every $\xi\geq0$ and $t>0$, $F(\xi,t)\geq0$.
\end{lemma}
\begin{proof}
Set
\[
 \mathcal L=-\partial_{tt}-5\coth t\,\partial_t+6,
 \qquad Z=\sinh t\,\mathcal LF,
 \qquad E=e^{2\xi\coth t}\partial_tZ.
\]
Differentiating \eqref{a716:eq:F910}, including its moving endpoint,
eliminates every integral from $E$. The resulting exact certificates are
\begin{align}
 E(0,t)&=0,\notag\\
 E_\xi(0,t)&=-\frac{3\csch^9t}{16\sqrt2}\,C_1(t),\notag\\
 E_{\xi\xi}(0,t)&=\frac{\csch^{10}t}{4\sqrt2}\,C_2(t),\notag\\
 E_{\xi\xi\xi}(0,t)&=-\frac{\csch^{11}t}{4\sqrt2}\,C_3(t),\notag\\
 E_{\xi\xi\xi\xi}(\xi,t)
 &=-16\sqrt2e^{2\xi\coth t}(4+\cosh2t)\csch^5t,
 \label{a716:eq:E-jets}
\end{align}
where
\begin{align*}
 C_1(t)&=-120t(2+\cosh2t)+185\sinh2t-4\sinh4t+\sinh6t,\\
 C_2(t)&=-600t\cosh t+435\sinh t+54\sinh3t+2\sinh5t-\sinh7t,\\
 C_3(t)&=-360t+256\sinh2t-40\sinh4t+\sinh8t.
\end{align*}
For explicit sign certificates put
\[
 D_1=-120t\cosh t+120\sinh t-5\sinh3t+3\sinh5t,
 \qquad
 D_2=300t-165\sinh2t-3\sinh4t+7\sinh6t.
\]
Direct differentiation gives
\begin{align*}
 C_1'&=4\sinh t\,D_1,&
 D_1'&=30\sinh t(\sinh4t-4t),\\
 C_2'&=-2\sinh t\,D_2,&
 D_2'&=96(13+7\cosh2t)\sinh^4t,\\
 C_3'&=512(3+\cosh2t)\sinh^6t.
\end{align*}
All five functions vanish at zero. Since $\sinh4t>4t$ for $t>0$,
the displayed formulas give $D_1',D_2'>0$, hence $D_1,D_2>0$.
It follows that $C_1',C_3'>0$ and $C_2'<0$, so $C_1,C_3\geq0$
and $C_2\leq0$. Taylor's formula in $\xi$, applied to $E$ with an
integral remainder involving the fourth derivative, now gives
$E\leq0$ by \eqref{a716:eq:E-jets}.

For clarity, every step back from this certificate has a specified
endpoint. To check the limits at infinity, put $q=e^{-t}$ and substitute
$r=1+(m-1)s^2$, $0\leq s\leq1$, in \eqref{a716:eq:F910}.
Using $m-1=2q^2/(1-q^2)$ and the explicit kernel formula
\eqref{ab:kernel}, whose expansion at one is
$g(r)=-1/(16\sqrt{r-1})+O(\sqrt{r-1})$, the transformed expression
and its first two $t$-derivatives give
$F,F_t,F_{tt}=O(q^2)$ for each fixed $\xi\geq0$.
The logarithmic terms remain integrable after these derivatives:
their $s$-dependence is bounded by a constant times
$s^4(1+|\log s|)$.
Hence $Z=\sinh t\,\mathcal LF=O(q)$, and therefore
\[
 \lim_{t\to\infty}Z(\xi,t)=0,
 \qquad \lim_{t\to\infty}F(\xi,t)=0.
\]
Since $Z_t=e^{-2\xi\coth t}E\leq0$, the first limit gives
$Z\geq0$, hence $\mathcal LF\geq0$. At the other endpoint,
\begin{equation}\label{a716:eq:F-zero-endpoint}
 F(\xi,0)=\frac\pi{2\sqrt2}
  +2\int_1^\infty e^{-2\xi r}(-1+6r^2)g(r)\,dr\geq0.
\end{equation}
Indeed $g<0$ makes the integral term in \eqref{a716:eq:F-zero-endpoint} nondecreasing in $\xi$, and at
$\xi=0$ one can compute $\int_1^\infty(-1+6r^2)g(r)\,dr=-\pi/(4\sqrt2)$,
so that $F(0,0)=0$, in agreement with the vanishing of the residual
of every two-expert control at the full collision.
Convergence and the endpoint passage follow from
$g(r)=-1/(14\sqrt2r^4)+O(r^{-6})$ at infinity and
$g(r)=O((r-1)^{-1/2})$ at one. The one-dimensional minimum
principle for $\mathcal L$, first on compact intervals and then using
the two endpoint limits, yields $F\geq0$.
\end{proof}

\begin{lemma}\label{a716:lem:G-positive}
The functions $G_9$ and $G_{10}$ are nonnegative.
\end{lemma}
\begin{proof}
For $G_{10}$ the identity is exact:
\[
 -\tfrac12\partial_tG_{10}+\coth t\,G_{10}-\csch t\,F=0.
\]
Since $G_{10}(\xi,t)/\sinh^2t\to0$ at infinity, integration gives
\[
 G_{10}(\xi,t)=2\sinh^2t\int_t^\infty
                  \frac{F(\xi,s)}{\sinh^3s}\,ds\geq0.
\]
For $G_9$, put $H=-\tfrac12G_{9,t}+\coth tG_9-\csch tF$.
The polynomial terms in \eqref{a716:eq:F910}--\eqref{a716:eq:G9}
give the explicit reduction
\begin{equation}\label{a716:eq:H9}
 H=12\int_1^m e^{-2\xi r}r(r^2-1)g(r)\,dr
   +\frac{4b(t)-2\sech t}{\sqrt2}
   -4e^{-2\xi m}\csch^4t\,g(m).
\end{equation}
Its derivative has no integrals:
\[
 H_t=\frac{e^{-2\xi m}\csch^8t\sech^2t}{16\sqrt2}
       N(\xi,t),
\]
where
\begin{align*}
 N(\xi,t)={}&-5(19\xi+24t)\cosh t
 -30(\xi+4t)\cosh2t\cosh t
 +30\sinh t+65\sinh3t+3\sinh5t\\
 &+\xi\csch t\,[120t\cosh^2t+\cosh t\sinh5t]
 -32\sinh^7t\,e^{2\xi\coth t}.
\end{align*}
The following elementary derivative test proves $N\leq0$:
\begin{align*}
 N_{\xi\xi}&=-128\cosh^2t\sinh^5t\,e^{2\xi\coth t}<0,\\
 N(0,t)&=-2\cosh^2t\,V(t),\qquad
 N_\xi(0,t)=Q(t)/\sinh t,\\
 V(t)&=120t\cosh t-80\sinh t-15\sinh3t+\sinh5t,\\
 Q(t)&=\sinh t\,[\cosh t(-94-28\cosh2t+2\cosh4t)
                +\cosh t(120t\coth t-64\sinh^6t)].
\end{align*}
For $0\leq k\leq6$, $V^{(k)}(0)=Q^{(k)}(0)=0$, whereas
\begin{align*}
 V^{(7)}(t)&=5[152\cosh t-6561\cosh3t+15625\cosh5t
                         +24t\sinh t]>0,\\
 Q^{(7)}(t)&=64[339\cosh2t-3584\cosh4t+15309\cosh6t
                         -16384\cosh8t+120t\sinh2t]<0.
\end{align*}
The first sign follows from monotonicity of $\cosh$; for the second,
$120t\sinh2t\leq30\sinh4t\leq30\cosh4t$, since
$2t\leq\cosh2t$ and $2\sinh2t\cosh2t=\sinh4t$.
The remaining positive terms
are dominated by the displayed negative terms. Thus $V\geq0$ and
$Q\leq0$, and concavity in $\xi$ gives $N\leq0$.

Finally \eqref{a716:eq:H9} gives $H(\xi,t)\to0$ as $t\to\infty$.
Consequently $H_t\leq0$ implies $H\geq0$, and
\[
 G_9(\xi,t)=2\sinh^2t\int_t^\infty
 \left[\frac{F(\xi,s)}{\sinh^3s}
       +\frac{H(\xi,s)}{\sinh^2s}\right]ds\geq0.
\]
\end{proof}

\begin{proposition}[Case A10]\label{a716:prop:A9A10}
The residual in case A10 is nonnegative throughout $A$.
\end{proposition}
\begin{proof}
At fixed $\lambda$, its two face traces satisfy the homogeneous
$2$--$2$ hyperbolic system of Lemma~\ref{ab:lem:two-two}; these
identities follow directly from \eqref{ad:eq:recurrence}.
Their $z=0$ traces are $F/8$ and $G_{10}/8$ by
\eqref{a716:eq:910-residual-link}. Lemmas~\ref{a716:lem:F-positive}
and \ref{a716:lem:G-positive}, followed by that hyperbolic minimum
principle and Proposition~\ref{ab:boundary-reduction}, prove A10.
Case A9 will be propagated together with A7 and A8 below.
\end{proof}

\subsection{The averaged boundary certificates for A7 and A8}

Define the sum of the two edge residuals by
\begin{align}
 F_{78}(\xi,t)={}&\frac12\int_1^m e^{-2\xi r}g(r)
  [(-1+6r^2)\cosh t-r(3+2r^2)\sinh t]dr\notag\\
 &+\frac12e^{-2\xi m}\csch^3t\,g(m)
 +\frac{2a(t)\cosh t+4b(t)\sinh t+\tanh t}{4\sqrt2},
 \label{a716:eq:F78}\\
 G_{78}(\xi,t)={}&\frac14\int_1^m e^{-2\xi r}g(r)
 [-1+5r^2-4r^4+(-1+7r^2+4r^4)\cosh2t-10r^3\sinh2t]dr
 \notag\\
 &+\frac{2a(t)\cosh^2t-3\sinh t+2b(t)\sinh2t}{4\sqrt2}.
 \label{a716:eq:G78}
\end{align}
In this normalization,
\[
 F_{78}=(R_7+R_8)(\xi/\sqrt2,0,t/\sqrt2,0),\qquad
 G_{78}=(R_7+R_8)(\xi/\sqrt2,t/\sqrt2,0,t/\sqrt2).
\]

\begin{lemma}\label{a716:lem:78-edges}
Both $F_{78}$ and $G_{78}$ are nonnegative.
\end{lemma}
\begin{proof}
First, both functions are nondecreasing in $\xi$. For $F_{78}$
the polynomial
\[
 p(r)=(-1+6r^2)\cosh t-r(3+2r^2)\sinh t
\]
has values $p(-1)=5e^t$, $p(0)=-\cosh t$, $p(1)=5e^{-t}$,
and $p(m)=4\coth t\csch t$. Its leading coefficient is negative.
It therefore has two roots below one and its third root above $m$;
in particular $p>0$ on $[1,m]$. Differentiation in $\xi$ replaces
$g(r)<0$ by $-2rg(r)>0$, and also makes the endpoint term positive.

For $G_{78}$ set $D=m-r$. Its polynomial, including the factor
$1/4$, equals $Dq(D)/4$, where
\[
 q(D)=4(\cosh2t-1)D^3-6\sinh2tD^2+(\cosh2t-1)D+4\coth t.
\]
Its values at $0,m-1,m,m+1$ are, respectively,
\[
 4\coth t,\quad5(1-e^{-2t}),\quad-\sinh2t,\quad5(e^{2t}-1).
\]
The positive leading coefficient and these values show that its
three roots lie in $(-\infty,0)$, $(m-1,m)$ and $(m,m+1)$.
Thus $q>0$ on $[0,m-1]$, proving monotonicity of $G_{78}$ in $\xi$.

It remains to check $\xi=0$. Put
\[
 K(t)=-60t+45\sinh2t-9\sinh4t+\sinh6t.
\]
Then $K(0)=0$ and
$K'(t)=192\sinh^6t\geq0$. The first endpoint is
\begin{equation}\label{a716:eq:F78-zero}
 F_{78}(0,t)=\frac{\sinh t}{112\sqrt2}
 \left[112b(t)+8\log\coth(t/2)+\frac32\csch^7t\,K(t)\right]
 \geq0.
\end{equation}
For the second endpoint define, directly from \eqref{a716:eq:G78},
\[
 J(t)=1376256\sqrt2\,\frac{G_{78}(0,t)}{\sinh^2t}.
\]
Differentiation and substitution of \eqref{ab:kernel} give
\begin{align}
 J''(t)={}&\frac{786432e^{2t}(1+e^{2t})}{(e^{2t}-1)^{10}}
 \Bigl[(e^{2t}-1)^7
       \{14b(t)+5\log\coth(t/2)\}
       +36e^{7t}K(t)\Bigr]\geq0.
 \label{a716:eq:J78-second}
\end{align}
The defining integral gives $J(t),J'(t)\to0$ at infinity.
Consequently $J(t)=\int_t^\infty(s-t)J''(s)ds\geq0$.
Equations \eqref{a716:eq:F78-zero} and \eqref{a716:eq:J78-second}
are obtained from the displayed edge integrals themselves; their
normalizing factors are positive. Monotonicity in $\xi$ completes
the proof.
\end{proof}

\subsection{Coupled propagation of A7, A8 and A9}

For fixed $\lambda\geq0$, set
\begin{align*}
 f_j(Y,z)&=R_j((\lambda+z/2)/\sqrt2,0,Y/\sqrt2,z/\sqrt2),\\
 r_j(Y,z)&=R_j((\lambda+z/2)/\sqrt2,Y/\sqrt2,0,(Y+z)/\sqrt2),
 \qquad g_j(t)=r_j(t,0).
\end{align*}
All derivatives in the following argument hold $\lambda$ fixed.
The common equality direction $b$ gives the exact interpolation
\begin{equation}\label{a716:eq:companion-interpolation}
 r_j(Y,z)=\frac{\sinh Y}{\sinh\alpha}g_j(\alpha)
       +\frac{\sinh(z/2)}{\sinh\alpha}f_j(0,2\alpha),
 \qquad\alpha=Y+z/2.
\end{equation}

\begin{proposition}[Cases A7, A8 and A9]\label{a716:prop:A7A8}
The three residuals are nonnegative throughout $A$.
\end{proposition}
\begin{proof}
The residual identities obtained from \eqref{ad:eq:residualformula} are
\begin{gather}
f_7=f_8=:h,\qquad f_9=:j,\qquad g_8=g_9,\notag\\
(2h)_Y-2(2h)_z=2\coth Y(2h)-2\csch Y(r_7+r_8),\notag\\
j_Y-2j_z=2\coth Yj-2\csch Yr_9,\qquad
(r_7)_Y=2\coth Yr_7-2\csch Yh.
\label{a716:eq:coupled-system}
\end{gather}
The identities $f_7=f_8$ and $g_8=g_9$ follow by interchanging the
tied ranks $2,3$ and $3,4$, respectively, using the symmetric $C^2$
extension. Thus no kernel calculation is needed for these two identities.
The first equation for $2h$ is the cancellation of the two opposite
first-equation errors, not an assertion about the equations for each
control separately.
At the triple tie $q_2=q_3=0$, permutation invariance and the
continuous Hessian give
\begin{equation}\label{a716:eq:triple-tie}
h(0,z)=j(0,z)=:h_0(z).
\end{equation}
In fact the three controls choose one of the three tied ranks, and
the Hessian quadratic forms agree under those permutations.
Together with $g_8=g_9$, \eqref{a716:eq:companion-interpolation}
also gives $r_8=r_9$.
The edge inequalities already proved say
\begin{align*}
h(Y,0)&=\tfrac12F_{78}(\lambda,Y)\geq0,
&j(Y,0)&=\tfrac18F(\lambda,Y)\geq0,\\
g_7+g_8&=G_{78}(\lambda,\cdot)\geq0,
&g_9&=G_9(\lambda,\cdot)/8\geq0.
\end{align*}
Apply Lemma~\ref{bo:lem:paired} with
\[
F=h,\qquad R=r_8,\qquad S=r_7,\qquad
r=g_8,\qquad s=g_7,\qquad e=0.
\]
The two required diagonal signs are
$r+s=g_7+g_8\geq0$ and
$2r+s=(g_7+g_8)+g_9\geq0$.
For fixed $z>0$, the required local one-sided $C^1$ regularity of
the residual traces follows from \eqref{ad:eq:residualformula},
not merely from the global $C^2$ regularity of $v$. The product form
$K=-H_\alpha H_\delta H_\mu$ removes the apparent denominators
$\sinh Y$, and the endpoint $M=\coth(Y+z/2)$ stays in a compact
subset of $(1,\infty)$ near $Y=0$. The integrands and their
first derivatives in the face variables are locally bounded by
$C|g(r)|$, which is integrable at $r=1$ because
$g(r)=O((r-1)^{-1/2})$. The moving-endpoint terms are smooth since
$g$ is smooth for $r>1$. Differentiation under the integral and
Leibniz' rule therefore give continuous one-sided first derivatives;
the four-expert terms are also smooth here by
\eqref{ad:eq:four-expert}.
The normalized terminal condition follows from
\eqref{ad:eq:uniform-tail} and the nonnegative four-expert residual.
The lemma gives $h\geq0$, and hence $h_0\geq0$.
Interpolation gives $r_8=r_9\geq0$, so
Lemma~\ref{paper:lem:onecomponent} gives $j\geq0$.
Finally, the exact equation for $r_7$ and its zero normalized terminal
value give
\[
r_7(Y,z)=2\sinh^2Y\int_Y^\infty
\frac{h(s,z)}{\sinh^3s}\,ds\geq0.
\]
Proposition~\ref{ab:boundary-reduction} proves all three cases.
\end{proof}

\subsection{The inequalities in region B}
\label{bo:sec:B}
We apply the paired boundary principle to the Region B face residuals.

For the rest of this subsection, write $W=V_4$ for the four-expert
correction in \eqref{ad:eq:four-expert}, fix $X\ge0$, and define the two traces
\begin{align}
 F_j(y,z)&=R_{B_j}\left(\frac X{\sqrt2},\frac z{\sqrt2},
                               \frac y{\sqrt2},0\right),\notag\\
 G_j(y,z)&=R_{B_j}\left(\frac X{\sqrt2},\frac{y+z}{\sqrt2},
                                     0,\frac y{\sqrt2}\right),\notag\\
 h_j(t)&=R_{B_j}\left(\frac X{\sqrt2},\frac t{\sqrt2},
                                     0,\frac t{\sqrt2}\right).
 \label{bo:eq:traces}
\end{align}
The equality direction $c$ in \eqref{ab:eq-directions} yields
\begin{equation}\label{bo:eq:Ginterpolation}
 G_j(y,z)=\frac{\sinh(z/2)}{\sinh(y+z/2)}F_j(0,2y+z)
       +\frac{\sinh y}{\sinh(y+z/2)}h_j(y+z/2).
\end{equation}
The diagonal traces $h_j$, and the values $F_j(y,0)$, lie on the
$A$--$B$ interface.  The matching of the second derivatives in
\eqref{sm:eq:ABsecondf}--\eqref{sm:eq:ABsecondr} identifies them with
the corresponding region-$A$ residuals.  They are therefore nonnegative
by the region-$A$ inequalities proved above.

Here are explicit algebraic formulas for the identities used below.
They also specify all endpoint terms in the differentiations.  Set
\[
 t=y+z/2,\qquad \ell=z/2,\qquad M=\coth t,
 \qquad H_s=\cosh s-r\sinh s,\quad J_s=\sinh s-r\cosh s.
\]
The transformed directions for the six cases of interest are
\begin{equation}\label{bo:eq:directions}
\begin{array}{c|c@{\qquad}c|c}
 j&\delta_j&j&\delta_j\\ \hline
 5&(\frac12,-1,1,0)&8&(-\frac12,0,-1,1)\\
 6&(-\frac12,1,0,0)&9&(-\frac12,-1,1,0)\\
 7&(-\frac12,0,0,-1)&10&(-\frac32,1,0,0).
\end{array}
\end{equation}
For $\delta=(\delta_1,\ldots,\delta_4)$ write
\[
 \beta=\delta_1+\frac{\delta_2-\delta_4}{2},\qquad
 \tau=\delta_3+\frac{\delta_2+\delta_4}{2},\qquad
 \eta=\frac{\delta_4-\delta_2}{2},\qquad
 \mu=\frac{\delta_2+\delta_4}{2}.
\]
In the following polynomial take $m=\ell$ on the $F$ face and $m=t$
on the $G$ face:
\begin{align}
 P_\delta={}&(\tau^2+\eta^2+\mu^2-1+4r^2\beta^2)H_tH_\ell H_m
       +2\tau\eta J_tJ_\ell H_m
       +2\tau\mu J_tH_\ell J_m
       +2\eta\mu H_tJ_\ell J_m\notag\\
 &-4r\beta\bigl(\tau J_tH_\ell H_m
                +\eta H_tJ_\ell H_m+\mu H_tH_\ell J_m\bigr).
 \label{bo:eq:polynomial}
\end{align}
Thus the correction to the four-expert residual on the two faces is
\begin{align}
 F_j-F_j^{(4)}
   &=\int_1^M e^{-2Xr}g(r)P_{\delta_j}^{F}(r)\,dr
       +\tau_j^2\csch^5t\,\sinh^2y\,e^{-2XM}g(M),\notag\\
 G_j-G_j^{(4)}
   &=\int_1^M e^{-2Xr}g(r)P_{\delta_j}^{G}(r)\,dr.
 \label{bo:eq:residualkernel}
\end{align}
Here $F_j^{(4)},G_j^{(4)}$ mean the residual of $W(c,b,a)$ in the
last three components of $\delta_j$, with the coordinates scaled by
$\sqrt2$.  Formula \eqref{bo:eq:polynomial} follows by differentiating
$-e^{-2Xr}H_tH_\ell H_m$ twice; it is also an immediate form of
\eqref{ad:eq:residualformula}.

Define
\[
 E_{1,j}=(\partial_y-2\partial_z)F_j-2\coth y\,F_j
                          +2\csch y\,G_j,
 \qquad
 E_{2,j}=\partial_yG_j-2\coth y\,G_j+2\csch y\,F_j.
\]
Substitution in \eqref{bo:eq:residualkernel} gives the finite list of
identities
\begin{align}
 F_5&=F_6,& F_9&=F_{10},& G_8&=G_9,\label{bo:eq:traceidentities}\\
 E_{1,5}+E_{1,6}&=0,& E_{1,9}+E_{1,10}&=0,
       &E_{1,7}=E_{1,8}&=-\frac1{\sqrt2},\label{bo:eq:firsterrors}\\
 E_{2,6}&=0,&E_{2,7}&=0,&E_{2,10}&=0.
 \label{bo:eq:seconderrors}
\end{align}
For clarity, the integral coefficients in the nonzero first errors,
before adding the four-expert term, are
\begin{align}
 p_5&=-(r^2-1)(r\cosh t-\sinh t),&p_6&=-p_5,\notag\\
 p_9&=-(r^2-1)(4r^2\sinh t-3r\cosh t-\sinh t),&p_{10}&=-p_9.
 \label{bo:eq:errorpolynomials}
\end{align}
The corresponding four-expert terms for $j=5,9$ are both
\[
 \frac{\operatorname{arctanh}(e^{-t})\cosh t
              +\arctan(e^{-t})\sinh t-1}{\sqrt2},
\]
and those for $j=6,10$ are their negatives.  All endpoint coefficients
in these first errors vanish.  For $j=7,8$ the integral and endpoint
coefficients vanish separately.  In \eqref{bo:eq:seconderrors}, the
integral coefficient, the endpoint coefficient, and the four-expert
term each vanish.  These assertions follow by using
$H_s'=J_s$, $J_s'=H_s$, and $M'=-\csch^2t$ in
\eqref{bo:eq:polynomial}; hence the identities require no evaluation
of an integral.

The regularity and normalized terminal condition in Lemma~\ref{bo:lem:paired}
can be checked directly here. For fixed $z>0$, the upper endpoint
$M=\coth(y+z/2)$ stays in a compact subset of $(1,\infty)$ near
$y=0$. In \eqref{bo:eq:residualkernel}, the integrands and
their first derivatives in $y,z$ are locally bounded by $C|g(r)|$,
whose singularity at $r=1$ is integrable. The endpoint terms are
smooth because $g$ is smooth at $M>1$, and the explicit four-expert
terms are smooth here as well. Thus differentiation under the
integral and Leibniz' rule give the local one-sided $C^1$ regularity
of the residuals required by the lemma. When
$t\to\infty$, $M-1=O(e^{-2t})$ and
$g(r)=O((r-1)^{-1/2})$ near $1$.  Uniformly for $1\le r\le M$ and
$0\le s\le t$, $|H_s|+|J_s|=O(e^{-s})$.  Consequently both the
integral and endpoint corrections in the first line of
\eqref{bo:eq:residualkernel} are $O(e^{-2t})$, uniformly for $X\ge0$.
Since the four-expert residual is nonnegative, it follows that
$\liminf_{y+z\to\infty}F_j(y,z)\ge0$, which supplies the required
normalized edge limit.

\begin{proposition}\label{bo:prop:pairedcases}
The inequalities B5, B6, B7, B8, B9, and B10 hold throughout region $B$.
\end{proposition}

\begin{proof}
For B5 and B6, apply Lemma~\ref{bo:lem:paired} with
\[
 F=F_5=F_6,\qquad R=G_5,\qquad S=G_6,\qquad e=0.
\]
The averaged first equation follows from
\eqref{bo:eq:firsterrors}, and the companion equation is
$E_{2,6}=0$.  Formula \eqref{bo:eq:Ginterpolation} supplies the
two representations \eqref{bo:eq:companions}, with
$r=h_5\ge0$ and $s=h_6\ge0$.  This proves all three trace inequalities.
Exactly the same argument, with $(F,R,S)=(F_9,G_9,G_{10})$, proves
B9 and B10, using $E_{2,10}=0$.

For B7 use $F=F_7$, $R=S=G_7$ and $e=-1/\sqrt2$.
All hypotheses again follow from the displayed identities.
Finally, \eqref{bo:eq:traceidentities} and the already proved
case B9 give $G_8=G_9\geq0$. The Region A verification and
the A/B two-jet matching give $F_8(y,0)\geq0$. Moreover,
\eqref{bo:eq:firsterrors} yields
\[
\begin{aligned}
(\partial_y-2\partial_z)F_8
&=2\coth y\,F_8-2\operatorname{csch}y\,G_8-\frac1{\sqrt2}\\
&\leq2\coth y\,F_8-2\operatorname{csch}y\,G_8.
\end{aligned}
\]
Lemma~\ref{paper:lem:onecomponent}, with $m=2$, $F=F_8$
and $R=G_8$, therefore gives $F_8\geq0$.
Proposition~\ref{ab:boundary-reduction} extends these trace
inequalities to the whole of Region B.
\end{proof}

\subsubsection{The remaining direct comparisons}

It remains to identify the cases transferred directly from region $A$.
Let $(a,b,c)=\sqrt2(q_4,q_3,q_2)$, and set
$m=(a+c)/2$, $\ell=(c-a)/2$, and $t=b+m$.  Thus
$t\ge m\ge\ell\ge0$.  The difference between the two four-expert
terms in \eqref{ad:eq:Btransform} is
\begin{align}
 \Delta(a,b,c)
 &=W(c/\sqrt2,b/\sqrt2,a/\sqrt2)
        -W(a/\sqrt2,b/\sqrt2,c/\sqrt2)\notag\\
 &=\frac{\sinh\ell}{\sqrt2}
       \bigl(\cosh m-L(t)\sinh t\sinh m\bigr),
 \qquad L(t)=2\operatorname{arctanh}(e^{-t}).
 \label{bo:eq:Delta}
\end{align}
Since $\operatorname{arctanh}u\le u/(1-u^2)$ for $0<u<1$,
$L(t)\sinh t\le1$, and hence $\Delta\ge0$.
Direct differentiation, using $L'=-\csch t$, gives
\begin{align}
 (1-\partial_b^2)\Delta
   &=\frac{\sinh\ell\sinh b}{\sqrt2\sinh t}\ge0,\notag\\
 \bigl(1-(\partial_a-\partial_c)^2\bigr)\Delta&=0,
 \qquad
 \bigl(1-(-\partial_a+\partial_b-\partial_c)^2\bigr)\Delta=0.
 \label{bo:eq:Deltaidentities}
\end{align}
All limiting cases follow by continuity.

Under $T$ the gap direction $d$ becomes
\[
 \delta=\left(d_1+\frac{d_2-d_4}{2},d_4,d_3,d_2\right).
\]
The correction term is the same on both sides of the comparison, so
\eqref{bo:eq:Delta}--\eqref{bo:eq:Deltaidentities} yield
\begin{equation}\label{bo:eq:directcomparisons}
 \begin{array}{c|ccccc}
 \text{case in }B&2&11&16&12&13\\ \hline
 \text{case in }A\text{ at }Tq&2&11&16&12&14\\
 \text{difference of residuals}&\Delta&
 \dfrac{\sinh\ell\sinh b}{\sqrt2\sinh t}&
 \dfrac{\sinh\ell\sinh b}{\sqrt2\sinh t}&0&0.
 \end{array}
\end{equation}
The region-$A$ inequalities therefore imply B2, B11, B12, B13,
and B16.  B1 follows from $v_B=W+C(Tq)\ge0$, since the four-expert
correction is nonnegative and $C\ge0$ by
\eqref{ad:eq:correction} and $g<0$.  B14 and B15 are precisely
the two equalities in \eqref{ab:eq-directions}.  Together with
Proposition~\ref{bo:prop:pairedcases} and the comparisons for B3 and B4,
this proves every inequality in region $B$.

\subsection{Two comparisons in region B}
\label{b34:section}

The following comparisons reduce the third and fourth controls to the
seventh and eighth, respectively. The second comparison uses the sign of
an integrated kernel rather than the sign of the kernel itself.

\begin{proposition}\label{b34:comparison}
Throughout $B$,
\[
 R_3\geq R_7,\qquad R_4\geq R_8.
\]
\end{proposition}
\begin{proof}
It suffices to prove the two comparisons on the faces in
Proposition~\ref{ab:boundary-reduction}. On $q_4=0$, write
\[
 q=\frac1{\sqrt2}(\lambda,z,Y,0),\qquad
 A=e^{2Y},\quad B=e^z,\quad M=\coth(Y+z/2).
\]
Under \eqref{ad:eq:Btransform}, the four directions become
\[
 \widetilde d_3=(\tfrac12,0,0,-1),\quad
 \widetilde d_7=(-\tfrac12,0,0,-1),\quad
 \widetilde d_4=(\tfrac12,0,-1,1),\quad
 \widetilde d_8=(-\tfrac12,0,-1,1).
\]
For each pair the four-expert contributions agree. The endpoint terms
in \eqref{ad:eq:residualformula} also agree, since $K$ and $M$ are
independent of the first gap. Hence
\begin{equation}\label{b34:integraldifferences}
 R_3-R_7=\int_1^M e^{-2\lambda r}g(r)\Pi_3(r)\,dr,
 \qquad
 R_4-R_8=\int_1^M e^{-2\lambda r}g(r)\Pi_4(r)\,dr.
\end{equation}
Differentiating the factored kernel
$K=-H_\alpha H_\delta H_\mu$ in
\eqref{a26:kernel-representation} gives
\[
 \Pi_3=-\frac{r\eta Q_3}{4\sqrt A B^{3/2}},\qquad
 \Pi_4=\frac{r\eta Q_4}{4\sqrt A B^{3/2}},\qquad
 \eta=(r-1)B-r-1,
\]
where
\begin{align*}
 Q_3={}&AB\{B(2r^3-r^2-4r+3)-2r^3+r^2+2r-1\}\\
       &+B(-2r^3-r^2+2r+1)+2r^3+r^2-4r-3,\\
 Q_4={}&AB\{B(2r^3-3r^2+1)-2r^3+3r^2+2r-3\}\\
       &+B(-2r^3-3r^2+2r+3)+2r^3+3r^2-1.
\end{align*}
These are polynomials in $r$ and affine functions of $A$.

For $1<r\leq M$, the constraints are
\[
 1\leq B\leq\frac{r+1}{r-1},\qquad
 1\leq A\leq A_*:=\frac{r+1}{(r-1)B}.
\]
Put $h=(r-1)(B-1)/2\in[0,1]$. At the two endpoints of the
admissible interval for $A$,
\[
 Q_3\big|_{A=1}=4(h-1)\{(2r+3)h+r\}\leq0,
 \qquad
 Q_3\big|_{A=A_*}=4(h-1)(r+1)\leq0.
\]
Affineness implies $Q_3\leq0$ throughout the interval. Since
$\eta\leq0$, this proves $\Pi_3\leq0$ and therefore $R_3\geq R_7$.

For the second comparison define the polynomial primitive
\[
 J(r)=4\sqrt A B^{3/2}\int_1^r\Pi_4(s)\,ds.
\]
The parameters $A,B$ are held fixed in this integral. For each fixed
$r$, $J(r)$ is affine in $A$. Set $k=1-h\in[0,1]$. Direct polynomial
integration gives its endpoint values
\begin{align*}
 J(r)\big|_{A=1}
 &=-\frac{2(r-1)}3
       \{1+k+k^2+(2r+1)^2k^3\},\\
 J(r)\big|_{A=A_*}
 &=-\frac{2(r-1)}{15}
       \{5+(4r^2+12r+14)k+(16r^2+28r+11)k^2\}.
\end{align*}
Both are nonpositive, so $J(r)\leq0$ for every $1\leq r\leq M$.
By Lemma~\ref{ab:kernel-signs},
$w(r)=-e^{-2\lambda r}g(r)$ is positive and nonincreasing. Integration
by parts yields
\[
 \int_1^M e^{-2\lambda r}g(r)\Pi_4(r)\,dr
 =\frac{-w(M)J(M)+\int_1^Mw'(r)J(r)\,dr}
          {4\sqrt A B^{3/2}}\geq0.
\]
The lower endpoint term vanishes: $J(r)=O(r-1)$ and
$g(r)=O((r-1)^{-1/2})$ as $r\downarrow1$.

On the second face, write
$q=2^{-1/2}(\lambda,Y,0,Y)$ and $A=e^{2Y}$.
The same cancellation of the four-expert and endpoint terms holds,
and the difference kernels reduce to
\begin{align*}
 \Pi_3&=\frac r{2A}\{(r-1)A-r-1\}
             \{(r-1)(r+2)A-(r-2)(r+1)\}\leq0,\\
 \Pi_4&=-\frac{r^2}{2A}\{(r-1)A-r-1\}^2\leq0.
\end{align*}
Indeed, the first bracket in $\Pi_3$ is nonpositive for
$r\leq\coth Y$, while the second is at least $2r>0$ because $A\geq1$.
Thus both comparisons hold on this face as well. Continuity treats
the degenerate endpoints, and the equality-direction interpolation
propagates the comparisons to all of $B$.
\end{proof}

\subsection{Completion of the regional verification}

\begin{theorem}[Region A and B inequalities]\label{ab:thm:complete}
For every binary control $w\in\{0,1\}^5$,
\[
 v(q)-\frac12\bigl(d(w)\cdot\nabla_q\bigr)^2v(q)\geq0
 \qquad(q\in A\cup B).
\]
The controls $01001$ and $00101$ are equality cases in $A$;
the controls $00110$ and $00101$ are equality cases in $B$.
\end{theorem}
\begin{proof}
In $A$, cases A2--A6 were proved in
Propositions~\ref{a26:case2}--\ref{a26:case5}.
The constant control and the two equality controls were treated
directly. Proposition~\ref{ab:comb-positive} treats A12;
Propositions~\ref{a716:prop:A11A16} and~\ref{a716:prop:A14}
treat A11, A14 and A16. Case A10 follows from
Proposition~\ref{a716:prop:A9A10}, and the coupled argument in
Proposition~\ref{a716:prop:A7A8} treats A7--A9.

The preceding region-B reductions treat all cases except B3 and B4,
which follow from Proposition~\ref{b34:comparison} and the
nonnegativity of B7 and B8. Complementation covers the remaining
sixteen controls. Continuity extends the inequalities to all
regional and ordering boundaries.
\end{proof}

\section{Verification in Region C}\label{rc:section}
The purpose of this section is to verify every Hamilton--Jacobi--Bellman
inequality for the Region C branch of the explicit candidate in
Section~\ref{pc:sec:construction}.  The proof combines
three exact equality directions, a two-face reduction, a
kernel comparison, three-control averaging, and a final
one-variable monotonicity argument.  No numerical sign check enters the
proof.
\subsection{Region C, residuals, and the main result}
Let $q=(q_1,q_2,q_3,q_4)$ denote the ordered gaps and write
\[
\overline\Omega_C=
\left\{q_i\geq0:\ q_4\geq q_2,\quad
2q_1\leq q_4-q_2\right\}.
\]
For a control $\omega=(\omega_1,\ldots,\omega_5)\in\{0,1\}^5$, put
\[
d(\omega)=(\omega_2-\omega_1,\ldots,
\omega_5-\omega_4),\qquad
\mathcal R_\omega[v]
=v-\frac12\bigl(d(\omega)\mathbin\cdot\nabla\bigr)^2v.
\]
Complementary controls give the same residual, so it is enough to use the
sixteen representatives with $\omega_1=0$.
\begin{theorem}[Region C inequalities]\label{rc:thm:regionC}
Assume the Region A residual inequalities, including their continuous
traces on the A/C interface.  Then, for every control $\omega$,
\[
\mathcal R_\omega[v]\geq0\qquad\hbox{on }\overline\Omega_C.
\]
Moreover,
\[
\mathcal R_{00101}[v]
=\mathcal R_{01001}[v]
=\mathcal R_{01110}[v]=0
\qquad\hbox{throughout }\overline\Omega_C.
\]
\end{theorem}
Introduce the variables used by the Region C representation,
\[
x=\sqrt2q_1,\qquad Y=\sqrt2(q_2+q_3),\qquad
z=\sqrt2(q_4-q_2),
\]
and write
\begin{equation}\label{rc:eq:regionCvalue}
v=\lambda_p p(x,Y,z)+\lambda_r r_3(x,Y,z),\qquad
\lambda_p=\frac{\sinh(\sqrt2q_3)}{\sinh Y},\quad
\lambda_r=\frac{\sinh(\sqrt2q_2)}{\sinh Y}.
\end{equation}
The values at $Y=0$ are understood by continuity.  The functions
$p,r_3$ are precisely the two components in \eqref{pc:eq:ACpair}
on Region C; their explicit single-integral form is
\eqref{pc:eq:pairCsingle}.
\subsection{Equality modes and reduction to two faces}
\begin{proposition}[Three equality directions]\label{rc:prop:zeromodes}
The three residuals corresponding to
\[
00101,\qquad01001,\qquad01110
\]
vanish identically in Region C.
\end{proposition}
\begin{proof}
For $00101$, move along
\[
(q_1,q_2,q_3,q_4)\longmapsto
(q_1,q_2+t,q_3-t,q_4+t).
\]
The variables $(x,Y,z)$ in \eqref{rc:eq:regionCvalue} remain fixed.  Only
the two weights change, and both satisfy $\lambda''=2\lambda$.
Consequently $v''=2v$, which is exactly
$\mathcal R_{00101}[v]=0$.
For the other two controls, set
\[
\eta=\sqrt2q_2,\qquad w=z-2x,\qquad
\alpha=\frac w3,\qquad a=x+Y,\qquad \rho=a+\alpha,
\]
Regard $(x,Y,\eta,\alpha)$ as independent coordinates and introduce
\[
Q_B=\partial_x-\partial_Y-\partial_\eta,
\qquad Q_C=\partial_x-\partial_\alpha.
\]
In the $01001$ direction, $\alpha,a,$ and $\rho$ are fixed; in the
$01110$ direction, $Y,\eta$ and $\rho$ are fixed.
The local-term and kernel identities \eqref{pc:eq:kernelzeros}
therefore pass directly under the integral in
\eqref{pc:eq:valueCsingle}: its lower endpoint $\rho$ is fixed,
and $U(r)$ is independent of the moving coordinates.
They give $(Q_B^2-1)v=(Q_C^2-1)v=0$.
The corresponding gap derivatives are $\sqrt2Q_B$ and
$\sqrt2Q_C$, so both identities are precisely
$(d(\omega)\cdot\nabla)^2v=2v$.
\end{proof}
\begin{lemma}[Commutation and face reduction]\label{rc:lem:facereduction}
For every control $\omega$, it is enough to prove
\begin{equation}\label{rc:eq:twofaces}
\mathcal R_\omega[v](0,0,q_3,q_4)\geq0,
\qquad
\mathcal R_\omega[v](0,q_2,0,q_4)\geq0.
\end{equation}
The Region A trace on $q_4-q_2=2q_1$ is taken from the already verified
Region A formula; the equality of the two traces of the residual follows
from the A/C two-jet matching proved in the smoothness proof in Section~\ref{sm:section}.
\end{lemma}
\begin{proof}
The regional kernels are smooth in the interior, so the four
derivatives used here are legitimate.  Constant-coefficient
directional operators commute.  Hence each equality
direction $d_*$ in Proposition~\ref{rc:prop:zeromodes} gives
\[
\left(1-\frac12(d_*\cdot\nabla)^2\right)
\mathcal R_\omega[v]=0.
\]
First take $d_*=(1,0,0,-1)$ and follow
\[
q(t)=q+t(1,0,0,-1),\qquad
-q_1\leq t\leq\frac{q_4-q_2-2q_1}{3}.
\]
The endpoints lie on $q_1=0$ and on the Region A/C interface.  Along the
segment, $f(t)=\mathcal R_\omega[v](q(t))$ satisfies
$f-f''/2=0$; the one-dimensional minimum principle reduces the sign to
the two endpoints.
On $q_1=0$, use $d_*=(0,1,-1,1)$ and the segment
\[
(0,q_2+t,q_3-t,q_4+t),\qquad -q_2\leq t\leq q_3.
\]
Its endpoints are exactly the two faces in \eqref{rc:eq:twofaces}.  A second
application of the same minimum principle proves the lemma.
\end{proof}
\subsection{The controls 00000, 00011, and 00001}
The constant control is immediate:
\[
\mathcal R_{00000}[v]=v\geq0,
\]
where the sign follows from the nonnegative boundary regulator in
the fixed-policy representation \eqref{pc:eq:localtime}.
The five-expert boundary profile $U$ is defined by
\eqref{pc:eq:Udef}.  For use below, write $h=H$, with $H$ as in
\eqref{pc:eq:Uintegral}; thus
\begin{align}
U(t)&=\cosh t\int_t^\infty h(r)\dd r,\label{rc:eq:Uprofile}\\
h(t)&=
\frac{3\left(t/16-\sinh(2t)/64-\sinh(4t)/64
+\sinh(6t)/192\right)}
{\sqrt2\,\sinh^5t\cosh^2t}.\label{rc:eq:hprofile}
\end{align}
\begin{proposition}\label{rc:prop:00011}
$\mathcal R_{00011}[v]\geq0$ throughout Region C.
\end{proposition}
\begin{proof}
On the first face in \eqref{rc:eq:twofaces}, put
$y=\sqrt2q_3$, $z=\sqrt2q_4$, $\rho=y+z/3$.  The boundary equation gives
\[
\mathcal R_{00011}[v](0,0,q_3,q_4)
=\frac{\sinh^3y}{\sinh^3\rho}
\bigl(U(\rho)-U''(\rho)\bigr).
\]
Direct differentiation of \eqref{rc:eq:Uprofile}--\eqref{rc:eq:hprofile}
yields
\[
U(t)-U''(t)=\frac{\sqrt2}{128\sinh^6t}
\bigl(\sinh(6t)-9\sinh(4t)+45\sinh(2t)-60t\bigr).
\]
Writing the numerator as $K(t)$, the elementary identity
$K'(t)=192\sinh^6t$ and $K(0)=0$ give the positive representation
\[
U(t)-U''(t)
=\frac{3\sqrt2}{2\sinh^6t}\int_0^t\sinh^6r\,dr\geq0.
\]
Thus the first face residual is nonnegative.  On the second face the
third and fourth ranks collide; exchanging their control bits turns
$00011$ into the equality control $00101$.  Lemma~\ref{rc:lem:facereduction}
completes the proof.
\end{proof}
\begin{proposition}\label{rc:prop:00001}
$\mathcal R_{00001}[v]\geq0$ throughout Region C.
\end{proposition}
\begin{proof}
Put $P(y,z)=p(0,y,z)$, $R(y,z)=r_3(0,y,z)$ and
\[
F=P-P_{zz},\qquad G=R-R_{zz}.
\]
These are the two traces in \eqref{rc:eq:twofaces}.  The boundary pair
satisfies
\begin{align*}
(\partial_y-3\partial_z)P
&=3\coth y\,P-3\csch y\,R,\\
R_y&=-2\csch y\,P+2\coth y\,R.
\end{align*}
Since $1-\partial_z^2$ commutes with this system, the pair $(F,G)$
satisfies the same equations.  The trace $F(y,0)$ is the verified Region A interface.
The face estimates below imply that $F$ is bounded below and
$G/\sinh^2y\to0$ at infinity.  Integration of the second equation
gives $G=\mathcal TF$.  Lemma~\ref{ab:lem:two-two}, with
$m=3$ and $\theta=1$, therefore proves $F,G\geq0$, and
Lemma~\ref{rc:lem:facereduction} completes the proof.
\end{proof}
\subsection{Reduction of the remaining ten controls}
For every control, define the two face traces
\begin{align*}
F_\omega(y,z)&=\mathcal R_\omega[v]
\left(0,0,\frac y{\sqrt2},\frac z{\sqrt2}\right),\\
G_\omega(y,z)&=\mathcal R_\omega[v]
\left(0,\frac y{\sqrt2},0,\frac{y+z}{\sqrt2}\right)
\end{align*}
\paragraph{Regularity and infinity conditions for the face residuals.}
The following estimates supply the lower bounds and terminal
conditions used in the kernel comparison.  Each face residual is continuous on the closed quadrant and
smooth in its interior.
There are constants $C,c>0$ and an integer $m$, independent of the
control, such that
\begin{align}\label{rc:eq:facetails}
|F_\omega(y,z)|+|G_\omega(y,z)|
&\leq Ce^{-cz} &&(y,z\geq0),\notag\\
\left|F_\omega(y,z)
-\frac{1-(\omega_5-\omega_4)^2}{2\sqrt2}e^{-z}\right|
+|G_\omega(y,z)|
&\leq C(1+y)^m e^{-cy} &&(y\geq1,\ z\geq0).
\end{align}
For these estimates, expand the explicit formulas
\eqref{pc:eq:P}, \eqref{pc:eq:K}, and \eqref{pc:eq:valueCsingle}
after the substitution $r=\rho+t$ in each tail integral.
The profile satisfies
$U(r)=(2\sqrt2)^{-1}+O((1+r)e^{-2r})$, with the same
remainder bound after any fixed positive number of derivatives.
When $\rho\geq1$, factoring exponentials leaves denominators that
are powers of $1-e^{-2r}$, uniformly separated from zero.
The remaining integrands and their required derivatives have
integrable exponential bounds in $t$.
On the first face the leading value is
$(2\sqrt2)^{-1}e^{-z}$, independent of the lower three ranks;
its residual is the displayed leading term.  On the second face
the last gap is $(y+z)/\sqrt2$, and the leading value and residual
tend to zero.  The same estimates with $z\to\infty$ give the first
bound.  When $y$ stays bounded and $z\geq1$, cancel removable
hyperbolic quotients before differentiating: the integration
endpoints remain separated from zero.  On the remaining compact
set, continuity follows from Section~\ref{sm:section}.
Consequently every first trace has nonnegative lower limit at
infinity, and every companion has zero terminal value after
division by $\sinh^2y$.  These are the relevant infinity
conditions; the profile $U$ itself has a positive limit.
Set $D=\partial_y-3\partial_z$,
\begin{align*}
E_{1,\omega}&=DF_\omega-3\coth y\,F_\omega
+3\csch y\,G_\omega,\\
E_{2,\omega}&=(G_\omega)_y+2\csch y\,F_\omega
-2\coth y\,G_\omega.
\end{align*}
Permutation symmetry at the tied ranks, together with complementarity
of controls, gives the following trace identities.
\begin{align}
F_{00100}=F_{01000}=F_{01111},&\qquad
F_{00110}=F_{01010}=F_{01101},\label{rc:eq:Fclasses}\\
F_{00111}=F_{01011}=F_{01100},&\qquad
G_{00010}=G_{00100},\label{rc:eq:mixedclasses}\\
G_{01000}=G_{01111},&\qquad
G_{01010}=G_{01011}=G_{01100}=G_{01101},
\label{rc:eq:Gclasses}
\end{align}
Differentiation of the Region C formula gives the error cancellations
\begin{align}
E_{1,00100}&=-2E_{1,01000}=-2E_{1,01111},\notag\\
E_{1,00110}&=-2E_{1,01010}=-2E_{1,01101},\label{rc:eq:Ecancel}\\
E_{1,00111}&=-2E_{1,01011}=-2E_{1,01100}.\notag
\end{align}
In addition,
\[
E_{1,00010}=0,
\qquad
E_{2,00110}=E_{2,00111}=E_{2,01000}=E_{2,01111}=0.
\]
These are algebraic identities: after differentiation, the coefficients
of every derivative of $U$ and the remaining integral kernel vanish
separately.
\begin{proposition}[Three-control averaging]\label{rc:prop:averaging}
All ten remaining controls follow from the two inequalities
\begin{equation}\label{rc:eq:twotargets}
G_{00010}\geq0,\qquad G_{01010}\geq0.
\end{equation}
\end{proposition}
\begin{proof}
Assume the two target signs.  Denote the three common first traces in
\eqref{rc:eq:Fclasses}--\eqref{rc:eq:mixedclasses} by
\[
f_1=F_{00100},\qquad f_2=F_{00110},\qquad f_3=F_{00111},
\]
and let $\overline G_i$ be the average of the three companions in
the corresponding class.  The cancellations \eqref{rc:eq:Ecancel} give
\[
(\partial_y-3\partial_z)f_i
=3\coth y\,f_i-3\csch y\,\overline G_i.
\]
The exact second equations and their zero normalized terminal values
give the remaining companions as $\mathcal Tf_i$.  Consequently
\begin{align}\label{rc:eq:averagedcompanions}
\overline G_1&=\frac{G_{00010}+2\mathcal Tf_1}{3},\notag\\
\overline G_2&=\frac{2G_{01010}+\mathcal Tf_2}{3},\\
\overline G_3&=\frac{2G_{01010}+\mathcal Tf_3}{3}.\notag
\end{align}
Thus $\overline G_i\geq\theta_i\mathcal Tf_i$, where
$(\theta_1,\theta_2,\theta_3)=(2/3,1/3,1/3)$.
The Region A verification and two-jet matching give $f_i(y,0)\geq0$;
\eqref{rc:eq:facetails} supplies the lower bounds and convergence.
Lemma~\ref{ab:lem:two-two}, with $m=3$, yields $f_i\geq0$.
The displayed companion identities then give every companion sign
in the three classes.  Finally, $E_{1,00010}=0$ and
$G_{00010}\geq0$ give $F_{00010}\geq0$ by
Lemma~\ref{paper:lem:onecomponent}.  These are all ten controls.
\end{proof}
\subsection{The final two traces and a quadratic reduction}
Write
\[
G_0(y,z):=G_{00010}(y,z),\qquad
G_1(y,z):=G_{01010}(y,z),\qquad
\rho=y+\frac z3 .
\]
These are the two residual traces left by the kernel comparison and
three-control averaging arguments.  Introduce
\[
q=e^{-2\rho},\qquad s=e^{-2y},\qquad x=e^{-2t};
\qquad 0<q\leq s\leq1.
\]
To distinguish functions from their coordinates, define explicitly
\[
\widehat G_j(q,s)
=G_j\!\left(-\frac12\log s,\frac32\log\frac{s}{q}\right),
\qquad 0<q\leq s\leq1.
\]
Using the profile representation \eqref{rc:eq:Uprofile}, put
\[
S(x)=x^6-3x^5-3x^4+12x^3\log x+3x^2+3x-1.
\]
The exact change of variables gives
\[
h\left(-\frac12\log x\right)
=\frac{\sqrt2\sqrt{x}\,S(x)}{2(x-1)^5(x+1)^2}.
\]
The numerator in \eqref{rc:eq:hprofile} is a positive multiple of
\[
\frac t{16}-\frac{\sinh(2t)}{64}-\frac{\sinh(4t)}{64}
+\frac{\sinh(6t)}{192}
=\int_0^t\sinh^4r\cosh^2r\,dr>0\qquad(t>0).
\]
Thus $h>0$, and the negative denominator in its displayed
$x$-representation gives $S(x)<0$ for $0<x<1$.
Substituting the integral representation of $U$ into the exact residual
representations,
reversing the order of integration, and multiplying by the positive
factor $s$ gives
\begin{equation}\label{rc:eq:Hrepresentation}
H_j(q,s):=s\widehat G_j(q,s)
=C(q)(s-1)^2+\int_0^q F_j(q,s,x)\dd x,
\end{equation}
where
\[
C(q)=-\frac{2\sqrt2q^{3/2}S(q)}{9(q-1)^8(q+1)}>0,
\]
\[
F_0=-\frac{\sqrt2S(x)P_0(q,s,x)}
{9q^2\sqrt{x}(x-1)^9(x+1)^2},\qquad
F_1=-\frac{\sqrt2S(x)P_1(q,s,x)}
{36q^2\sqrt{x}(x-1)^9(x+1)^2},
\]
and
\begin{align*}
P_0={}&4q^4x+2q^4+3q^3x^2+2q^3x+q^3\\
&+s\bigl(-6q^3x^2-4q^3x-2q^3-2qx^4-4qx^3-6qx^2\bigr)\\
&+s^2\bigl(qx^4+2qx^3+3qx^2+2x^4+4x^3\bigr),\\[2mm]
P_1={}&10q^4x+5q^4+3q^3x^2+2q^3x+q^3
+9q^2(x^3+x^2+x)\\
&+s\bigl(3q^3x^2+2q^3x+q^3-36q^2(x^3+x^2+x)
+qx^4+2qx^3+3qx^2\bigr)\\
&+s^2\bigl(9q^2(x^3+x^2+x)+qx^4+2qx^3+3qx^2
+5x^4+10x^3\bigr).
\end{align*}
Thus $H_0$ and $H_1$ are exactly quadratic in $s$.  No approximation of
$U$ or of an integral has been made in \eqref{rc:eq:Hrepresentation}.
\subsection{Both quadratics are concave}
Set
\[
J_j(q):=-q^2\partial_s^2H_j(q,s).
\]
Since $H_j$ is quadratic, $J_j$ is independent of $s$.  Exact
moving-endpoint differentiation of \eqref{rc:eq:Hrepresentation} gives
\begin{align}
J_0''(q)&=-\frac{2\sqrt2q^{3/2}B_0(q)}{3(q-1)^{10}},\label{rc:eq:J0}\\
J_1'''(q)&=\frac{\sqrt2\sqrt q\,B_1(q)}{2(q-1)^{11}},\label{rc:eq:J1}
\end{align}
where
\begin{align*}
B_0={}&6q^5-60q^4\log q+125q^4-120q^3\log q-80q^3
-60q^2+10q-1,\\
B_1={}&22q^6-300q^5\log q+763q^5-1260q^4\log q+615q^4\\
&\hspace{12mm}-840q^3\log q-1180q^3-250q^2+33q-3.
\end{align*}
Their signs require only finite differentiation:
\[
B_0^{(5)}(q)=\frac{720(q-1)^2}{q^2}>0,
\qquad B_0^{(k)}(1)=0\quad(0\leq k\leq4),
\]
and
\[
B_1^{(6)}(q)=
\frac{1440(q-1)(11q^2-14q+7)}{q^3}<0,
\qquad B_1^{(k)}(1)=0\quad(0\leq k\leq5).
\]
Successive integration backwards from $q=1$ gives
\[
B_0(q)<0,\qquad B_1(q)<0\qquad(0<q<1).
\]
Consequently the right sides of \eqref{rc:eq:J0}--\eqref{rc:eq:J1} are
positive.  The fixed-interval substitution $x=qv^2$ gives
\[
J_0(q)=\frac{8\sqrt2}{105}q^{7/2}
+O\bigl(q^{9/2}|\log q|\bigr),\qquad
J_1(q)=\frac{4\sqrt2}{35}q^{7/2}
+O\bigl(q^{9/2}|\log q|\bigr).
\]
Integrating \eqref{rc:eq:J0} twice and \eqref{rc:eq:J1} three times from zero
therefore proves $J_0,J_1>0$.  Hence
\begin{equation}\label{rc:eq:concavity}
\boxed{\partial_s^2H_0(q,s)<0,\qquad
\partial_s^2H_1(q,s)<0.}
\end{equation}
\subsection{The edge s = 1}
Put $E_j(q)=\widehat G_j(q,1)$ and
\[
K_j(q)=q^2E_j(q),\qquad L_j(q)=q^{-2}E_j(q),
\qquad K_j=q^4L_j.
\]
Five differentiations remove the integral:
\begin{align}
K_0^{(5)}(q)&=-\frac{3\sqrt2M_0(q)}
{4q^{3/2}(q-1)^{11}},\label{rc:eq:K0}\\
K_1^{(5)}(q)&=-\frac{\sqrt2M_1(q)}
{4q^{3/2}(q-1)^{11}},\label{rc:eq:K1}
\end{align}
where
\begin{align*}
M_0={}&q^8-11q^7+104q^6-540q^5\log q+1021q^5
-1320q^4\log q\\
&\quad-540q^3\log q-1021q^3-104q^2+11q-1,\\
M_1={}&q^8-11q^7+184q^6-1500q^5\log q+3261q^5
-4200q^4\log q\\
&\quad-1500q^3\log q-3261q^3-184q^2+11q-1.
\end{align*}
The two terminal signs have especially short monotonicity proofs.  Define
\[
D_0=\frac{M_0}{60q^3(9q^2+22q+9)},\qquad
D_1=\frac{M_1}{300q^3(5q^2+14q+5)}.
\]
Then $D_0(1)=D_1(1)=0$ and
\begin{align*}
D_0'&=\frac{(q-1)^6(27q^4+52q^3+162q^2+52q+27)}
{60q^4(9q^2+22q+9)^2}>0,\\
D_1'&=\frac{(q-1)^6(5q^4+12q^3+158q^2+12q+5)}
{100q^4(5q^2+14q+5)^2}>0.
\end{align*}
Thus $M_0,M_1<0$, and \eqref{rc:eq:K0}--\eqref{rc:eq:K1} imply
$K_0^{(5)},K_1^{(5)}<0$.
This sign propagates all the way back without estimating an integral.
Let $\cD=q\partial_q$, set
\[
Q_{j,1}=\cD L_j,\qquad
Q_{j,k+1}=(\cD+k)Q_{j,k}\quad(1\leq k\leq4).
\]
Since $K_j=q^4L_j$,
\[
Q_{j,5}=qK_j^{(5)}<0,
\qquad (q^kQ_{j,k})'=q^{k-1}Q_{j,k+1}.
\]
At zero, the exact fixed-interval expansions are
\[
L_0(q)=\frac{8\sqrt2}{35}q^{-1/2}
+O(q^{1/2}|\log q|),\qquad
L_1(q)=\frac{8\sqrt2}{105}q^{-1/2}
+O(q^{1/2}|\log q|).
\]
Therefore $q^kQ_{j,k}\to0$ for $1\leq k\leq4$.
Successively integrating from zero gives
\[
Q_{j,4}<0,\ Q_{j,3}<0,\ Q_{j,2}<0,\ Q_{j,1}<0.
\]
Thus both $L_j$ are decreasing.  At the full-collision endpoint,
permutation symmetry of the Hessian and the $k=2$ equality give
\[
L_0(1)=\frac{U(0)}3=\frac{15\pi^2}{512\sqrt2}>0,
\qquad L_1(1)=0.
\]
Indeed, the Hessian of the translation-invariant five-expert value has
$D^2u=a(I-\mathbf1\mathbf1^{\mathsf T}/5)$ there; the $k=2$
equality gives $a=5U(0)/3$, and the $k=1$ residual is therefore $U(0)/3$.
(The same limits follow directly from the one-integral formulas.)  Hence
\begin{equation}\label{rc:eq:s1}
\boxed{H_0(q,1)=\widehat G_0(q,1)>0,\qquad
H_1(q,1)=\widehat G_1(q,1)>0\quad(0<q<1).}
\end{equation}
\subsection{The edge s = q from Region A}
When $s=q$, the change of variables gives $z=0$.  With
$y=-\tfrac12\log q$, this edge is the point
\[
q^{\rm gap}=\left(0,\frac y{\sqrt2},0,\frac y{\sqrt2}\right)
\]
on the A/C interface.  The full two-jet matching proved in
Section~\ref{sm:section} identifies the two residual traces there:
\[
\widehat G_0(q,q)=R_{A5}(q^{\rm gap}),\qquad
\widehat G_1(q,q)=R_{A12}(q^{\rm gap}).
\]
Propositions~\ref{a26:case5} and \ref{ab:comb-positive} therefore give
\begin{equation}\label{rc:eq:sq}
H_0(q,q)\geq0,\qquad H_1(q,q)\geq0\qquad(0<q\leq1).
\end{equation}
This uses only the Region A verification and the preceding regularity
theorem; neither depends on the Region C inequalities.
\subsection{Proof of the Region C theorem}
For fixed $q$, concavity \eqref{rc:eq:concavity} places $H_j(q,s)$ above
the chord joining its values at $s=q$ and $s=1$.  Equations
\eqref{rc:eq:s1} and \eqref{rc:eq:sq} therefore imply
\[
H_j(q,s)\geq0\qquad(0<q\leq s\leq1).
\]
Since $s>0$, this proves the two desired inequalities:
\begin{equation}\label{rc:eq:final}
\boxed{G_{00010}(y,z)\geq0,\qquad
G_{01010}(y,z)\geq0\qquad(y,z\geq0).}
\end{equation}
\begin{proof}[Proof of Theorem~\ref{rc:thm:regionC}]
Equation \eqref{rc:eq:final} and Proposition~\ref{rc:prop:averaging} close the
ten controls in the averaging classes.  Proposition~\ref{rc:prop:00001}
closes $00001$, Proposition~\ref{rc:prop:00011} closes $00011$, and
$00000$ is immediate.  Together with the three equality controls in
Proposition~\ref{rc:prop:zeromodes}, this accounts for all sixteen
representatives with first bit zero.  Complementation accounts for the
remaining controls and proves every inequality in Region C.
The reductions use the already verified Region A residual on the A/C
interface.  The smoothness proof in Section~\ref{sm:section} proves equality of the
full two-jets there, so the regional residuals have the same trace.
This also supplies the $s=q$ endpoint signs used in the final quadratic
reduction.
\end{proof}
\begin{corollary}[The exact COMB optimality set]\label{paper:cor:comb-locus}
For ordered gaps $q\in\mathbb R_+^4$, the control $01010$ attains
the Hamiltonian maximum if and only if $q_2=q_4=0$.
At every other point its residual is strictly positive.
\end{corollary}
\begin{proof}
Write $R=\mathcal R_{01010}[v]$. All regional inequalities have
already been proved, so $R\geq0$; optimality is equivalent to $R=0$.

In $A$, use the notation of Proposition~\ref{ab:comb-positive}.
If $z>0$, both integrals in \eqref{ab:comb-factorizations} are
strictly positive, since $g<0$ and their brackets are positive for
$1<\rho<M$. This includes $Y=0$: then $S=h>0$, $u=0$, and both
brackets reduce to $S^3(M-\rho)$, so their common continuous limit
is strictly positive. If $z=0$ and $Y>0$, the first face residual
is zero and the second is strictly positive. The interpolation
along $a$ in Proposition~\ref{ab:boundary-reduction} has second
weight $\sinh(\sqrt2q_2)/\sinh Y$; hence in this case $R=0$
exactly when $q_2=0$. The remaining case $Y=z=0$ has $R=0$ by
continuity along $q_2=q_4=0$. Thus in $A$ the zero set is precisely
$q_2=q_4=0$. In $B$, \eqref{bo:eq:directcomparisons} gives
$R_B(q)=R_A(Tq)$, where
$Tq=(q_1+(q_2-q_4)/2,q_4,q_3,q_2)$; the same characterization follows.

For $C$, put $f=F_{01010}$ and $G=G_{01010}$.
The strict $s=1$ sign in \eqref{rc:eq:s1}, the $A$ strictness just
proved on the $s=q$ edge, and concavity \eqref{rc:eq:concavity}
give $G(y,z)>0$ whenever $y+z>0$.
The trace identities and averaging equations
\eqref{rc:eq:Fclasses}--\eqref{rc:eq:Ecancel} give
\[
 (\partial_y-3\partial_z)f
 =3\coth y\,f-\operatorname{csch}y\,(2G+\mathcal Tf).
\]
Since $f\geq0$, an interior zero with $z>0$ would have zero
gradient, contradicting this identity and $G>0$.
On the singular edge, the trace definitions give
$f(0,z)=G(0,z)>0$ for $z>0$.

On $q_1=0$, set $Y=\sqrt2(q_2+q_3)$ and
$z=\sqrt2(q_4-q_2)$. For $Y>0$, the second characteristic in
Lemma~\ref{rc:lem:facereduction} gives exactly
\[
 R(0,q_2,q_3,q_4)
 =\frac{\sinh(\sqrt2q_3)}{\sinh Y}f(Y,z)
  +\frac{\sinh(\sqrt2q_2)}{\sinh Y}G(Y,z).
\]
If $q_4>0$, this is positive: use the second term when $q_2>0$,
and the first when $q_2=0$. For $Y=0$ the residual is instead
$f(0,\sqrt2q_4)>0$. If $q_1>0$ and $q_4-q_2>2q_1$, the
first characteristic of that lemma has a strictly positive
$q_1=0$ endpoint (its fourth gap is $q_4+q_1>0$), a nonnegative
$A/C$ endpoint, and positive interpolation weights. Thus $R>0$.
The $A/C$ endpoint itself has already been covered by $A$.
Since $q_4=0$ in $C$ forces $q_1=q_2=0$, these arguments cover
every point outside the asserted zero set; the zero set itself
is already contained in $A$.
\end{proof}
\begin{remark}
This is the same optimality set as in \cite[Theorem~2.2]{CD2026}:
in their decreasing rank convention it is
$x_1=x_2$ and $x_3=x_4$. Here it follows from strictness in the
regional verification above.
\end{remark}

\subsection{Completeness of the verification}
\label{paper:sec:coverage}
Table~\ref{paper:tab:coverage} records the proof for every complementary
pair.  Its controls have at most two ones, so there are
$1+5+\binom52=16$ distinct representatives.  For Region C, the complement
is used when the first bit is one.  In particular, case 10 is the
Region C equality $01110$.
\begin{table}[htbp]
\centering\small
\setlength{\tabcolsep}{5pt}
\begin{tabular}{ccccc}
\toprule
Case & Control & Region A & Region B & Region C\\
\midrule
1 & $00000$ & $v\geq0$ & $v\geq0$ & $v\geq0$\\
2 & $10000$ & \ref{a26:case2} & \eqref{bo:eq:directcomparisons} & \ref{rc:prop:averaging}\\
3 & $01000$ & \ref{a26:case3} & \ref{b34:comparison} & \ref{rc:prop:averaging}\\
4 & $00100$ & \ref{a26:case4} & \ref{b34:comparison} & \ref{rc:prop:averaging}\\
5 & $00010$ & \ref{a26:case5} & \ref{bo:prop:pairedcases} & \eqref{rc:eq:final}, \ref{rc:prop:averaging}\\
6 & $00001$ & \ref{a26:case6} & \ref{bo:prop:pairedcases} & \ref{rc:prop:00001}\\
7 & $11000$ & \ref{a716:prop:A7A8} & \ref{bo:prop:pairedcases} & \ref{rc:prop:averaging}\\
8 & $10100$ & \ref{a716:prop:A7A8} & \ref{bo:prop:pairedcases} & \ref{rc:prop:averaging}\\
9 & $10010$ & \ref{a716:prop:A7A8} & \ref{bo:prop:pairedcases} & \ref{rc:prop:averaging}\\
10 & $10001$ & \ref{a716:prop:A9A10} & \ref{bo:prop:pairedcases} & Equality\\
11 & $01100$ & \ref{a716:prop:A11A16} & \eqref{bo:eq:directcomparisons} & \ref{rc:prop:averaging}\\
12 & $01010$ & \ref{ab:comb-positive} & \eqref{bo:eq:directcomparisons} & \eqref{rc:eq:final}, \ref{rc:prop:averaging}\\
13 & $01001$ & Equality & \eqref{bo:eq:directcomparisons} & Equality\\
14 & $00110$ & \ref{a716:prop:A14} & Equality & \ref{rc:prop:averaging}\\
15 & $00101$ & Equality & Equality & Equality\\
16 & $00011$ & \ref{a716:prop:A11A16} & \eqref{bo:eq:directcomparisons} & \ref{rc:prop:00011}\\
\bottomrule
\end{tabular}
\caption{Coverage of all controls. Bare references denote propositions;
parenthesized references denote equations. Equality controls are established
in \eqref{ab:eq-directions} and Proposition~\ref{rc:prop:zeromodes}.}
\label{paper:tab:coverage}
\end{table}
The two regional interfaces are $q_4=q_2$ and $q_4-q_2=2q_1$;
their full two-jet matching was proved in Section~\ref{sm:section}.
That section also establishes continuous second derivatives on the closed
chamber and the four adjacent-rank smooth-fit identities.  Hence the
inequalities and the common equality $\mathcal R_{00101}=0$ pass to every
interface, every ordering wall, and every multiple collision by continuity.
Permutation symmetry then passes the PDE to every chamber.  There is no
additional control inequality at a tie, and there is no third
codimension-one B/C interface.  The conditions at infinity used by the
minimum principles are supplied by \eqref{ad:eq:uniform-tail} and
\eqref{rc:eq:facetails} and the corresponding Region B estimates.

\appendix
\section{Moving-endpoint differentiation}
\label{paper:app:mathematica}

There is a distinction between a function and its expression in new
coordinates.  With $q=e^{-2(y+z/3)}$ and $s=e^{-2y}$, the notation is
\[
 \widehat G_j(q,s)=G_j\left(-\tfrac12\log s,
                               \tfrac32\log(s/q)\right),
 \qquad H_j(q,s)=s\widehat G_j(q,s).
\]
In particular, differentiating in $q$ holds $s$ fixed, hence holds
$y$ fixed and changes $z$.  Differentiating $H_j(q,q)$ is a total
derivative along the edge, not a partial derivative of $H_j$.

The representation \eqref{rc:eq:Hrepresentation} can be manipulated
without asking a computer algebra system for an antiderivative.  For
\[
 Z(q)=a(q)+\int_0^q b(q,x)\,dx,
\]
store the pair $(a,b)$.  Leibniz' rule gives the exact update
\begin{equation}\label{paper:eq:pairderivative}
 (a,b)\longmapsto\left(a'(q)+b(q,q),\ \partial_q b(q,x)\right).
\end{equation}
The variables $q$ and $x$ must be treated as independent until the
endpoint substitution $x=q$ is made.  Repeated application automatically
differentiates all endpoint terms produced on earlier steps.

After multiplying the curvature by $-q^2$, the remaining kernels are
polynomials in $q$ of degree one for $j=0$ and degree two for $j=1$.
For $K_j=q^2\widehat G_j(q,1)$ they have degree four, and for
the edge $s=q$ no differentiation is needed: its signs are inherited
from Region A.  Two or three applications of
\eqref{paper:eq:pairderivative} for the curvature, and five for $K_j$,
remove the integral exactly.  This explains the derivative orders in
\eqref{rc:eq:J0}--\eqref{rc:eq:J1} and
\eqref{rc:eq:K0}--\eqref{rc:eq:K1}.

\section[Comparison with Calder and Drenska]
{Comparison with Calder and Drenska~\texorpdfstring{\cite{CD2026}}{}}
\label{paper:sec:CDcomparison}

Our contribution comprises both a probabilistic construction and an
independent analytic verification of the nonlinear PDE. The verification
is not a consequence of solving the linear equation for a proposed
control: it requires all the Hamiltonian inequalities and their passage
through interfaces and ordering collisions. The principal methodological
difference from~\cite{CD2026} lies in how these inequalities are reduced
and how the terminal signs are proved. Table~\ref{paper:tab:comparison}
summarizes the differences. References to sections and equations
of~\cite{CD2026} below refer to its first arXiv version.

The normalization and regional correspondence can be made explicit.
After reversal of rank order, their Regions I, II, and III are our
$B$, $A$, and $C$. Their distinguished control is our $00101$.
Moreover their scalar trace $e$ in~\cite[(2.8)]{CD2026} satisfies
\begin{equation}\label{paper:eq:profilecorrespondence}
 U(t)=\frac{e(t)}{\sqrt2},\qquad
 I(t)=\int_0^t\sinh^4r\cosh^2r\,dr
 =\frac t{16}-\frac{\sinh(2t)}{64}
             -\frac{\sinh(4t)}{64}+\frac{\sinh(6t)}{192}.
\end{equation}
Thus their $I$ is precisely the numerator bracket in
\eqref{pc:eq:Uintegral}. Both papers establish the same limiting value,
global $C^2$ regularity, and identification with the viscosity solution
of at most linear growth. The full-collision jet in
\eqref{sm:eq:fullcollisionjet} is likewise the jet
in~\cite[(2.18)]{CD2026}: writing $u_0=45\pi^2/(512\sqrt2)$,
\[
 \nabla u(0)=\tfrac15\mathbf1,\qquad
 D^2u(0)=\frac{5u_0}{3}
       \left(I-\tfrac15\mathbf1\mathbf1^{\mathsf T}\right).
\]
Calder and Drenska state the exact COMB optimality locus as a separate
theorem. The same locus, $q_2=q_4=0$, follows from strictness in our
regional proofs and is recorded in Corollary~\ref{paper:cor:comb-locus}.
Thus the lists of regional equality controls in
Sections~\ref{sec:ab-verification} and~\ref{rc:section} describe identities
throughout a region, not the absence of additional pointwise equalities
on its boundary. The distinction is the construction and
verification argument, not a different scalar profile or value.

\begin{table}[!htbp]
\centering
\caption{Differences in construction, smoothness, and verification.}
\label{paper:tab:comparison}
\vspace{0.5\baselineskip}
\begingroup
\small
\setlength{\tabcolsep}{5pt}
\renewcommand{\arraystretch}{1.15}
\begin{tabular}{@{}>{\raggedright\arraybackslash}p{.16\linewidth}
 >{\raggedright\arraybackslash}p{.39\linewidth}
 >{\raggedright\arraybackslash}p{.39\linewidth}@{}}
\toprule
Part of the proof & Present paper & Calder--Drenska~\cite{CD2026} \\
\midrule
Construction &
Reflected gap process and discounted local-time representation;
reflection determines the trace equations, solved by curvature
transport and Green inversion (Section~\ref{pc:sec:construction}). &
Analytic trace systems, truncated hyperbolic modes, and their
propagation (Section~3). \\
\addlinespace
Smoothness &
Cubic contact and derivative-level remainders are transported through
the propagation and companion operators before the ordering-wall
extension (Section~\ref{sm:section}). &
Direct differentiation of the final formulas, density bounds for
differentiated kernels, and explicit logarithmic endpoint estimates
(Section~4.1). \\
\addlinespace
Verification: reductions &
Equality-direction face reductions, positive companion-kernel
comparison without singular-boundary signs, and cancellation across
controls with identical traces
(Sections~\ref{paper:sec:comparison}--\ref{rc:section}). &
Multiaffine corner reductions and scalar Laplace-transform sign
criteria (Sections~4.2--4.4). \\
\addlinespace
Verification: terminal signs &
Global differential factorizations and specified endpoint data;
the final Region C traces are treated by concavity and explicit
edge-sign proofs, all given in the text. &
Rational bounds for logarithms on seven intervals, followed by
147 exact rational Bernstein certificates for 21 scalar signs
(Section~4.5). \\
\bottomrule
\end{tabular}
\endgroup
\end{table}

\paragraph{Construction.}
The reflected gap process gives the discounted local-time formula
\eqref{pc:eq:localtime}; its identification with the constructed
candidate is proved by It\^o's formula in~\eqref{pc:eq:itoverification}.
In particular, the coefficient $2/3$ in the collapsed system
\eqref{pc:eq:collapsed} is derived from the reflection mechanism.
This supplies a probabilistic interpretation of the boundary equations
and specifies their boundary data. The analytic solution is then
derived: the interface curvature obeys
\eqref{id:eq:curvaturetransport}, whose Green inverse produces the
density $g$, while the $m$--$2$ systems ($m=2,3$) reduce to
\eqref{pc:eq:curvaturetransport}. Boundedness and the original,
undifferentiated system fix the remaining homogeneous modes.
Calder and Drenska instead derive the trace systems analytically and
solve their Cauchy problems by the truncated modes
$\Phi_t(X)=\sinh(t-X)/\sinh t$ for $X\leq t$
\cite[Section~3.4]{CD2026}. The normalized propagation problems are
the same, and both constructions admit Green-kernel representations;
we do not claim different solution operators. The difference lies in
the reflected-process derivation of the equations and in solving
them through transported curvature. More explicitly, the kernel in
\eqref{pc:eq:compactGreen} is exactly
$\sinh r\,\Phi_r(t)\Phi_r(d)^m$ in
\cite[(3.30)]{CD2026}; their extension to bounded data likewise uses
integration by parts. This common analytic structure makes the
different nonlinear verification arguments especially relevant.

\paragraph{Smoothness.}
At the $A/C$ interface, our argument starts with matching normal
two-jets of the propagation data. Their difference has cubic contact,
which is preserved by the propagation and companion operators
(equations~\eqref{sm:eq:PminusB} and~\eqref{sm:eq:ACcubic}).
The $A/B$ calculation compares the full weighted regional values
in~\eqref{sm:eq:ABcubic}. Near collisions we also track
derivative-level remainders of the form
\[
 |\partial^\beta R|\leq
 C\rho^{3-|\beta|}(1+|\log\rho|)^m,\qquad |\beta|\leq2,
\]
as specified in~\eqref{sm:eq:E3class}. These estimates establish
regularity up to the closed chamber before the symmetric extension.
The emphasis differs from the direct differentiated-kernel and
finite-formula endpoint estimates in~\cite[Section~4.1]{CD2026}.

\paragraph{Verification: reducing the inequalities.}
There are initially $48$ regional inequalities: sixteen complementary
control representatives in each of $A,B,C$. Our first reduction uses
the linear equations associated with equality controls to propagate
residual signs from lower-dimensional faces. The residual traces then
satisfy coupled hyperbolic equations, whose structure can remove
boundary sign requirements and cancel errors across different controls.

In particular, Lemma~\ref{ab:lem:two-two} treats, for $m>0$,
\[
 F_y-mF_z\leq m\coth y\,F-m\csch y\,G,\qquad
 G\geq\theta\mathcal TF,\qquad 0\leq\theta\leq1.
\]
Under its lower-boundedness and integrability hypotheses,
$F(y,0)\geq0$ implies $F,G\geq0$. Positivity of the companion
kernel, its mass bound, and integration along characteristics give
a Gronwall inequality for the negative part of $F$ on $y\geq\varepsilon$.
Letting $\varepsilon\downarrow0$ uses continuity only: no sign or
boundary-derivative expansion at $y=0$ is needed. The same argument
handles the $2$--$2$ and $3$--$2$ systems and their averaged equations.
For the paired A/B cases, Lemma~\ref{bo:lem:paired} instead derives
and integrates a scalar equation for the common edge trace. These
arguments eliminate separate singular-boundary verification problems.

Cancellation between controls supplies a second reduction. With the
trace and error notation of Section~\ref{rc:section},
\[
 F_{00100}=F_{01000}=F_{01111},\qquad
 E_{1,00100}=-2E_{1,01000}=-2E_{1,01111}.
\]
The averaged first equation is consequently exact, without requiring
a sign for any of these errors separately. For this common trace $f$,
the averaged companion is
\[
 \overline G=\tfrac13G_{00010}+\tfrac23\mathcal Tf.
\]
Thus $G_{00010}\geq0$ gives the kernel lower bound with
$\theta=2/3$. The other two trace classes give $\theta=1/3$
once $G_{01010}\geq0$ is known.
Proposition~\ref{rc:prop:averaging} thereby reduces ten remaining
controls to two companion traces, $G_{00010}$ and $G_{01010}$.
This is a coupled comparison argument; it does not assume that the
averaged pair satisfies a complete homogeneous two-equation system.
The multiaffine corner reduction in~\cite[Section~4.2]{CD2026}
instead leads to eight one-variable generators. These counts concern
different stages: our two traces still depend on two variables and
require the following further reduction.

\paragraph{Verification: proving the terminal signs.}
Write $G_0=G_{00010}$ and $G_1=G_{01010}$. Using $\xi$ for the scalar
variable called $q$ in Section~\ref{rc:section}, put
\[
 \xi=e^{-2(y+z/3)},\qquad s=e^{-2y},\qquad
 H_j(\xi,s)=sG_j\!\left(-\tfrac12\log s,
                              \tfrac32\log(s/\xi)\right).
\]
The functions $H_j$ are exactly quadratic and concave in $s$.
Consequently, for $0<\xi<1$ and $\xi\leq s\leq1$,
\[
 H_j(\xi,s)\geq
 \frac{1-s}{1-\xi}H_j(\xi,\xi)
 +\frac{s-\xi}{1-\xi}H_j(\xi,1)\geq0.
\]
The edge $s=\xi$ inherits its signs from Region A, using the
full second-order matching already proved at the interface.
Concavity and the remaining edge $s=1$ are proved by finite
differentiation with specified endpoint data
(equations~\eqref{rc:eq:concavity} and~\eqref{rc:eq:s1}).
For instance, the function $B_0$ occurring in~\eqref{rc:eq:J0}
satisfies the global identity
\[
 B_0^{(5)}(\xi)=\frac{720(\xi-1)^2}{\xi^2}>0,\qquad
 B_0^{(k)}(1)=0\quad(0\leq k\leq4).
\]
Taylor's formula gives $B_0<0$ on $(0,1)$; the displayed derivative
identities and endpoint limits then give $J_0>0$ and
$\partial_s^2H_0<0$. The other concavity and edge signs are closed by
analogous explicit chains.

The relation to the terminal functions of~\cite{CD2026} can also be
made exact. Let $D_1,T_1,P_3,T_3$ denote their generators
in~\cite[(4.11)]{CD2026}. Reversing the rank order maps our companion
face \mbox{$q=(0,y/\sqrt2,0,(y+z)/\sqrt2)$} to their coordinates
\[
 (a_1,a_2,a_3,a_4)=(z/3,z/3,L,L),\qquad L=y+z/3.
\]
Their control table and normalization~\cite[(4.10)]{CD2026} give,
for $L>0$ (with the value at $L=0$ obtained by continuity),
\begin{align*}
 \bigl(G_0(L,0),G_1(L,0)\bigr)
 &=\frac{\cosh L}{\sqrt2}
       \bigl(T_1(L)+2D_1(L),D_1(L)\bigr),\\
 \bigl(G_0(0,3L),G_1(0,3L)\bigr)
 &=\frac{\cosh^3L}{\sqrt2}
       \bigl(T_3(L)+3P_3(L),P_3(L)\bigr).
\end{align*}
The difference in verification extends to the terminal one-dimensional
problems. Our hyperbolic comparison and control-averaging arguments
lead to two companion traces, whose quadratic structure reduces
positivity to concavity and edge inequalities. These are established
through global differential identities and endpoint data.
Calder--Drenska's corner reduction instead produces scalar generators
whose terminal signs are certified by intervalwise Bernstein-polynomial
bounds~\cite[Sections~4.2--4.5]{CD2026}. Thus the reduction techniques
determine different families of auxiliary inequalities and different
mechanisms for proving them. The edge identities above show that some
individual quantities nevertheless coincide up to positive factors
or combinations; the distinction concerns the terminal systems and
their proofs, not the inequivalence of every individual expression.

Both proofs use differential elimination and monotonicity. The
distinction at the final sign stage is that our global derivative
factorizations and endpoint arguments are supplied in the text;
no external archive of intervalwise polynomial coefficients is
needed to complete those arguments. Symbolic checks have helped
derive and audit the identities, so this is a claim about the
self-contained proofs, not about the absence of computational aids
during their development. In~\cite[Section~4.5]{CD2026},
the terminal rational-logarithmic expressions are instead enclosed
by rational functions on fixed intervals and certified using exact
Bernstein coefficients. Those certificates are exact, not numerical
sign samples. Symbolic computation can reproduce our algebraic and
differentiation identities, but it does not replace the comparison,
endpoint, or regularity arguments.

There are also differences in the presentation of the final formula.
Calder and Drenska give a finite Region III expression in their
single non-elementary profile $e$ and its derivatives
\cite[(2.16)--(2.17)]{CD2026}, and rederive the four-expert formula
within their analytic framework. Our Region C representation retains
single integrals, chosen to expose the equality directions and the
sign structure used in the verification. Thus the finite expression
in~\cite{CD2026} is not fully elementary, and the comparison above
does not assert that our integral representation is a different value
function.

\section*{Code availability}
An optional symbolic-check supplement is available in the
\href{https://github.com/ebayr/five-expert-verification}{\texttt{five-expert-verification}}
repository on GitHub.
It reproduces selected algebraic identities, residual reductions, and
differentiation formulas, with a guide matching the checks to the
manuscript. The proofs are self-contained and do not depend on running
the code; in particular, the supplement does not replace the analytic
comparison, endpoint, or regularity arguments.

\section*{Acknowledgments}
The authors thank Antonis Papapantoleon, Sotirios Sabanis, and
Alexandros Saplaouras for organizing the Stochastic Analysis session
of the 19th Panhellenic Conference on Mathematical Analysis, held
in Athens on 18--20 December 2025. A preliminary construction was
described in the conference abstract for Nikolaos Kolliopoulos's
five-expert talk~\cite{PCMA2025}; that abstract explicitly left the
nonlinear verification open.

Erhan Bayraktar and Nikolaos Kolliopoulos are supported in part by the
National Science Foundation (NSF) under grant DMS-2406232.
Ibrahim Ekren is supported in part by the National Science Foundation
(NSF) under grant DMS-2406240.

\bibliographystyle{plain}
\IfFileExists{references.bib}
  {\bibliography{references}}
  {\bibliography{ArXiv/references}}
\end{document}